\documentclass[]{interact}
\usepackage[utf8]{inputenc}
\usepackage{amsmath}
\usepackage{graphicx}
\usepackage{amsfonts}

\usepackage{xcolor}

\usepackage{adjustbox}
\usepackage{caption}
\usepackage{subcaption}

\usepackage{todonotes}
\usepackage{float}

\title{PDE constrained shape optimisation with first-order and Newton-type methods in the $W^{1,\infty}$ topology}
\author{Klaus Deckelnick\textsuperscript{a},
Philip J. Herbert\textsuperscript{b}\thanks{Contact P.J.~Herbert. Email: p.herbert@sussex.ac.uk},
and
Michael Hinze\textsuperscript{c}
\affil{
\textsuperscript{a}Otto-von-Guericke-University Magdeburg, Institut f\"ur Analysis und Numerik, Universit\"atsplatz 2, 39106 Magdeburg;
\textsuperscript{b}Department of Mathematics, University of Sussex, Brighton, BN1 9RF, United Kingdom;
\textsuperscript{c}Mathematical Institute, University of Koblenz, Universitätsstr. 1, D-56070 Koblenz
}
}

\date{\today}

\DeclareMathOperator{\iden}{I}
\newcommand{\A}{\mathcal{A}}
\newcommand{\detDashDash}{\mathbb{D}}
\newcommand{\ADashDash}{\mathbb{A}}

\newtheorem{theorem}{Theorem}
\newtheorem{algorithm}[theorem]{Algorithm}

\usepackage[utf8]{inputenc}
\usepackage{amsmath}
\usepackage{amsfonts}
\usepackage{amssymb}
\usepackage{verbatim}
\usepackage[]{graphicx}
\usepackage{xcolor}
\usepackage{todonotes}
\usepackage{hyperref}
\usepackage[style=alphabetic,giveninits,citestyle=alphabetic,doi=false,isbn=false,url=false,eprint=false,maxnames=50,sorting=nyt]{biblatex}
\hypersetup{
    citecolor=green,
    colorlinks=true,
    linkcolor=blue,
    filecolor=magenta,      
    urlcolor=red,
}

\newcommand{\dee}{{\rm d}}

\newcommand{\R}{\mathbb{R}}

\DeclareMathOperator{\Div}{div}

\DeclareMathOperator*{\argmin}{arg\,min}
\DeclareMathOperator{\Tr}{Tr}
\DeclareMathOperator{\id}{id}

\begin{document}
\maketitle
\begin{abstract}
    We present a general shape optimisation framework based on the method of mappings in the $W^{1,\infty}$ topology.
    We propose steepest descent and Newton-like minimisation algorithms for the numerical solution of the respective shape optimisation problems.
    Our work is built upon previous work of the authors in \hyperlink{cite.0@DecHerHin21}{Deckelnick, Herbert, and Hinze, ESAIM: COCV 28 (2022)}, where a $W^{1,\infty}$ framework for star-shaped domains is proposed.
    To illustrate our approach we present a selection of PDE constrained shape optimisation problems and compare our findings to results from so far classical Hilbert space methods and recent $p$-approximations.
\end{abstract}


\section{Introduction}
This work considers the use of vector-valued Lipschitz continuous functions in the task of optimising shapes.
The use of Lipschitz functions for shape optimisations is not necessarily new, having appeared in \cite{DecHerHin21} in a limited setting of  star-shaped domains.
While it may be possible to formulate many problems in a star-shaped setting, it lacks the natural formulation which is useful to practitioners, therefore making uptake in the community less likely.
Recent work which aimed to approximate this Lipschitz approach used a $p$-Laplacian.
Of course one is interested in the limit $p \to \infty$, however it can prove troublesome to utilise significantly large $p$ in computation.
Our approach will not require the use of degenerate elliptic operators, but remain appropriate for the implementation by practitioners and will appear in upcoming work.
This article will not concern itself with an extensive analysis of the convergence of the proposed algorithms, but will focus on the ease of implementation and examples.
A particular novelty which we consider is the use of second order data in this Lipschitz setting for which we demonstrate an effective algorithm for its implementation in this shape context.
An analysis for a model problem with a first order method is presented within \cite{DecHerHin23-NA}.
The second order method we present, or some variety of it would be very interesting to analyse.

We are interested in the numerical solutions of a number of shape optimisation problems
\begin{equation}
    \min \mathcal{J}(\Omega),\, \Omega \in \mathcal{S},
\end{equation}
where $\mathcal{S}$ is a collection of admissible domains.
This collection and the functional $\mathcal{J}$ will vary depending on the application.
To find, at least local, minima of this problem, we will consider a descent method.
By this, we mean that, given $\Omega\in \mathcal{S}$, we seek $V^* \colon \R^n \to \R^n$ such that $\mathcal{J}'(\Omega)[V^*]<0$ and set $\Omega_{\rm new} = ({\rm id} + \alpha V^*)(\Omega)$ for some suitably chosen $\alpha>0$.
To ensure that the map $ \id + \alpha V^*$ is a homeomorphism, it is sufficient to restrict to $\alpha$ to be small enough that $\alpha |DV^*|<1$ a.e., where, { $|\cdot|$ is pointwise the spectral (operator) norm.
While it is sufficient to take any sub-multiplicative norm, the spectral norm is convenient as it relates the Lipschitz and $W^{1,\infty}$ semi-norms.}

In the literature, it is common to seek $V^*$ in a Hilbert space $H$ which represents the negative gradient i.e.
\begin{equation}\label{eq:RieszGradient}
    (V^*,\eta)_H
    = (-\nabla_H \mathcal{J}(\Omega),\eta)_H
    := -\mathcal{J}'(\Omega)[\eta]
\end{equation}
for all $\eta \in H$, or equivalently, one might seek
\begin{equation}\label{eq:MinInHilbertSpace}
    V^* \in \argmin \left\{ \frac{1}{2}\|V\|_H^2 + \mathcal{J}'(\Omega)[V] : V \in H \right\}.
\end{equation}
A crucial issue in this context is the regularity of the solution $V^*$ to problem \eqref{eq:MinInHilbertSpace}, which strongly depends on the regularity of the current domain $\Omega$ as well as the choice of $H$.
It for example is not clear whether $\id+\alpha V^*$ defines a Lipschitz transformation for many frequent choices of $H$. In order to avoid these issues it was
suggested in \cite{DecHerHin21} to work directly in the space $W^{1,\infty}(\Omega,\mathbb R^d)$ and to consider the following problem
\begin{equation}\label{eq:generalMinimisationForDirection}
    V^* \in \argmin \left\{  \mathcal{J}'(\Omega)[V] :  V \in W^{1,\infty}(\Omega,\mathbb R^d), | D V | \leq 1 \mbox{ a.e. in } \Omega \right\}.
\end{equation}
In \cite{DecHerHin21} this idea was analysed and implemented for shape optimisation problems involving  star-shaped domains.
It is the purpose of this
paper to extend this approach to more general domains including the use of Newton--type methods.

\subsection*{Literature}
There continues to be rapid development in the mathematical and numerical analysis of shape optimisation.
The seminal works of {\sc Delfour and Zol\'esio} \cite{DelZol11}, {\sc Sokolowski and Zol\'esio}, and the recent overview by {\sc Allaire, Dapogny, and Jouve} \cite{ADJ21} and the comprehensive bibliographies within provide an extensive overview of the topic of shape optimisation.
The analysis, both mathematical and numerical, of shape optimisation problems has an extensive history, see e.g.~\cite{BFLS97,GM94,MS76,S80}.
While computational power has increased in recent years, it has encouraged further development of shape optimisation \cite{SSW15,SSW16,SW17}, particularly fluid dynamical applications \cite{BenCarGui15,FLUU17,GHHK15,GHKL18,RolCutPey16,HUU20,HSU21,KMHR19,SISG13}.
Many articles have considered different choices of inner products on Hilbert spaces.
A variety of choices are presented in \cite{HPS15}.
One particularly interesting example is \cite{IglSturWec18} which uses a penalty to weakly enforce the Cauchy-Riemann equations however it only appears applicable in two dimensions.
Another category of interesting choices are reproducing kernel Hilbert spaces \cite{EigStu18}, which for certain kernels, one may provide an explicit shape gradient.
While in a Hilbertian setting, the work \cite{OnySie21} considers non-smooth terms to ensure that a mesh does not become degenerate.
Some methods very much target having a particularly good mesh, a particular example is the so-called pre-shape calculus \cite{LufSch21-A,LufSch21-B}.

The utilisation of Banach spaces for shape optimisation is gathering attention.
To the best of our knowledge this was introduced in \cite{DecHerHin21} and considered $W^{1,\infty}$ perturbations for a star-shaped setting.
The direction of steepest descent in a star-shaped setting has been linked to optimal transport \cite{Her23}.
The star-shaped setting is frequently exploited \cite{EHS07,BouChSa20} to allow for a deeper analysis at the expense of generality.
A $p$-approximation to the infinity problem \eqref{eq:generalMinimisationForDirection}
is utilised in \cite{MulKulSie21} to optimise a fluid dynamic problem using a $p$-Laplace relaxation.
Such a fluid problem is frequently discussed in shape optimisation as it is known \cite{Pin74} that, for Stokes flow, the optimal shape should have a tip.
In \cite{MulKulSie21}, experiments demonstrate that the $p$-method will form a tip as opposed to in more classical Hilbertian methods.
The article \cite{MulPinRun22} develops upon \cite{MulKulSie21} to consider the computational scalability of a method closely related to a $p$-Laplace relaxation of \eqref{eq:generalMinimisationForDirection}.

Higher order methods are also of interest and will be considered in this work; second order methods have been considered in \cite{SchSch22}, utilising a so-called linear version of the second shape derivative.

\subsection*{Outline}
We begin in Section \ref{sec:Preliminaries} by outlining some necessary definitions and results for shape optimisation, mentioning the Lagrange approach from optimisation to write down first and second derivatives and providing examples which we will consider.
We then move onto a discussion about the discretisation of the infinity method in Section \ref{sec:Discretisation}.
Section \ref{sec:Applications} then provides numerical experiments of the previously described numerical experiments using the novel $W^{1,\infty}$ method we discuss.

\section{Shape derivatives and Lagrangian calculus}\label{sec:Preliminaries}

\subsection{Preliminaries}
In what follows we denote by $D \subset \mathbb R^d$ a convex hold-all domain. We
consider the shape optimisation problem
\begin{equation}
    \min \mathcal{J}(\Omega),\, \Omega \in \mathcal{S},
\end{equation}
where $\mathcal{S}$ is a collection of admissible domains such that $\Omega \Subset D$ for all
$\Omega \in \mathcal S$.
Here, we use the symbol $\Subset$ to denote compactly contained.
It is not difficult to see that $\id + V$ is a bi-Lipschitz transformation from
$D$ to $D$ provided that $V \in W_0^{1,\infty}(D,\mathbb R^d)$ with $\Vert DV \Vert_{L^\infty} <1$.
Assuming that $(\id + V)(\Omega) \in \mathcal S$ for such $V$ we say that $\mathcal J$ is shape
differentiable at $\Omega$ if  (cf.~\cite[Definition 4.1]{ADJ21}) 
$ V \mapsto \mathcal J\bigl( (\id+V)(\Omega) \bigr)$ is Fr\'{e}chet--differentiable at $V=0$ as a mapping
from $W_0^{1,\infty}(D,\mathbb R^d)$ into $\mathbb R$. An update step in a descent algorithm based on the Fr\'{e}chet derivative of
$\mathcal J$ will then seek a direction $V \in W_0^{1,\infty}(D,\mathbb R^d)$ such that
$\mathcal J'(\Omega)[V] <0$. In order to determine the direction of steepest descent  we are led to the problem of finding
$V^* \in W_0^{1,\infty}(D,\mathbb R^d)$ with
\begin{equation} \label{eq:descent1}
  V^* \in \argmin \left\{  \mathcal{J}'(\Omega)[V] :  V \in W_0^{1,\infty}(D,\mathbb R^d), | D V | \leq 1
    \mbox{ a.e. in } D \right\}.
\end{equation}
Let us note that we are including the hold-all domain within this minimisation problem for the determination of a direction of steepest descent, along with a Dirichlet boundary condition on the boundary of the hold-all domain.
Note that the fact that $D$ is convex ensures that $V$ is a Lipschitz continuous function with Lipschitz constant $1$.
Using the direction \eqref{eq:descent1} within a descent algorithm hence requires the solution of a highly nontrivial constrained
minimisation problem which can be {approximated} at the discrete level with the help of an alternating direction method of multipliers (ADMM).

The above approach will lead to a first order method.
{If $\mathcal{J}$ is twice shape differentiable,} it is worthwhile considering
a Newton--type approach as well. This can be achieved by replacing the minimisation problem \eqref{eq:descent1} by
\begin{equation}\label{eq:NewtonProblem}
    \min \left\{\frac{t}{2} \mathcal{J}''(\Omega)[V,V] + \mathcal{J}'(\Omega)[V] : V \in W_0^{1,\infty}(D;\R^d),\, |DV| \leq 1 \mbox{ a.e. in } D \right\},
\end{equation}
where $t>0$ may be interpreted as a damping factor.
Here, the evaluation of $\mathcal J''(\Omega)$ is by no means straightforward and we will use the Lagrangian calculus described in the next subsection
to carry out the calculations for the class of problems that we are interested in.
Our motivation for the formulation of \eqref{eq:NewtonProblem} is the following approximation: for $t \in (-1,1)$ and $V \in W_0^{1,\infty}(D;\R^d)$ with $|DV| \leq 1$ a.e. in $D$ it holds that
\begin{equation}\label{eq:NewtonApproximation}
    \mathcal{J} ( (\id + tV )(\Omega)) = \mathcal{J} (\Omega) + t \mathcal{J}'(\Omega)[V] + \frac{t^2}{2} \mathcal{J}''(\Omega)[V,V] + \mbox{higher order terms}.
\end{equation}
Throwing away these higher order terms and minimising over the admissible $V$, we recover \eqref{eq:NewtonProblem}.

Under appropriate conditions, one may show the existence of solutions of \eqref{eq:descent1} and \eqref{eq:NewtonProblem}, the approach being similar to results presented in \cite{PaWeFa18} and \cite{DecHerHin21}.
In Theorem \ref{thm:ContinuousMinimisers} we verify the existence of solutions with some reasonable assumptions.
In a discrete setting, the assumptions given for the continuous case are trivially satisfied.

\subsection{Lagrangian framework for PDE--constrained optimisation} \label{pdeconstrained}
For the ease of exposition, let us consider a shape functional of the form
\begin{equation} \label{defmathcalj}
\mathcal J(\Omega)= \int_{\Omega} j(\cdot,y_{\Omega}) \, \dee x,
\end{equation}
where $j\colon D \times \mathbb R \rightarrow \mathbb R$ is assumed to be sufficiently smooth and $y_{\Omega}$ denotes
the solution of a PDE posed in $\Omega$.
Let us note that one may also consider the gradient of the solution in the functional $\mathcal J$ which follows very similarly, but adds a layer of complexity to the already large formulae.
We shall adapt the Lagrangian framework developed in Sections 1.6.4 and 1.6.5 of
\cite{HinPinUlb08} in order to compute $\mathcal J'(\hat \Omega)$ and $\mathcal J''(\hat \Omega)$ at a {\it fixed} domain
$\hat \Omega \Subset D$. The main aspect of the Lagriangian method is to, in effect, decouple the state, $y_\Omega$, from, in the setting we consider, the shape, $\Omega$.
Denoting by $B$ a small open neighbourhood of $0$ in $W_0^{1,\infty}(D,\mathbb R^d)$ we associate with $V \in B$ the
perturbed domain $\Omega_V:= (\id+V)(\hat \Omega)$. By transforming to $\hat \Omega$ we find that, for the choice made in \eqref{defmathcalj}, $\mathcal J(\Omega_V)=
J(V,y_{\Omega_V} \circ (\id+V))$, where
\begin{equation}\label{eq:CostFunctionalJ}
    J(V,y) := \int_{\hat \Omega} j(\id + V, y) \det( \iden + DV) \, \dee \hat{x},
  \end{equation}
and we note that $|\det( \iden + D V)| = \det( \iden + D V)$ if $|DV|$ is  sufficiently small.
The derivatives of this choice of $J$ may be found in Appendix \ref{app:JDerivatives}.
In order to incorporate the PDE constraint we let $y=y_{\Omega_V} \circ (\id+V)$ and suppose that $y_{\Omega_V}$ solves the given
PDE problem on $\Omega_V$ if and only if $e(V,y)=0$ for some mapping
$e\colon B \times X \rightarrow Z$. Here,  $X$, $Z$ are suitable function spaces on $\hat \Omega$ and we assume in what follows
that  $e_y(0,\hat y)\colon X \rightarrow Z$ is invertible, where $\hat y=y_{\hat \Omega}$. After choosing $B$ smaller if necessary to apply an Implicit Function Theorem,
there exists for every $V \in B$ a unique $y=y(V) \in X$ such that $e(V,y(V))=0$, so that we may write
\begin{displaymath}
  \mathcal J(\Omega_V)=J(V,y(V))
\end{displaymath}
where, in the context of optimal control, the map $V \mapsto \mathcal{J}(\Omega_V)$ takes the role of a reduced cost functional.
In order to calculate the derivatives of $\mathcal{J}$ it is convenient to introduce the
Lagrange functional $L \colon X \times B \times Z^* \to \R$
\begin{equation}
    L(y,V,p) = {J(V,y)} + \langle p, e(V,y) \rangle,
\end{equation}
so that
\begin{displaymath}
    \mathcal J(\Omega_V) = L(y(V),V,p) \quad \mbox{ for any } p \in Z^*.
\end{displaymath}
If we denote by $p(V)$ the solution of $L_y(y(V),V,p(V))=0$, one immediately obtains that 
\begin{equation}\label{eq:FirstShapeDerivativeLagrangian}
    \mathcal{J}'(\hat \Omega)[V]
    = L_V(\hat y,0,\hat p)[V],
\end{equation}
where $\hat p=p(0)$. In a similar way one finds for the second derivative
\begin{eqnarray}
  \mathcal{J}''(\hat \Omega)[V,W]
  & = & 
        L_{yy}(\hat y,0, \hat p)[y'(0)[V], y'(0)[W]] + L_{yV}(\hat y,0, \hat p)[V,y'(0)[W]] \nonumber \\
  &  & +  L_{Vy}(\hat y,0,\hat p)[y'(0)[V],W]  + L_{VV}(\hat y,0,\hat p)[V,W], \nonumber
\end{eqnarray}
where $y'(0)[V]$ is the derivative of $W \mapsto y(W)$ at $W=0$ in direction $V \in W^{1,\infty}_0(D,\mathbb R^d)$,
which satisfies
\begin{equation} \label{yV}
   \langle p, e_y(0,\hat y) [y'(0)[V]] \rangle = -\langle p, e_V(0,\hat y)[V] \rangle, \quad \mbox{ for all } p \in Z^*.
\end{equation}
{For the implementation of the Newton-like method in \eqref{eq:NewtonProblem}, it is necessary to evaluate $\mathcal{J}''(\hat \Omega)[V,W]$ many times.}
In order to carry out the corresponding calculations as efficiently as possible we would like to avoid the frequent evaluation of
$y'(0)[W]$. To do, let us write
\begin{displaymath}
    \mathcal{J}''(\hat \Omega)[V,W] = \langle h_1, y'(0)[W] \rangle + \langle h_2, W \rangle,
\end{displaymath}
where
\begin{eqnarray*}
  \langle h_1,y\rangle & = & L_{yy}(\hat y,0,\hat p)[y'(0)[V], y] + L_{yV}(\hat y,0,\hat p)[V,y], \\
  \langle h_2,W \rangle & = & L_{Vy}(\hat y,0,\hat p)[y'(0)[V],W]  + L_{VV}(\hat y,0,\hat p)[V,W].
\end{eqnarray*}
We then first define $g \in Z^*$ as the solution of
\begin{equation} \label{eqadj}
  \langle g, e_y(0,\hat y) [y] \rangle = \langle h_1,y \rangle \quad \forall y \in X
\end{equation}
and then set
\begin{displaymath}
  \langle h_3,W \rangle = - \langle g,e_V(0,\hat y) [W] \rangle, \quad W \in W_0^{1,\infty}(D,\mathbb R^d).
\end{displaymath}
This gives
\begin{eqnarray*}
  \mathcal{J}''(\hat \Omega)[V,W]
  & = &  \langle h_1, y'(0)[W] \rangle + \langle h_2, W \rangle  =  \langle g, e_y(0,\hat y)[y'(0)[W]] \rangle + \langle h_2,W \rangle \\
  & = &  - \langle g,e_V(0,\hat y)[W] \rangle   + \langle h_2,W \rangle  =  \langle h_2+h_3,W \rangle.
\end{eqnarray*}
The evaluation of $\mathcal{J}''(\hat{\Omega})[V,\cdot]$ hence essentially requires the solutions of \eqref{yV} and of the adjoint problem \eqref{eqadj}.

\subsection{Existence of descent-like directions}
Now it is demonstrated how to construct the first and second derivatives, we demonstrate the well-posedness of the problem given in \eqref{eq:NewtonProblem}.
The result is very similar to those which appear in \cite{PaWeFa18} and \cite{DecHerHin21}
\begin{theorem}\label{thm:ContinuousMinimisers}
    Given $t\geq 0$, suppose that the maps
    \begin{align}
        V \mapsto& \mathcal J'(\Omega)[V] \mbox{ and}
        \\
        V\mapsto& t \mathcal J ''(\Omega)[V,V]
    \end{align}
    are weak-$*$ lower-semi-continuous, then there exists
    \begin{equation}
        V^* \in \argmin \left\{ \frac{t}{2}\mathcal{J}''(\hat \Omega)[V,V] +  \mathcal{J}'(\hat \Omega)[V] :  V \in W^{1,\infty}_0(D;\R^d),\, | D V | \leq 1 \right\}.
    \end{equation}
\end{theorem}
\begin{proof}
    Consider a sequence $\{V_n\}_{n=1}^\infty \subset \{ V \in W^{1,\infty}_0(D;\R^d) : |DV|\leq 1\}$ such that
    \begin{multline}
        \frac{t}{2}\mathcal{J}''(\hat \Omega)[V_n,V_n] + \mathcal{J}'(\hat \Omega)[V_n] \searrow
        \\
        \searrow \inf \left\{ \frac{t}{2}\mathcal{J}''(\hat \Omega)[V,V] +  \mathcal{J}'(\hat \Omega)[V] : V \in W^{1,\infty}_0(D;\R^d) ,\, |DV|\leq 1\right\}\mbox{ as }n \nearrow \infty.  
    \end{multline}
    It holds that $V_n$ is bounded, hence there is a weak-$*$ convergent subsequence, say $\{V_{n_k}\}_{k=1}^\infty$, and limit $V^*\in W^{1,\infty}_0 (D;\R^d)$ such that $|DV_h^*| \leq 1$ a.e. in $D$ and $V_{n_k}\overset{*}{\rightharpoonup} V^*$ in $W^{1,\infty}_0(D;\R^d)$.
    Furthermore, by the assumed weak-$*$ lower-semi-continuity of $t\mathcal{J}''(\hat \Omega)[V,V]$ and $\mathcal{J}'(\hat \Omega)[V]$, we have that $V^*$ is a minimiser.
\end{proof}
Let us briefly discuss the assumptions above.
The condition that $V \mapsto \mathcal{J}'(\Omega)[V]$ is weak-$*$ lower-semi-continuous is typically verified in practice by observing that there is $g_0 \in L^1(\Omega;\R^d)$ and $g_1 \in L^1(\Omega;\R^{d \times d})$ such that
\begin{equation}
    \mathcal J'(\Omega)[V] = \int_\Omega \left( g_0 \cdot V + g_1 : DV \right) \dee x,
\end{equation}
which provides weak-$*$ continuity.
The second condition is more tricky.
As we have seen in Section \ref{pdeconstrained} that the construction of the second derivative is non-trivial, as such the condition is less easy to explicitly check.
However it is the case that, if the mapping $V\mapsto\mathcal{J}''(\Omega)[V,V]$ non-negative, then it holds that it is weak-$*$ lower semi-continuous.
This certainly depends on the shape optimisation problem at hand.
However it is perhaps not too unreasonable to assume that, near a minimiser, the second derivative of the energy is non-negative. 

\section{Example shape optimisation problems}
Here, we now discuss a few example shape optimisation problems.
\subsection{Poisson problem}
As a first PDE constraint we here consider the Poisson problem.
We set $X= H^1_0(\hat{\Omega})$ and $Z = X^*$.
Since we are in a reflexive setting, we use the canonical injection and identify $Z^*$ with $X$.
By $y_{\Omega}$ we denote
the solution of
\begin{equation} \label{poisson}
- \Delta y_{\Omega} =F  \mbox{ in }  {\Omega}, \quad y_{\Omega} = 0 \mbox{ on } \partial {\Omega}
\end{equation}
for a given $F \in L^{2}(D)$.
In particular, we find that $y_{\Omega_V} \circ (\id +V)$ is a solution of $e(V,y)=0$ where $e\colon B \times H^1_0(\hat \Omega) \rightarrow H^{-1}(\hat \Omega)$ is given by
\begin{equation} \label{defe1}
\langle e(V,y), p \rangle = \int_{\hat \Omega} \left( A(V) \nabla y \cdot \nabla p -  F \circ (\id +V) p \right)   \det(\iden + DV)
 \, \dee \hat x, \quad p \in H^1_0(\hat \Omega),
\end{equation}
and  $A(V):= (\iden + DV)^{-1} (\iden + DV)^{-T}$.
{Derivatives of the map $e$ may be found in Appendix \ref{app:Poisson:Derivatives}.}
With the  Lagrange functional 
\begin{displaymath}
L(V,y,p) = \int_{\hat \Omega} j(\id + V,y) \det(\iden + DV) \,  \dee \hat x + \langle e(V,y), p \rangle 
\end{displaymath}
we deduce from \eqref{eq:FirstShapeDerivativeLagrangian} the well--known formula
\begin{eqnarray}
\mathcal J'(\hat \Omega)[V] & = & L_V(\hat y,0,\hat p)[V] \label{firstderivpoi}
\\
& = &  \int_{\hat \Omega} \left( j(\cdot,\hat y) \Div V + j_x(\cdot,\hat y)\cdot V   + \A[V]\nabla \hat y \cdot \nabla \hat p - \hat p \Div(F V) \right) d \hat x, \nonumber
\end{eqnarray}
where
\begin{equation}
    \A[V] := \iden \Div V - DV - DV^T,
\end{equation}
$\hat{y} \in H^1_0(\hat{\Omega})$ satisfies $e(0,\hat{y})=0$, and the adjoint $\hat p \in H^1_0(\hat \Omega)$ satisfies $L_y(\hat y,0,\hat p)=0$, i.e.
\begin{equation}
    \int_{\hat \Omega} \nabla\hat  p \cdot \nabla \eta \, \dee \hat x = -\int_{\hat \Omega} j_y(\cdot,\hat y) \eta \, \dee \hat x \quad \mbox{ for all } \eta \in H^1_0(\hat \Omega).
\end{equation}

\subsection{Bi-Laplace-type equation}
Let us next consider the minimsation of $\mathcal J$ as in \eqref{defmathcalj} subject to the linear PDE of fourth order
\begin{equation} \label{bilaplace}
 \Delta^2 y_\Omega =F  \mbox{ in }  \Omega, \quad y_\Omega = \Delta y_{\Omega}=0  \mbox{ on } \partial \Omega.
\end{equation}
If the boundary of $\Omega$ is sufficiently regular the above problem can be split into two second order Poisson problems by introducing $-\Delta y_\Omega$ as an additional variable.
Let us note that this splitting is analytically useful to ensure that the shape derivative exists in the sense of \cite[Definition 4.1]{ADJ21}, due to the fourth order nature of the problem.
Let us comment that this need not be necessary since the shape differentiability, particularly boundedness in Lipschitz functions, with a fourth order constraint was demonstrated in \cite{EllHer21} in a surface context, while \cite{Las17} shows this for a fourth order eigenvalue problem.

On the fixed domain, we set $X = \left(H^1_0(\hat{\Omega})\right)^2$ and $Z = X^*$.
Again we will use the canonical injection to identify $Z^*$ with $X$.
Posing the split formulation of \eqref{bilaplace} on $\Omega_V$ and transforming it back onto $\hat \Omega$ in the same way as above we write the map $e$ which represents the PDE constraint as,
\begin{eqnarray}
   \langle e(V,y ),p \rangle & = &    \int_{\hat \Omega} \bigl( A(V)  \nabla y_1 \cdot \nabla p_2 -  y_2 p_2  \bigr) \det(\iden +DV) \, \dee \hat x  \nonumber \\
      &  &+ \int_{\hat \Omega} \bigl(  A(V) \nabla y_2 \cdot \nabla p_1  - F \circ ( \id + V)  p_1  \bigr) \det(\iden +DV)  \, \dee \hat x, \label{defe2}
\end{eqnarray}
for all $y = (y_1,y_2), p =(p_1,p_2) \in \left( H_0^1(\hat{\Omega}) \right)^2$.
{Derivatives of the map $e$ may be found in Appendix \ref{app:coupledPoisson:Derivatives}.}
Similar to \eqref{firstderivpoi} we obtain for the shape derivative 
\begin{eqnarray}
    \mathcal{J}'(\hat \Omega)[V] & = &    \int_{\hat \Omega} \left(  j(\cdot, \hat y_1)\Div V + j_x(\cdot,\hat y_1)\cdot V +
    \A[V] \nabla \hat y_1 \cdot \nabla \hat p_2  -  \hat y_2  \hat p_2 \Div V  \right) \, \dee \hat x  \nonumber  \\
    & & + \int_{\hat \Omega} \left(  \A[V] \nabla \hat y_2 \cdot \nabla \hat p_1 - \Div(FV) \hat p_1 \right) \, \dee \hat x, \label{firstderivbil}
\end{eqnarray}
where $\hat y = (\hat y_1, \hat y_2) \in \left( H^1_0(\hat{\Omega}) \right)^2$ satisfies $e(0,\hat{y}) =0$ and the adjoint $\hat p =(\hat p_1,\hat p_2) \in ( H^1_0(\hat \Omega))^2 $ satisfies
\begin{align}
    \int_{\hat \Omega}  \nabla \hat p_2 \cdot \nabla \eta_1 \, \dee \hat x &= -\int_{\hat \Omega} j_y(\cdot,\hat y_1) \eta_1 \, \dee \hat x \quad \forall \eta_1 \in H^1_0(\hat \Omega),
    \\
    \int_{\hat \Omega}  \nabla \hat p_1 \cdot \nabla \eta_2 \, \dee \hat x  &= \int_{\hat \Omega} \hat  p_2 \eta_2  \, \dee \hat x  \quad \forall \eta_2 \in H^1_0(\hat \Omega).
\end{align}

\subsection{Optimisation of the first eigenvalue for the Laplacian}
\label{sec:application:evalue}
Our aim is to apply the above Lagrangian framework also for the optimisation of the first Dirichlet eigenvalue of the Laplacian, i.e.
\begin{equation}
    \mathcal{J}(\Omega) = \lambda_1(\Omega),
\end{equation}
where $\lambda_1(\Omega)$ is defined by
\begin{equation}\label{eq:eigenValue}
    \lambda_1 (\Omega) := \inf \left\{\int_{\Omega} |\nabla z|^2 \, \dee x : z \in H^1_0(\Omega),\, \int_{\Omega} z^2 \, \dee x= 1 \right\}.
\end{equation}
With the notation of Section \ref{pdeconstrained} we again fix a $\hat{\Omega} \Subset D$ which we now assume to be connected and set
$\Omega_V = (\id + V)(\hat{\Omega})$.  We transform the eigenvalue relation 
\begin{displaymath}
-\Delta z_{\Omega_V} = \lambda z_{\Omega_V} \mbox{ in } \Omega_V, \quad z_{\Omega_V} =0 \mbox{ on } \partial \Omega_V
\end{displaymath}
together with the condition $\int_{\Omega_V} z_{\Omega_V}^2 \dee x =1$ onto $\hat \Omega$ and write it in the form $e(V,y)=0$,
where $e\colon B \times X \rightarrow Z$, with $X= H^1_0(\hat \Omega)\times \R, Z = X^*$, and
\begin{eqnarray}
    \langle e(V,y), p \rangle & = & 
    \int_{\hat \Omega}  \left( A(V) \nabla z \cdot \nabla q -  \lambda z q \right) \, \det(\iden + DV) \, \dee \hat x \nonumber    \\
                              && + \mu \left(1-\int_{\hat \Omega}  z^2 \, \det (\iden + DV) \, \dee \hat x \right), 
                                  \label{defe3}
\end{eqnarray}
for $y = (z,\lambda), p=(q,\mu) \in H^1_0(\hat \Omega) \times \mathbb R$. Derivatives of the map $e$ may be found
in Appendix \ref{app:eValue:Derivatives}. 
Let $\hat z\in H^1_0(\hat \Omega)$ be an eigenfunction to the first Dirichlet
eigenvalue $\hat \lambda$ with $\int_{\hat \Omega} \hat z^2\, \dee x =1$. Then we have for all $p=(q,\mu) \in H^1_0(\hat \Omega)\times \mathbb R$
\begin{align}
    \langle e_V(0,\hat y)[V],p \rangle 
    =&
    \int_{\hat \Omega} \left( \A[V] \nabla \hat z \cdot \nabla q -  \hat \lambda \Div V \hat z q - \mu \Div V \hat z^2 \right) \, \dee \hat x,
    \\
    \langle e_y (0, \hat y)[(\eta,\tilde \eta)], p \rangle 
    =&
    \int_{\hat \Omega} \left( \nabla \eta \cdot \nabla q- \hat \lambda \eta q - \tilde \eta \hat z q - 2 \mu \hat z \eta \right) \, \dee \hat x,
\end{align}
where $\hat y =(\hat z,\hat \lambda)$.
Since $\hat{\lambda}$ is simple, cf.\ $\hat \Omega$ is connected and $\hat{\lambda}$ is the first Dirichlet eigenvalue \cite[Theorem 8.38]{GilTru77}, it can be shown that
$e_y(0,\hat y)\colon H^1_0(\hat \Omega) \times \mathbb R \rightarrow H^{-1}(\hat \Omega) \times \mathbb R$ is invertible. Thus we
can write for $V \in B$
\begin{displaymath}
  \mathcal J(\Omega_V)=J(V,y(V)), \quad \mbox{ where } J(V,(z,\lambda))=\lambda. 
\end{displaymath}
The Lagrange functional is given by $L(y,V,p)= \lambda+ \langle e(V,y),p \rangle$ so that we derive with the help of \eqref{firstderivpoi}
\begin{equation} \label{firstdereig}
  \mathcal J'(\hat \Omega)[V] = L_V(\hat y,0,\hat p)[V] = \langle e_V(0,\hat y)V,\hat p \rangle.
\end{equation}
The adjoint $\hat p=(\hat q ,\hat \mu)$ is given by the relation $L_y(\hat y,0,\hat p)=0$, i.e.
\begin{displaymath}
  \tilde{\eta} + \int_{\hat \Omega} \left( \nabla \eta \cdot \nabla \hat q - \hat \lambda \eta \hat q - \tilde \eta \hat z \hat q -
    2 \hat \mu \hat z \eta \right) \dee \hat x =0 \quad \forall (\eta,\tilde \eta) \in H^1_0(\hat \Omega) \times \mathbb R.
\end{displaymath}
We infer that $\int_{\hat \Omega} \hat z \hat q \, d \hat x =1$ as well as
\begin{displaymath}
  \int_{\hat \Omega} \left( \nabla \eta \cdot \nabla \hat q - \hat \lambda \eta \hat q \right) \dee \hat x = 2 \hat \mu \int_{\hat
    \Omega} \hat z \eta \, \dee \hat x \quad \forall \eta \in H^1_0(\hat \Omega).
\end{displaymath}
Choosing $\eta=\hat z$ we deduce that $\hat \mu=0$, so that $\hat q$ is an eigenfunction for the eigenvalue $\hat \lambda$.
Since $\hat \lambda$ is simple and $\int_{\hat \Omega} \hat z \hat q\, \dee \hat x =1$ we infer that $\hat q=\hat z$ and hence by
\eqref{firstdereig} that (cf.~\cite{Hen06,HenPie18}) 
\begin{equation}\label{eq:eigenValueDerivative}
    \lambda_1'(\hat \Omega)[V]
    =
    \int_{\hat \Omega} \left( \A[V]\nabla \hat z \cdot \nabla \hat z - \hat \lambda \Div V \hat z^2 \right) \dee \hat x.
\end{equation}
It is known that the first eigenvalue scales with volume, as such we are interested in fixing the volume of $\Omega$.
While it is known that the minimiser of the first eigenvalue is a ball, the methodology is interesting and can
be applied to more complicated eigenvalue problems.

\section{Discretisation}\label{sec:Discretisation}
Our aim is to formulate a descent algorithm which produces in each step a polygonal domain and which replaces
a possible PDE constraint with a corresponding finite element approximation. To begin, let $\mathcal{T}_h^0$ be
a triangulation of the hold--all domain $D$.
We look for discrete directions of descent in the  finite element spaces
\begin{equation}
  \mathcal{V}_h^n := \left\{ V_h \in C^0(\overline{D};\R^d) : V_h|_T \in P^1(T;\R^d),\, \forall T \in \mathcal{T}_h^n, \,
V_h = 0 \mbox{ on } \partial D  \right\},
\end{equation}
where $P^1(T;\R^d)$ denotes polynomials of degree at most one on $T$ with values in $\R^d$ and $\mathcal{T}_h^n$ is to be determined for $n\geq 1$.
With a polygonal initial domain, $\Omega_0$ which is a union of the triangles in the triangulation $\mathcal{T}_h^0$, we will set $\Omega_{n+1} = (\id +t_n V_n)(\Omega_n)$ for $n\geq 1$, where $t_n \in (0,1)$ is a step size and we will shortly explain how to choose $V_n \in \mathcal{V}_h^n$.
As well as updating the domain, the triangulation will also be updated, $\mathcal{T}^{n+1}_h = \{ (\id + t_n V_n)(T) : T \in \mathcal{T}_h^n \}$.
By the choice of $\mathcal{V}_h^n$ and the fact that $V_n$ will satisfy $|DV_n| \leq 1$, it holds that the updated mesh will be admissible since $t_n \in (0,1)$.
\subsection{Choice of descent direction}
Let $n\geq 0$ be fixed and let us denote the polygonal domain $\hat \Omega = \Omega_n \Subset D$ which is a union of triangles in $\mathcal T_h^n$.
For simplicity we will henceforth neglect the dependence on $n$.
Given $t \geq 0$, we aim to find $V_h^* \in \mathcal{V}_h$ such that
\begin{eqnarray*}
  V_h^* & \in &  \argmin \left\{ \frac{t}{2}\mathcal{J}''(\hat \Omega)_h[V_h,V_h] \mathcal{J}'(\hat \Omega)_h[V_h] :  V_h \in  \mathcal{V}_h, | D V_h | \leq 1
                \mbox{ a.e. in } D \right\} \\
      & = &  \argmin \left\{ \int_D \phi(DV_h) \dee x + \frac{t}{2}\mathcal{J}''(\hat \Omega)_h[V_h,V_h] + \mathcal{J}'(\hat \Omega)_h[V_h] :  V_h \in  \mathcal{V}_h \right\}.
\end{eqnarray*}
In this setting, $\mathcal{J}'(\hat \Omega)_h$ and $\mathcal{J}''(\hat \Omega)_h$ are  suitable approximations of $\mathcal J'(\hat \Omega)$ and $\mathcal J''(\hat \Omega)$ respectively.
The function $\phi$ is given by
\begin{equation}
    \phi(A):=
    \begin{cases}
        0, &|A| \leq1,
        \\
        \infty, &|A| >1.
    \end{cases}
\end{equation}

The approximations of the shape derivatives of $\mathcal J(\Omega)$ may be given in many forms, it may not be the case that one wishes to take the shape derivative of the discrete energy, one may prefer to do some post-processing.
This leads to discussions of order of discretisation and optimisation which we do not wish to include here.

Let us now state that there exists a solution.
The proof of which is almost identical to that of Theorem \ref{thm:ContinuousMinimisers}.
\begin{theorem}\label{thm:DiscreteMinimisers}
    Given fixed $h>0$ and $t\geq 0$, there exists
    \begin{equation}
        V_h^* \in \argmin \left\{ \frac{t}{2}\mathcal{J}''(\hat \Omega)_h[V_h,V_h] +  \mathcal{J}'(\hat \Omega)_h[V_h] :  V_h \in  \mathcal{V}_h,\, | D V_h | \leq 1 \right\}.
    \end{equation}
\end{theorem}
Unlike before, we did not need to make conditions on $\mathcal{J}'(\hat \Omega)_h$ or $\mathcal{J}''(\hat \Omega)_h$, this is due to the finite dimensional nature ensures that the existence of a strongly convergengent subsequence of an infimising sequence.

Let us comment that the question of in what way and with which estimates does $V_h^*$ converge as $h \to 0$ is open.
In \cite{Bar20}, a problem similar to the case $t=0$ is considered.
Using discontinuous elements, the author was able to provide estimates on the energy.
Let us note that this lack of estimate is not expected to be prohibitive to an initial analysis of our proposed algorithm.

In order to find the function $V_h^*$, we use an alternating direction method of multipliers (ADMM) approach in order to solve the above problem.
To do so, we set
\begin{equation}
    \mathcal{Q}_h := \left\{ q_h \in L^2(D;\R^{d\times d}) : q_h|_T \in P^0(T;\R^{d\times d}),\, \forall T \in \mathcal{T}_h \right\}
\end{equation}
and consider for a given $\tau >0$ the functional $\mathcal L_\tau\colon  \mathcal{V}_h \times
   \mathcal{Q}_h \times  \mathcal{Q}_h \rightarrow \mathbb R$ with
\begin{equation}\begin{split}
  \mathcal{L}_\tau(V_h,q_h;\lambda_h) :=&
  \int_D \left( \phi(q_h) + \lambda_h : (DV_h-q_h) \right) \dee x
  + \frac{\tau}{2}\|DV_h-q_h\|_{L^2(D;\R^{d\times d})}^2
  \\
  &+\frac{t}{2}\mathcal{J}''( \hat \Omega)_h[V_h,V_h] + \mathcal{J}'(\hat \Omega)_h[V_h].
\end{split}\end{equation}
The idea of ADMM is to alternatively minimise $\mathcal{L}_\tau$ over $q_h$ and $V_h$,
then perform an update to $\lambda_h$ and repeat this until a certain quantity is small.
More precisely, the algorithm has the following form: 
\begin{algorithm}
    \ \\ [2mm]
    0. {Choose $V^0_h$ and $\lambda^0_h$ such that $\mathcal{J}'(\hat \Omega)_h[V^0_h] < \infty$}
    \\[2mm]
    1. {Set $R = \infty$, $j=1$}
    \\[2mm]
    While {$R > tol$:}
    \\[2mm]
    2. {Find $q^j_h \in \argmin \{ \mathcal{L}_\tau(V^{j-1}_h,q_h;\lambda^{j-1}_h) : q_h \in  \mathcal{Q}_h, |q_h| \leq 1 \}$}
    \\[2mm]
    3. {Find $V^j_h \in \argmin \{ \mathcal{L}_\tau(V_h,q^j_h;\lambda^{j-1}_h) : V_h \in \mathcal{V}_h \}$}
    \\[2mm]
    4. {Set $\lambda^j_h = \lambda^{j-1}_h + \tau ( DV^j_h - q^j_h)$}
    \\[2mm]
    5. {Set $R = \left( \|\lambda^j_h-\lambda^{j-1}_h \|_{L^2(D ;\R^{d\times d})}^2 + \tau^2 \|DV^j_h - DV^{j-1}_h \|_{L^2(D;\R^{d\times d})}^2 \right)^{\frac{1}{2}}$ }
    \\[2mm]
    6. {Update j = j+1}
    \label{alg:ADMM}
\end{algorithm}
  
Let us  also mention \cite{BarMil20} which considers more general ADMM methods with variable $\tau$.
Particularly, in our experiments we make use of such an algorithm with variable $\tau$, namely \cite[Algorithm 3.19]{BarMil20}.

In a finite dimensional setting, since we have the existence of minimisers, c.f. Theorem \ref{thm:DiscreteMinimisers}, it is almost immediate that Algorithm \ref{alg:ADMM} converges.
It is necessary to check the sign of $\mathcal{J}''(\Omega)_h$; in the locality of a minimiser, one might expect that it is non-negative.
\begin{theorem}[ADMM converges]
    Suppose that $t \in [0,1)$ and $t\mathcal{J}''(\hat{\Omega})_h$ is a non-negative operator, then it holds that Algorithm \ref{alg:ADMM} converges.
\end{theorem}
This follows from \cite{HeYua15}.
The assumption that $t\mathcal{J}''(\Omega)_h$ is non-negative is made in order to ensure the appropriate functional is convex, which is a sufficient condition for the convergence of ADMM.
\subsection{Evaluation of \texorpdfstring{$\mathcal J'(\hat \Omega)_h$}{J'(Omega)h}}
\subsubsection{PDE--constrained shape optimisation}
Let us formulate suitable approximations of the shape derivatives derived in \ref{pdeconstrained}. Given the polygonal domain
$\hat \Omega$ we denote by  $\mathbb{S}_h(\hat \Omega)$  the space of linear finite elements on $\hat \Omega$ (resolved by a subtriangulation of
$\mathcal{T}_h$) which vanish on $\partial \hat \Omega$. If the constraint is given by \eqref{poisson} we set 
\begin{equation}
    \mathcal{J}'(\hat \Omega)_h[V_h] =
    \int_{\hat \Omega} \left( j(\cdot,\hat y_h)\Div V_h + j_x(\cdot,\hat y_h)\cdot V_h + \A[V_h] \nabla \hat y_h \cdot \nabla \hat p_h
     - \hat p_h \Div (F V_h ) \right) \dee \hat x,
\end{equation}
for $V_h \in \mathcal{V}_h$, where $\hat y_h,\hat p_h \in \mathbb{S}_h(\hat \Omega)$ satisfy
\begin{equation}
    \begin{split}
        \int_{\hat \Omega} \nabla \hat y_h \cdot \nabla \eta_h \, \dee \hat x
        =&
        \int_{\hat \Omega} F \eta_h \, \dee \hat x  \quad \forall \eta_h \in \mathbb{S}_h(\hat \Omega),
        \\
        \int_{\hat \Omega} \nabla \hat p_h \cdot \nabla \eta_h \, \dee \hat x
        =&
        -\int_{\hat \Omega} j_y(\cdot,\hat y_h) \eta_h \, \dee \hat x  \quad \forall \eta_h \in \mathbb{S}_h(\hat \Omega).
    \end{split}
\end{equation}
On the other hand, if the constraint is given by \eqref{bilaplace} then we let 
\begin{equation}
    \begin{split}
        \mathcal{J}'(\hat \Omega)_h[V_h]
        =&
        \int_{\hat \Omega} \left( j(\cdot, \hat y_{h,1}) \Div V_h + j_x(\cdot,\hat y_{h,1})\cdot V_h
        +
        \A[V]\nabla \hat y_{h,1} \cdot \nabla \hat p_{h,2} - \hat y_{h,2} \hat p_{h,2} \Div V_h \right) \, \dee \hat x
        \\
        &
         +\int_{\hat \Omega} \left( \A[V_h] \nabla \hat y_{h,2} \cdot \nabla \hat p_{h,1}  - \Div (V_h F) \hat p_{h,1} \right) \, \dee \hat x,
    \end{split}
\end{equation}
for $V_h \in \mathcal{V}_h$.
Here, $\hat y_h=(\hat y_{h,1},\hat y_{h,2}) \in (\mathbb{S}_h(\hat \Omega))^2$ satisfies
\begin{align}
    \int_{\hat \Omega} \nabla \hat y_{h,1} \cdot \nabla \eta_h \, \dee \hat x  &= \int_{\hat \Omega} \hat y_{h,2} \eta_h \, \dee \hat x \quad \forall \eta_h \in \mathbb{S}_h(\hat \Omega),
    \\
    \int_{\hat \Omega} \nabla \hat y_{h,2} \cdot \nabla \eta_h \, \dee \hat x  &= \int_{\hat \Omega} F \eta_h \, \dee \hat x  \quad \forall \eta_h \in \mathbb{S}_h(\hat \Omega)
\end{align}
while  $\hat p_h=(\hat p_{h,1},\hat p_{h,2}) \in (\mathbb{S}_h(\hat \Omega))^2$ satisfies
\begin{align}
    \int_{\hat \Omega} \nabla \hat p_{h,2} \cdot \nabla \eta_h \, \dee \hat x  &= -\int_{\hat \Omega} j_y(\cdot,y_{h,1})\eta_h \, \dee \hat x \quad \forall \eta_h \in\mathbb{S}_h(\hat \Omega),
    \\
    \int_{\hat \Omega} \nabla \hat p_{h,1} \cdot \nabla \eta_h \, \dee \hat x  &= \int_{\hat \Omega} \hat p_{h,2} \eta_h \, \dee \hat x  \quad \forall \eta_h \in \mathbb{S}_h(\hat \Omega).
\end{align}

\subsubsection{Optimisation of the first eigenvalue for the Laplacian}
For a given polygonal domain we determine $\hat z_h \in \mathbb{S}_h(\hat \Omega)$ and $\hat \lambda_h>0$ such that $\int_{\hat \Omega} \hat y_h^2 \, \dee x=1$ and
\begin{displaymath}
  \hat \lambda_h = \inf \left\{\int_{\hat \Omega} |\nabla \hat z_h|^2 \, \dee \hat x  : \hat z_h  \in \mathbb{S}_h(\hat \Omega),\, \int_{\hat \Omega} \hat z_h^2\, \dee \hat x = 1 \right\} =
  \int_{\hat \Omega} | \nabla \hat z_h |^2 \, \dee \hat x.
\end{displaymath}
Supposing that the eigenvalue  $\lambda_h$ is simple we let, recalling \eqref{eq:eigenValueDerivative}
\begin{equation}
  \mathcal{J}'(\hat \Omega)_h[V_h] = \int_{\hat \Omega} \left( \A[V_h] \nabla \hat z_h \cdot \nabla \hat z_h - \hat\lambda_h \hat z_h^2 \Div V_h
    \right) \, \dee \hat x \mbox{ for } V_h \in \mathcal{V}_h.
\end{equation}

\section{Numerical experiments}\label{sec:Applications}
We now {provide numerical experiments for the applications we have described}.
In the integrals for the energy we use quadrature of order 2, while for the shape derivatives, we are using
the order which is automatically decided by the software. \\
As mentioned above we will solve the state and adjoint PDEs with a finite element approximation.
The finite element approximation is performed with DUNE \cite{duneReference}, making particular use of the DUNE Python bindings \cite{DunePython1,DunePython2}.
We consider a construction of update direction using four different approaches.
Our approaches will be:
\begin{itemize}
    \item
    The direction of steepest descent method using the $W^{1,\infty}$-topology, constructed with an adaptive ADMM method, as mentioned after Algorithm \ref{alg:ADMM}.
    This will be referred to as $p = \infty$.
    \item
    A Newton-type direction, which will be a discrete minimiser of \eqref{eq:NewtonProblem} for a given $t>0$, referred to as the Newton method.
    Much like the $p=\infty$ case above, this will be constructed with an adaptive ADMM method.
    \item
    To compare against existing approaches, for $p=2,4$, we will consider the minimiser of $\mathcal{V}_h \ni V_h \mapsto \mathcal{J}'(\Omega)_h + \frac{1}{p}\int_D |DV_h|^p\dee x $.
    In the case that $p=2$, this is seen to coincide with the discrete case of \eqref{eq:MinInHilbertSpace} with $H = H^1_0(D;\R^d)$ with the Dirichlet inner product.
    We will refer to these cases by their $p$ value.
\end{itemize}
The discrete functions produced by the $p=2$ and $p=4$ methods will be normalised so that they have a $W^{1,\infty}(D;\R^d)$ semi-norm of $1$.
This normalisation is performed so that we need not check whether the mesh has overlapped.
For each of the experiments, we will set the hold-all domain to be the box $D = (-2,2)^2$.
With these directions, we will move the vertices of our mesh {according to an Armijo step rule}.
We will stop after 20 shape updates have been made.
In most cases the domain has become close to stationary at this point and the Armijio step-size has become rather small.

The energy along the iterations will be plotted.
In the case that the minimiser is known, the origin will be offset by the known value, when the minimiser is not known, the origin will be offset by the smallest value attained in the experiment.

\subsection{Minimisation without a PDE constraint}\label{sec:application:noPDE}
Here we will consider that there is no PDE constraint, so that the map $e$ need not be included.  We comment that the no PDE example may be derived as an example from the following Section \ref{sec:application:poisson} where one chooses right hand side data $F=0$ so that the state constraint guarantees $y=0$.

For this experiment, the main contributions to the error is that induced by the quadrature rules when calculating the energy and the shape derivative, as well as the direction of descent with the chosen method.

\subsubsection{No PDE experiment 1}\label{sec:experiment:NOPDE:Experiment1}
For this problem, we consider
\begin{equation}
    j(x,y) = -Z(x)
\end{equation}
where
\begin{equation}
    Z(x) = 
    \begin{cases}
        \cos(0.5\pi x_1)\cos(0.5\pi x_2) & |x_1| \mbox{ and } |x_2| \leq 1,\\
        \frac{\pi}{4}(1-x_1^2) & |x_1|>1 \mbox{ and } |x_2|<1,\\
        \frac{\pi}{4}(1-x_2^2) & |x_1|<1 \mbox{ and } |x_2|>1,\\
        \frac{\pi}{4}(2-x_1^2-x_2^2) & \mbox{otherwise}.
    \end{cases}
\end{equation}
For the Newton direction we take $t= 0.0625$.
One expects the square $(-1,1)^2$ to be a minimiser of $\mathcal{J}$.
It holds that
$\mathcal{J}\left((-1,1)^2\right) = -\frac{16}{\pi ^2}$.
We start with the initial domain $(-1.5,-1)\times (-1,1)$.
The triangulation of the domain and hold-all is displayed in Figure \ref{fig:experiment:NOPDE:Experiment1:InitialDomain}.
In Figure \ref{fig:experiment:NOPDE:Experiment1:graphs}, the energy
of shapes along the minimising sequences we produce are given.
\begin{figure}\centering
    \begin{subfigure}[b]{.49\linewidth}
    \centering
    \includegraphics[width=.8\linewidth]{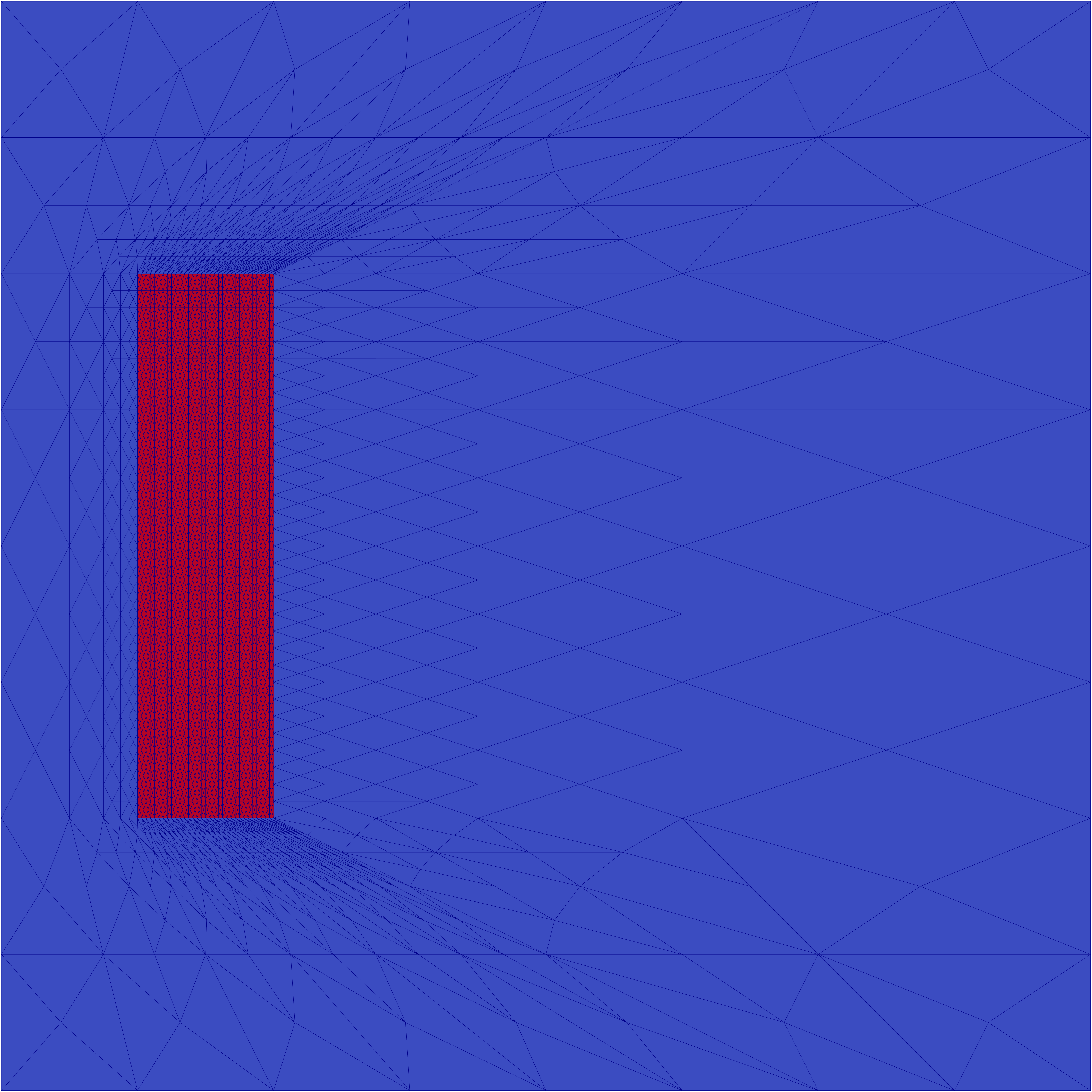}
    \caption{Initial domain for the first \emph{No PDE} experiment in Section \ref{sec:experiment:NOPDE:Experiment1}, with $(-1.5,-1)\times(-1,1)$ in red and the hold-all, $(-2,2)^2$ in blue.}
    \label{fig:experiment:NOPDE:Experiment1:InitialDomain}
    \end{subfigure}\hfill
    \begin{subfigure}[b]{.49\linewidth}
    \centering
    \includegraphics[width=1\linewidth]{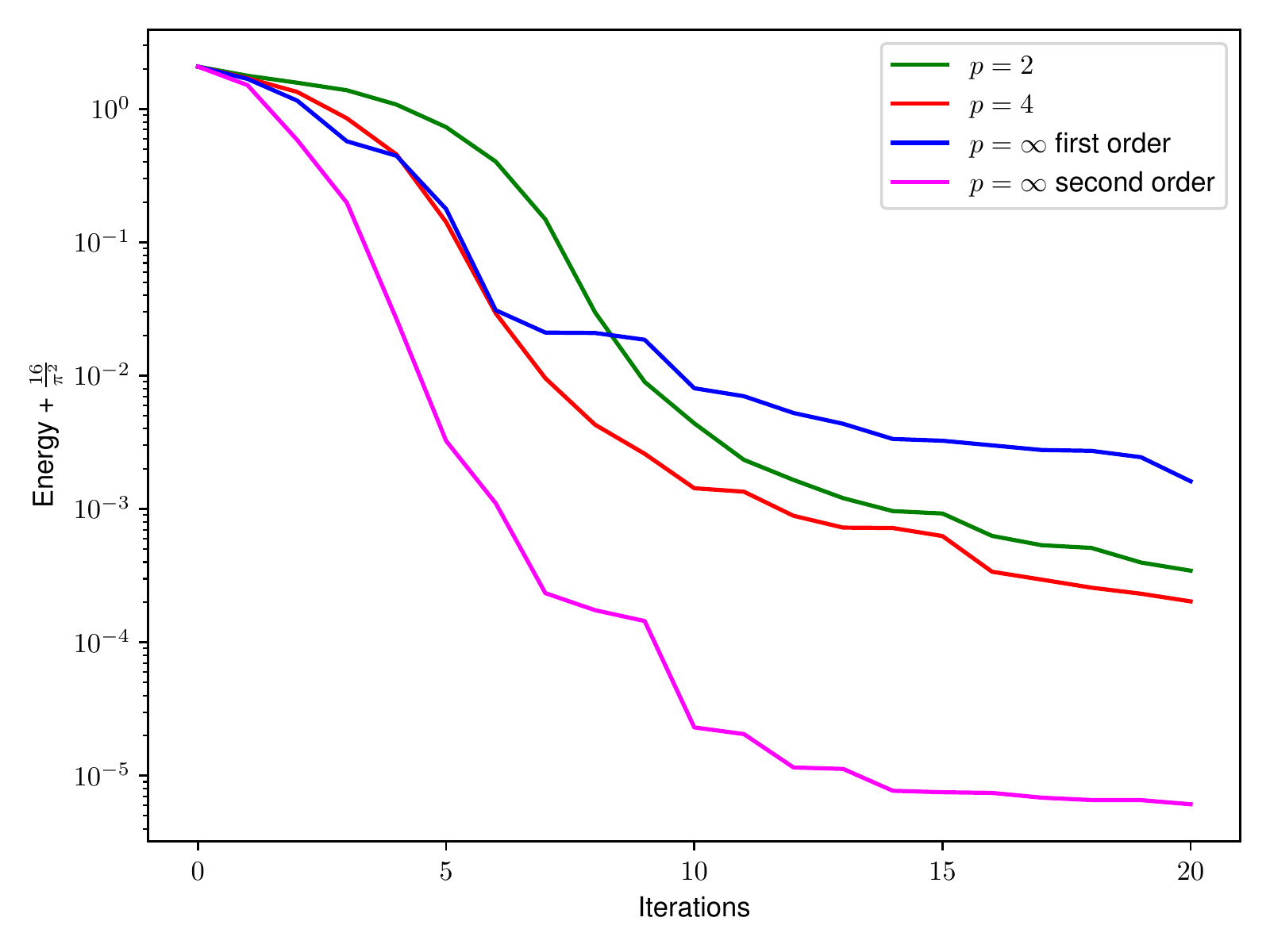}
    \caption{Graph of the energy for the iterates in the first \emph{No PDE} experiment in Section \ref{sec:experiment:NOPDE:Experiment1}.
    It is seen that the Newton-type method is energetically performing the best while the first order infinity method appears to struggle compared to the traditional $p=2$ method.
    }
    \label{fig:experiment:NOPDE:Experiment1:graphs}
    \end{subfigure}
    \caption{Initial mesh and graph of the energy for the experiment in Section \ref{sec:experiment:NOPDE:Experiment1}.} 
\end{figure}
In Figure \ref{fig:experiment:NOPDE:Experiment1:meshes}, the meshes of the final domains $\Omega$ for each of the methods are given.
\begin{figure}
    \vspace{-.45\linewidth} 
    \centering
    \adjustbox{trim = {0\width} {0.5\height} {0.\width} {-1\height}, clip ,width = .45\linewidth }{
    \includegraphics[width=.35\linewidth]{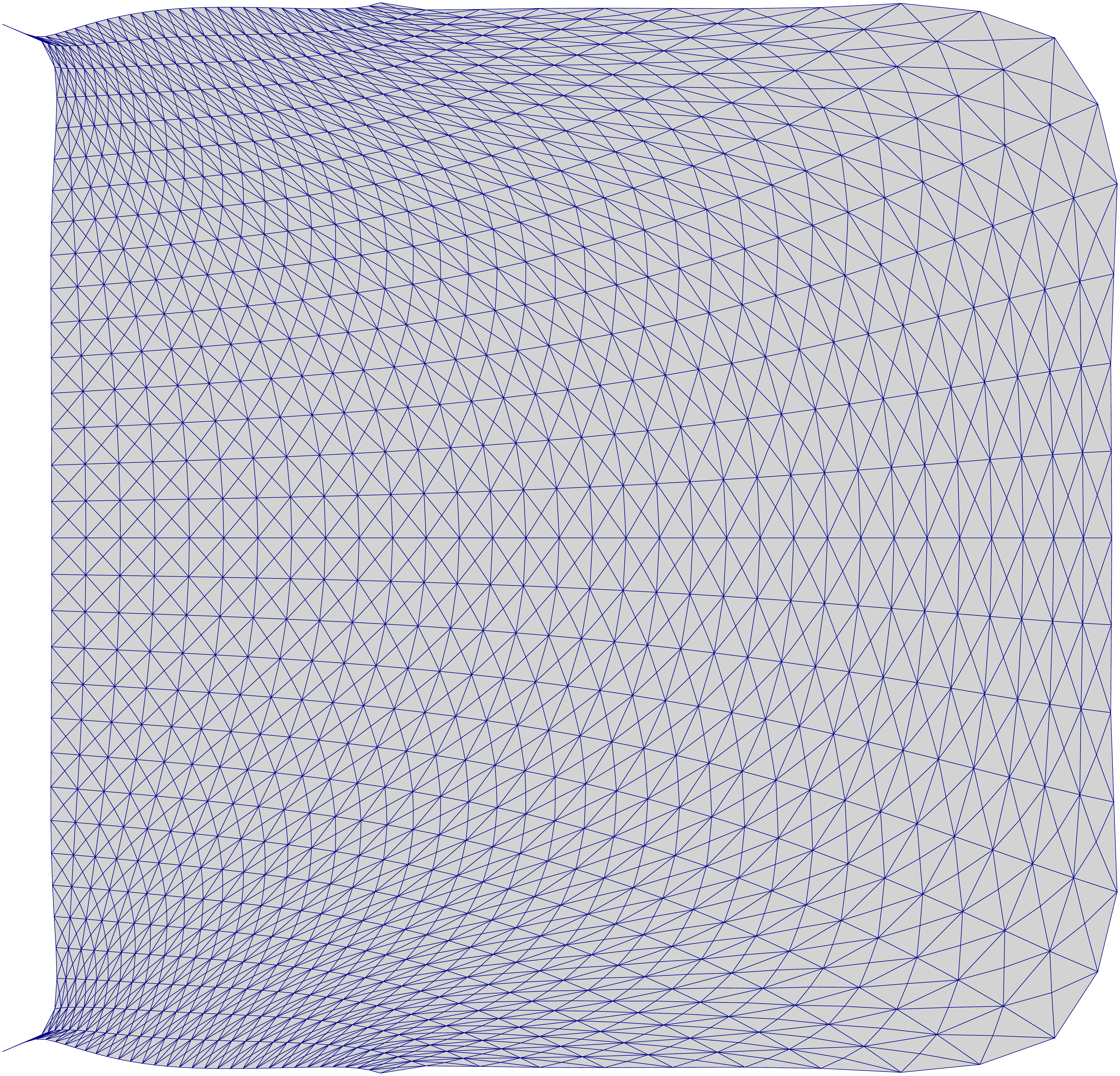}}\hspace{.05\linewidth}
    \adjustbox{trim = {0.\width} {0.5\height} {0\width} {-1\height}, clip ,width = .45\linewidth }{
    \includegraphics[width=.35\linewidth]{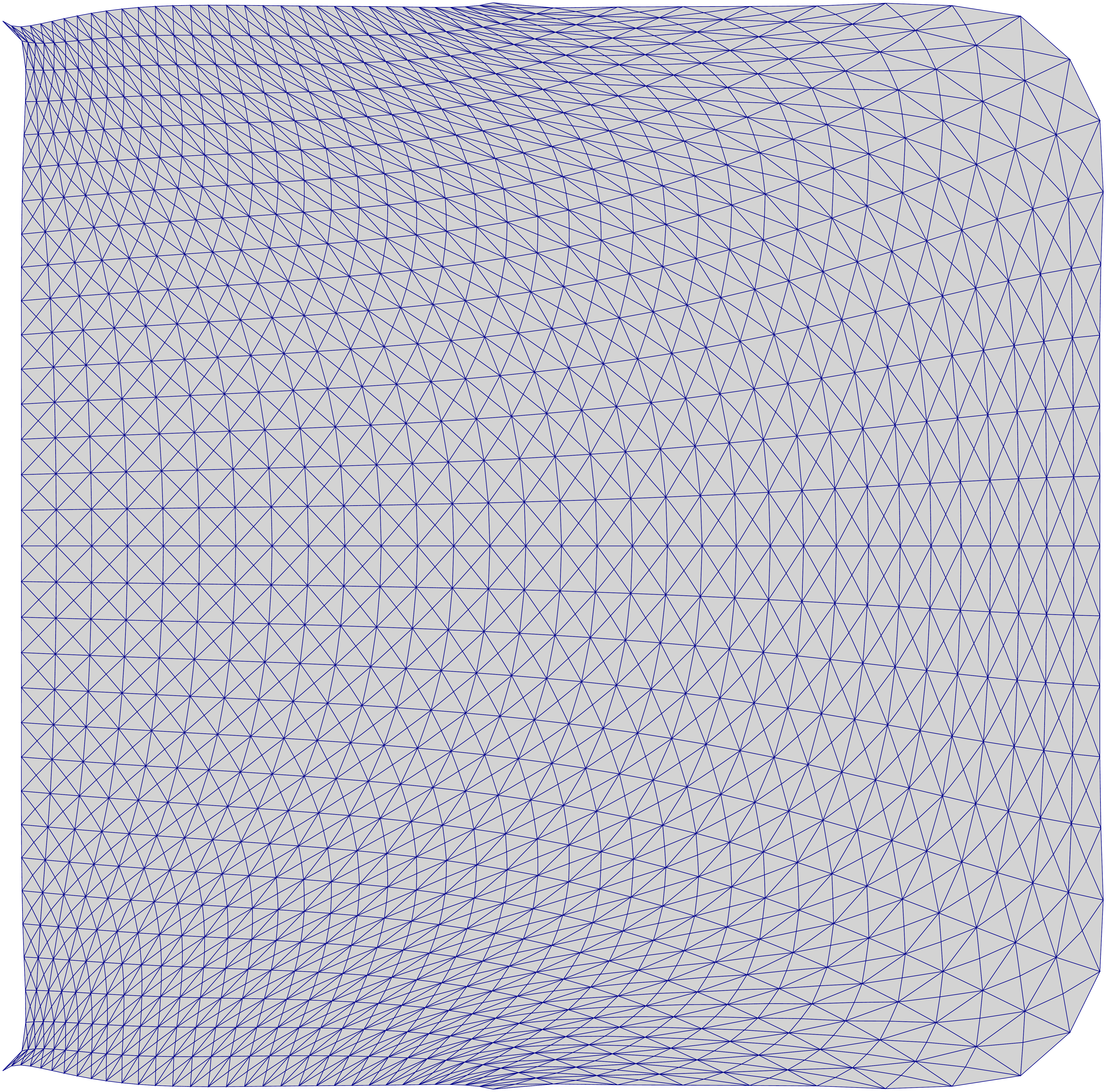}}
    \\
    \adjustbox{trim = {0\width} {0\height} {0.\width} {.5\height}, clip ,width = .45\linewidth }{
    \includegraphics[width=.35\linewidth]{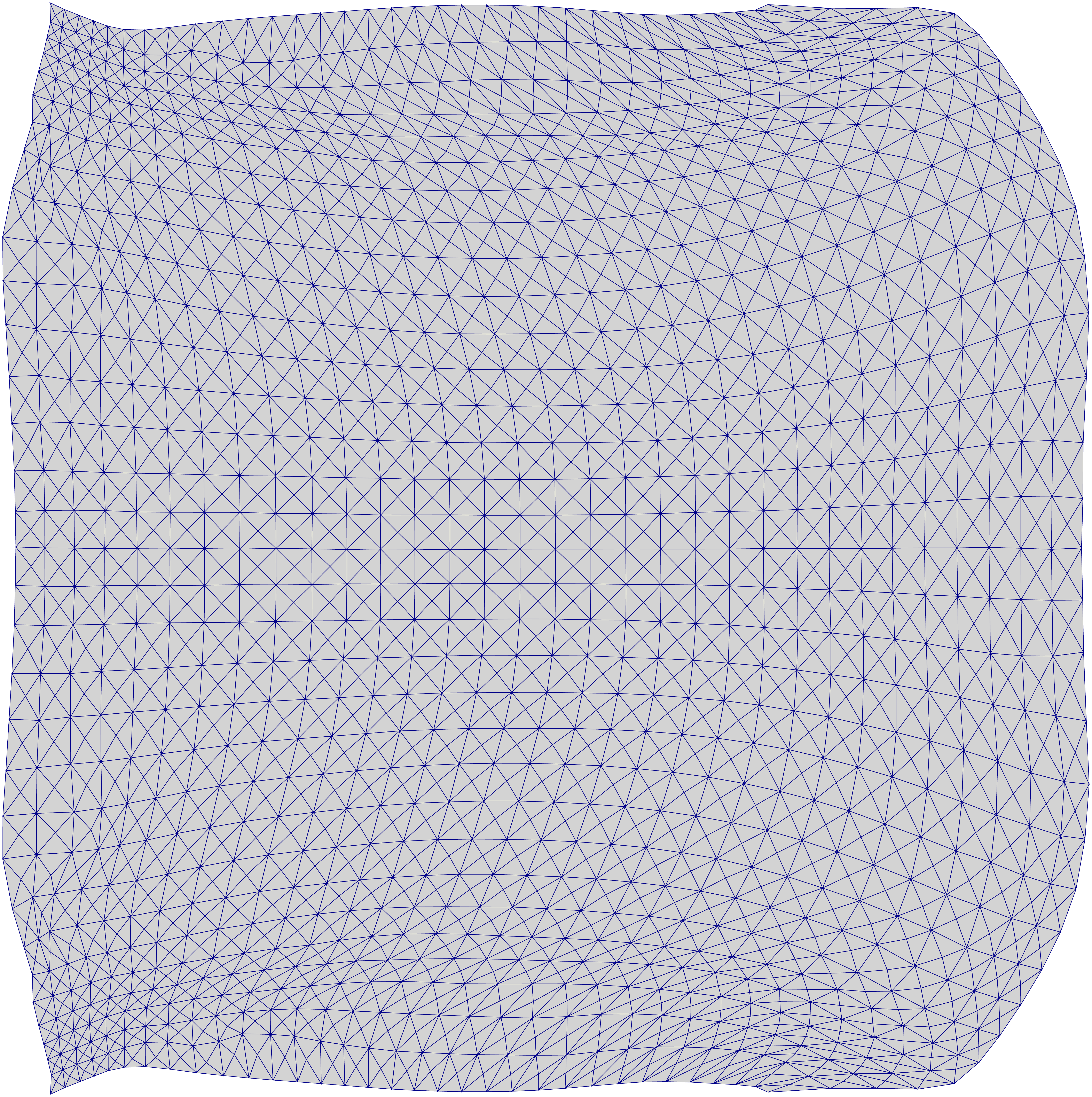}}\hspace{.05\linewidth}
    \adjustbox{trim = {0.\width} {0\height} {0\width} {.5\height}, clip ,width = .45\linewidth }{
    \includegraphics[width=.35\linewidth]{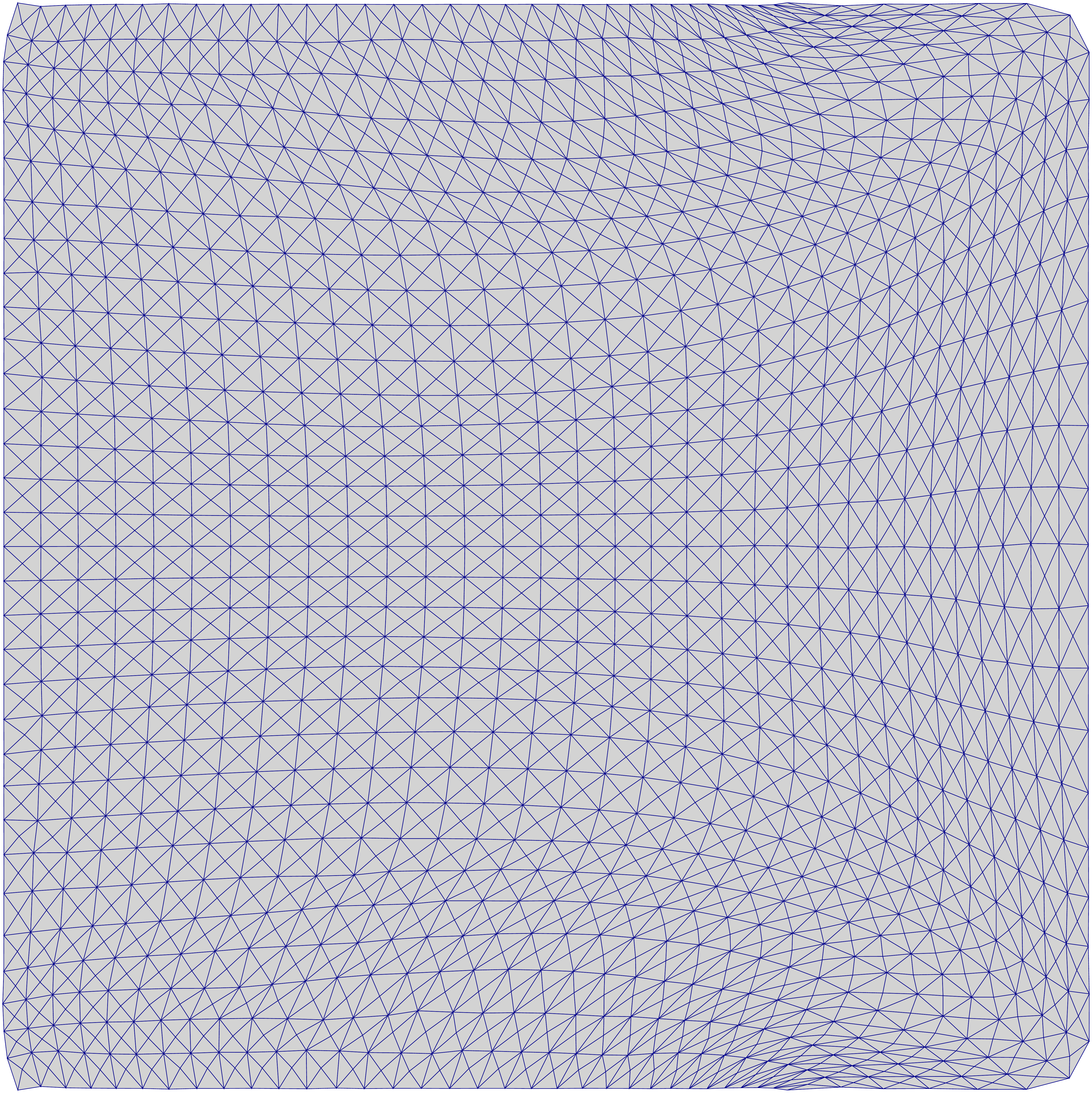}}
    \caption{The meshes of $\Omega$ for the final domains produced in the first \emph{No PDE} experiment in Section \ref{sec:experiment:NOPDE:Experiment1}: top left to bottom, $p=2,4,\infty$ and second order.
    Due to symmetry of the result, we show only a half of each mesh.
    It is seen that none of the methods correctly captures the corners which are expected from the minimising shape.
    This is unsurprising as there is not a particularly large influence of the energy around the corners, where $Z$ is quite close to zero.
    The Newton-type method is the closest to forming corners.
    }
    \label{fig:experiment:NOPDE:Experiment1:meshes}
\end{figure}

\subsubsection{No PDE experiment 2}\label{sec:experiment:NOPDE:Experiment2}
For this problem, we consider
\begin{equation}
    j(x,y) = \frac{1}{2} Z(x)^2
\end{equation}
where for given $\epsilon >0$,
\begin{equation}
    Z(x) = \sqrt{(x_1+x_2)^2 + \epsilon} + \sqrt{(x_1-x_2)^2 + \epsilon }
\end{equation}
is a smooth approximation to $|x_1+x_2| + |x_1-x_2|$.
For the Newton direction we take $t= 0.125$.
A very similar experiment with the non-smooth energy was considered in \cite{DecHerHin21}.
This smooth approximation is used because we intend to employ the Newton method for which it would be useful to have (weak) second derivatives of $Z$.
Without any constraint, we know that the theoretical minimiser is degenerate, a measure zero set.
To avoid this, we will fix the area to be constrained equal to $4$.
We expect this to have minimiser close to the square $(-1,1)^2$ which, for $\epsilon = 0$, has energy $4$.
Our directions of descent will only preserve the area constraint in a linear sense by restricting to $V$ with $\int_\Omega \Div V = 0$.
We will perform a projection step to fix the area after each update.

We take $\epsilon = 10^{-4}$ and start with an approximation of a ball of radius $\frac{2}{\sqrt{\pi}}$ at the origin.
The triangulation of the domain and hold-all is displayed in Figure \ref{fig:experiment:NOPDE:Experiment2:InitialDomain}.
In Figure \ref{fig:experiment:NOPDE:Experiment2:graphs}, the energy of shapes along the minimising sequences we produce are given.
\begin{figure}\centering
    \begin{subfigure}[b]{.49\linewidth}
    \centering
    \includegraphics[width=.8\linewidth]{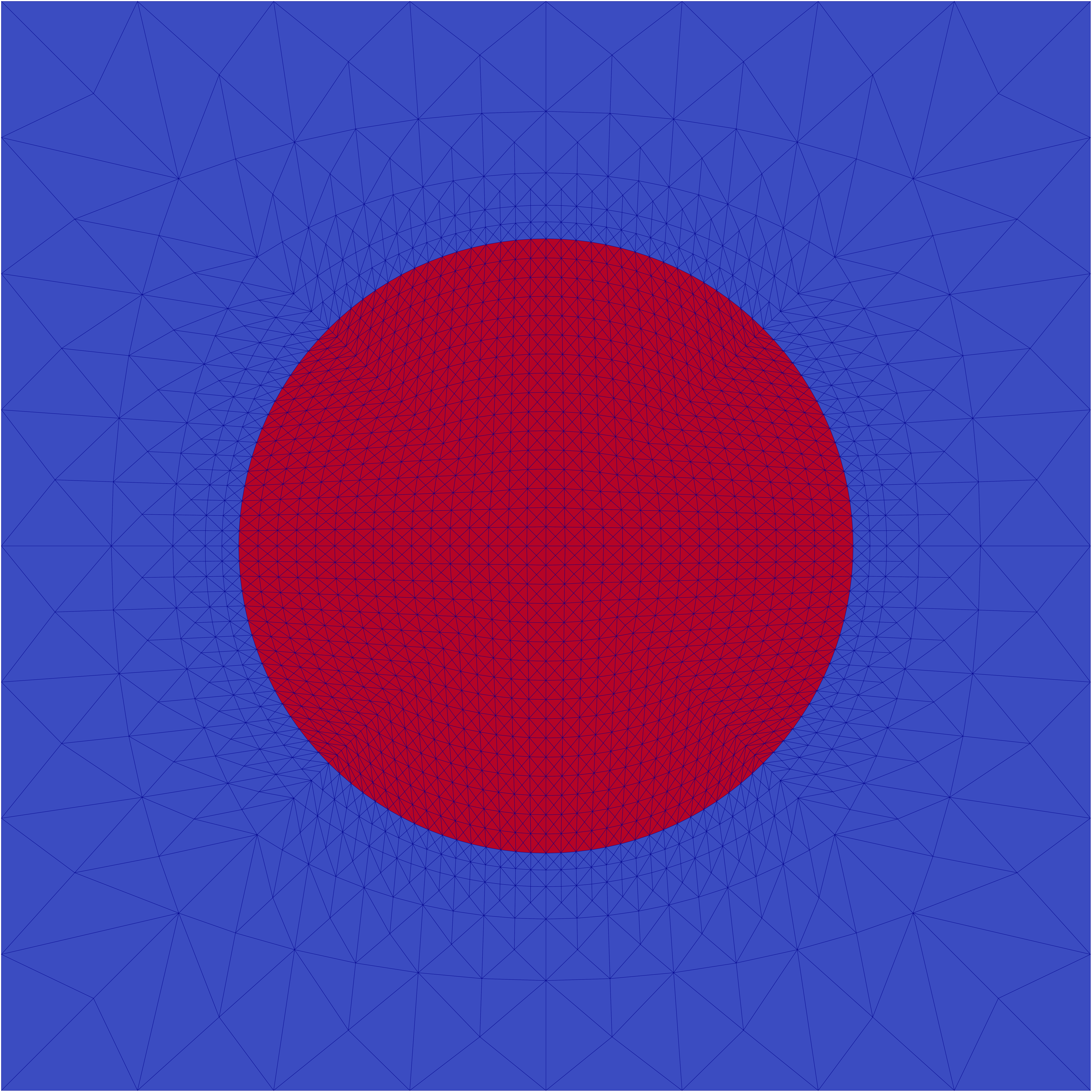}
    \caption{Initial domain for the second \emph{No PDE} experiment in Section \ref{sec:experiment:NOPDE:Experiment2}, an approximation of the ball of radius $\frac{2}{\sqrt{\pi}}$ at the origin, is in red and the hold-all, $(-2,2)^2$ in blue.}
    \label{fig:experiment:NOPDE:Experiment2:InitialDomain}
    \end{subfigure}\hfill
    \begin{subfigure}[b]{.49\linewidth}
    \centering
    \includegraphics[width=1\linewidth]{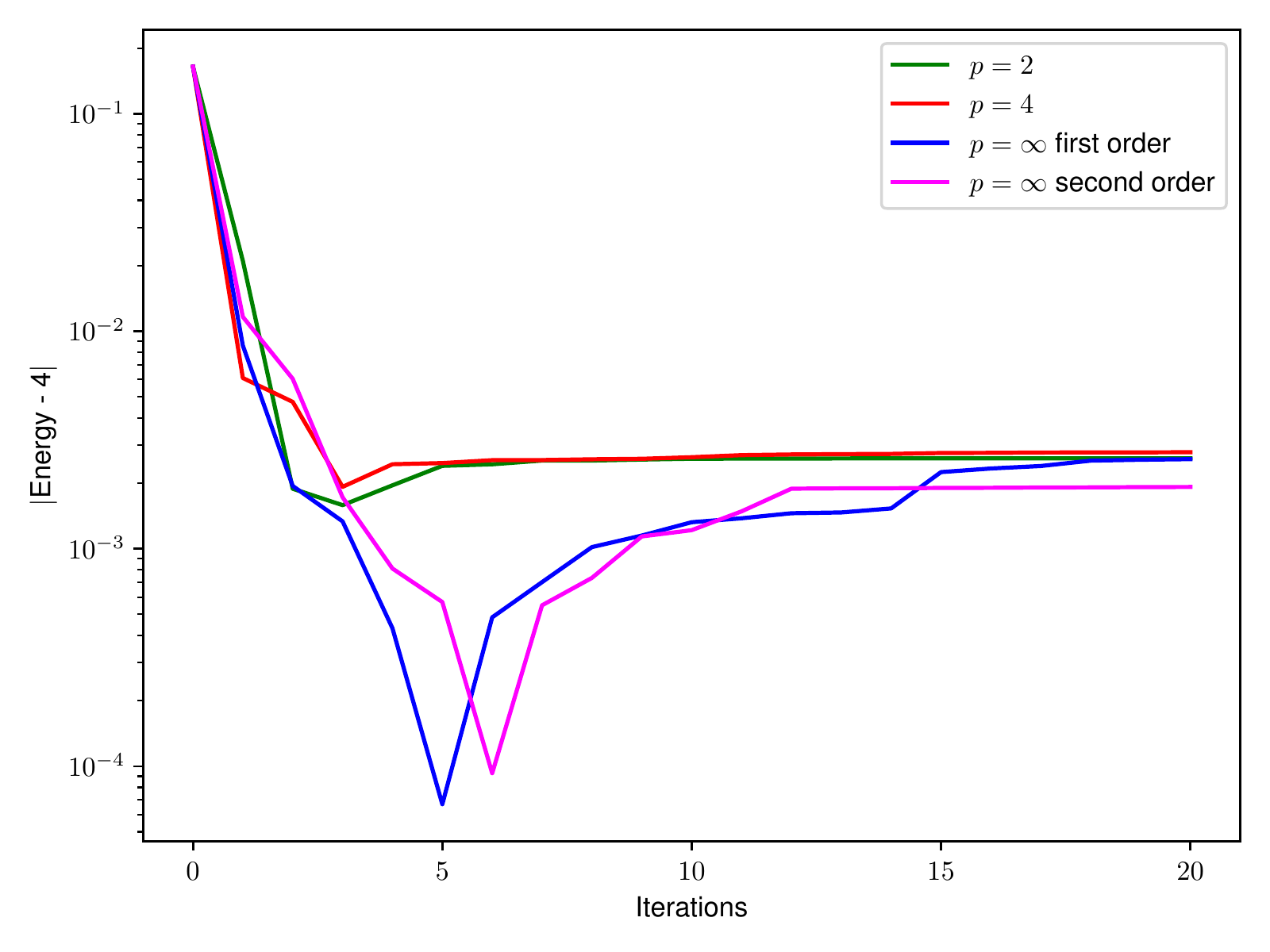}
    \caption{Graph of the energy for the iterates in the second \emph{No PDE} experiment in Section \ref{sec:experiment:NOPDE:Experiment2}.
    Let us note that the bumps in the graphs are due to the absolute value we are using.
    It is seen that the Newton method has the closest energy.}
    \label{fig:experiment:NOPDE:Experiment2:graphs}
    \end{subfigure}
    \caption{Initial mesh and graph of the energy for the experiment in Section \ref{sec:experiment:NOPDE:Experiment2}.}
\end{figure}
In Figure \ref{fig:experiment:NOPDE:Experiment2:meshes}, the meshes for the final domains $\Omega$ for each of the methods are given.
\begin{figure}
    \vspace{-.45\linewidth} 
    \centering
    \adjustbox{trim = {0\width} {0\height} {0.5\width} {-1\height}, clip ,width = .225\linewidth }{
    \includegraphics[width=.35\linewidth]{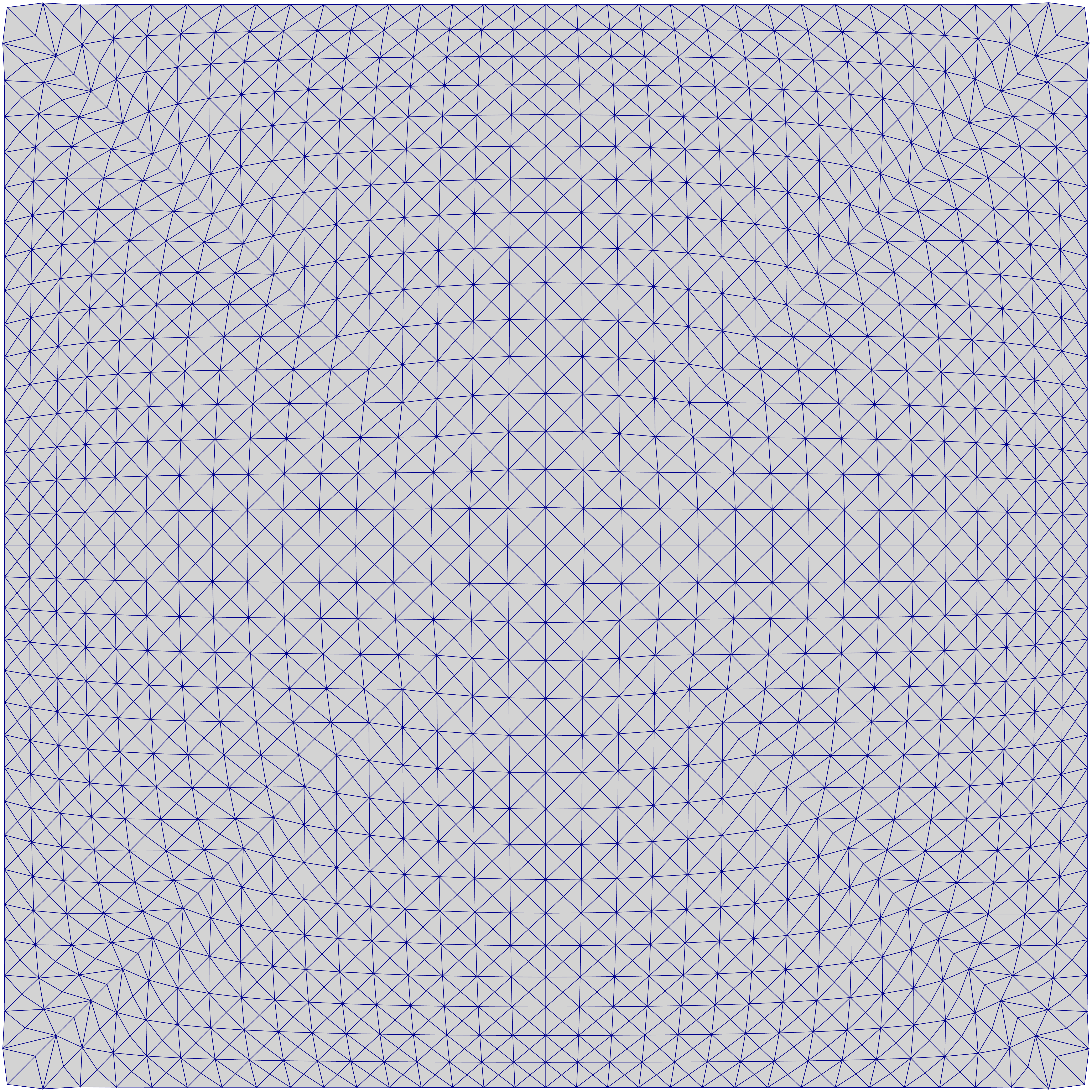}}
    \adjustbox{trim = {0.5\width} {0\height} {0\width} {-1\height}, clip ,width = .225\linewidth }{
    \includegraphics[width=.35\linewidth]{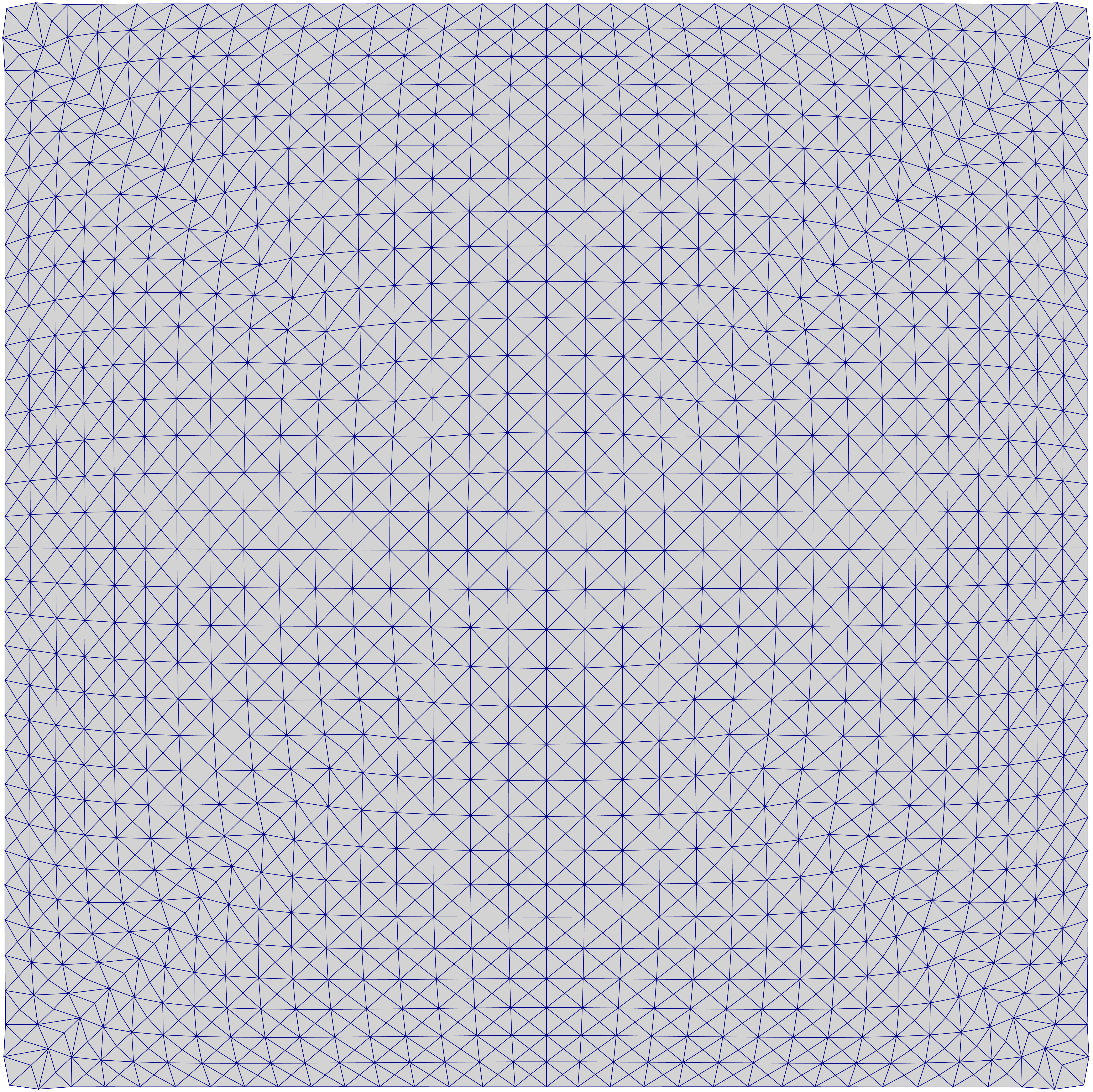}}
    \hspace{0.05\linewidth}
    \adjustbox{trim = {0\width} {0\height} {0.5\width} {.\height}, clip ,width = .225\linewidth }{
    \includegraphics[width=.35\linewidth]{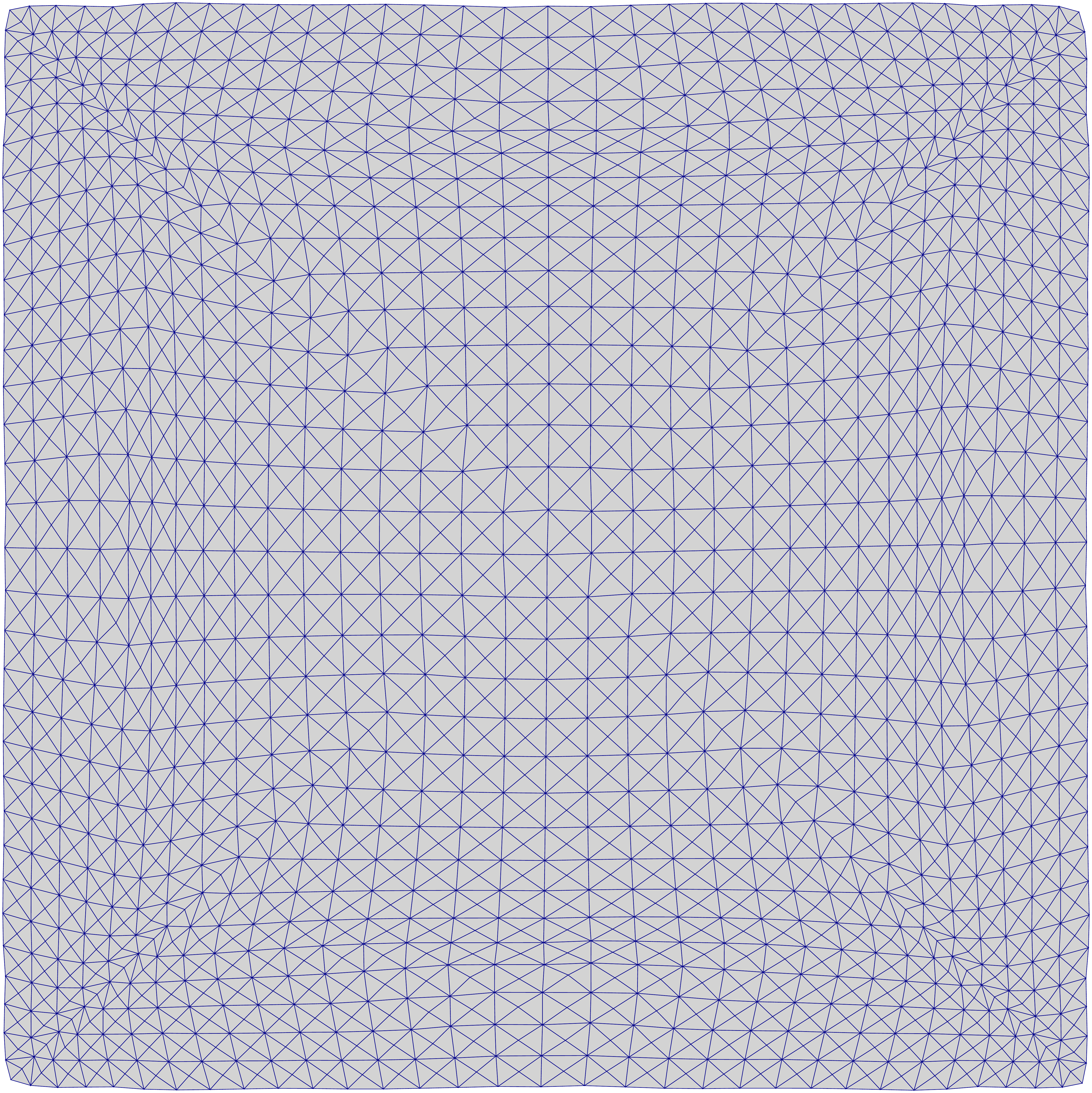}}
    \adjustbox{trim = {0.5\width} {0\height} {0\width} {.\height}, clip ,width = .225\linewidth }{
    \includegraphics[width=.35\linewidth]{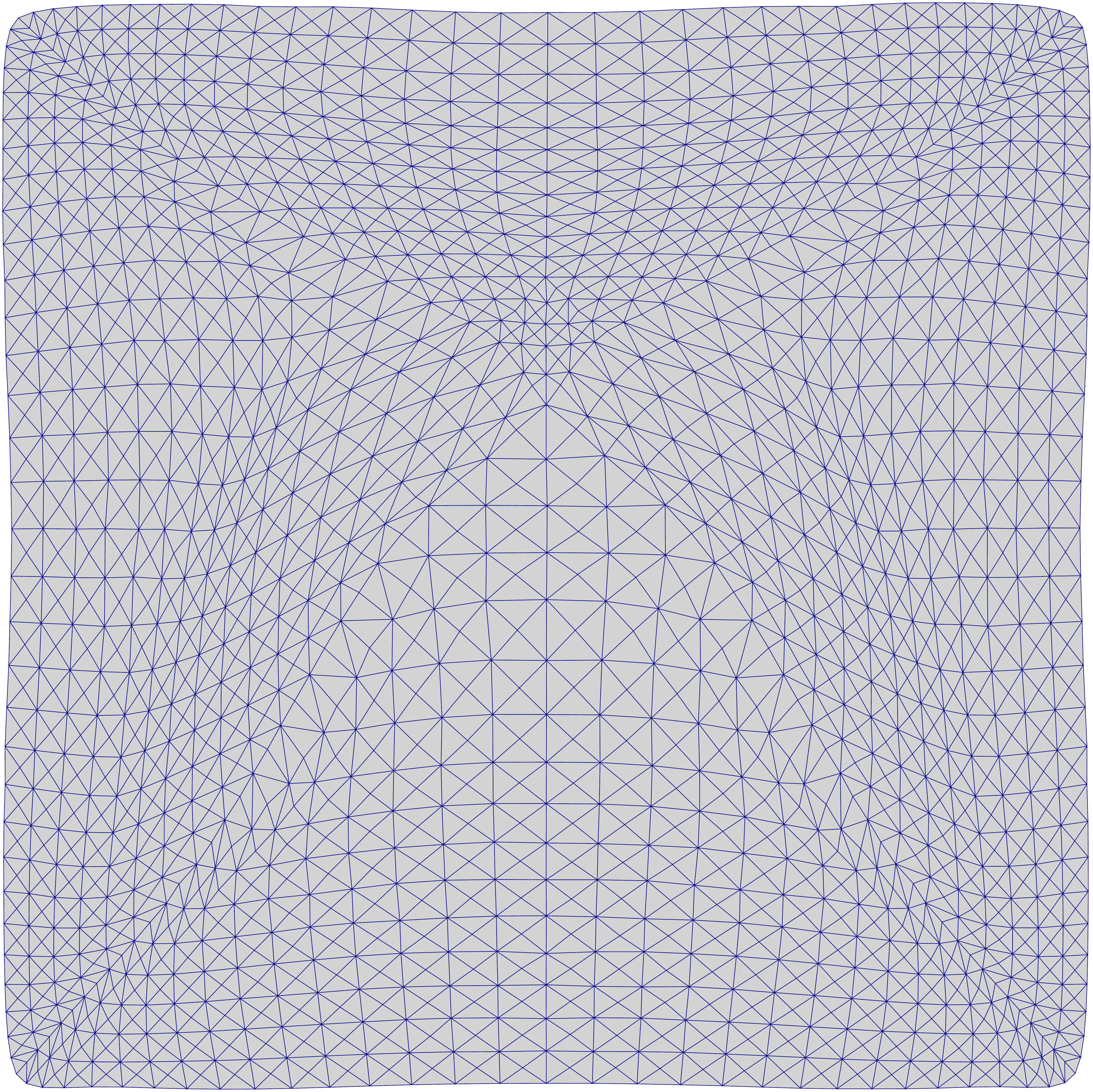}}
    \caption{The meshes of $\Omega$ for the final domains produced in the second \emph{No PDE} experiment in Section \ref{sec:experiment:NOPDE:Experiment2}: left to right, $p=2,4,\infty$ and second order.
    Due to symmetry of the result, we show only a half of each mesh.
    Generally, the first order methods provide a rather good approximation of the expected minimising shape of a square with $p=\infty$ having the most regular triangles around the corners.
    The Newton method does not quite form the corners we expect.}
    \label{fig:experiment:NOPDE:Experiment2:meshes}
\end{figure}

\subsection{A Poisson problem}\label{sec:application:poisson}

\subsubsection{Poisson experiment 1}\label{sec:experiment:Poisson:Experiment1}
For our first experiment we consider $j(x,y) = y$ and $F(x) =2.5 (x_1 + 0.5 - x_2^2)^2 + x_1^2 + x_2^2 - 1$.
This has appeared in \cite{EtlHerzLoa20,HerLoa21}, for example, as a benchmark for the comparisons of shape optimisation algorithms.
For the Newton direction we take $t= 0.125$.
The minimising shape is not explicitly known, however it appears to be a shape not so dissimilar to a kidney.
Similarly, the energy of a minimiser is not known.

We start with an approximation of an ellipse with semiaxes $\frac{2}{\sqrt{\pi}}$ and $\frac{1}{\sqrt{\pi}}$ centred at the origin.
The triangulation of the domain and hold-all is displayed in Figure \ref{fig:experiment:Poisson:Experiment1:InitialDomain}.
In Figure \ref{fig:experiment:Poisson:Experiment1:graphs}, the energy of shapes along the minimising sequences we produce are given.
\begin{figure}\centering
    \begin{subfigure}[b]{.49\linewidth}
    \centering
    \includegraphics[width=.8\linewidth]{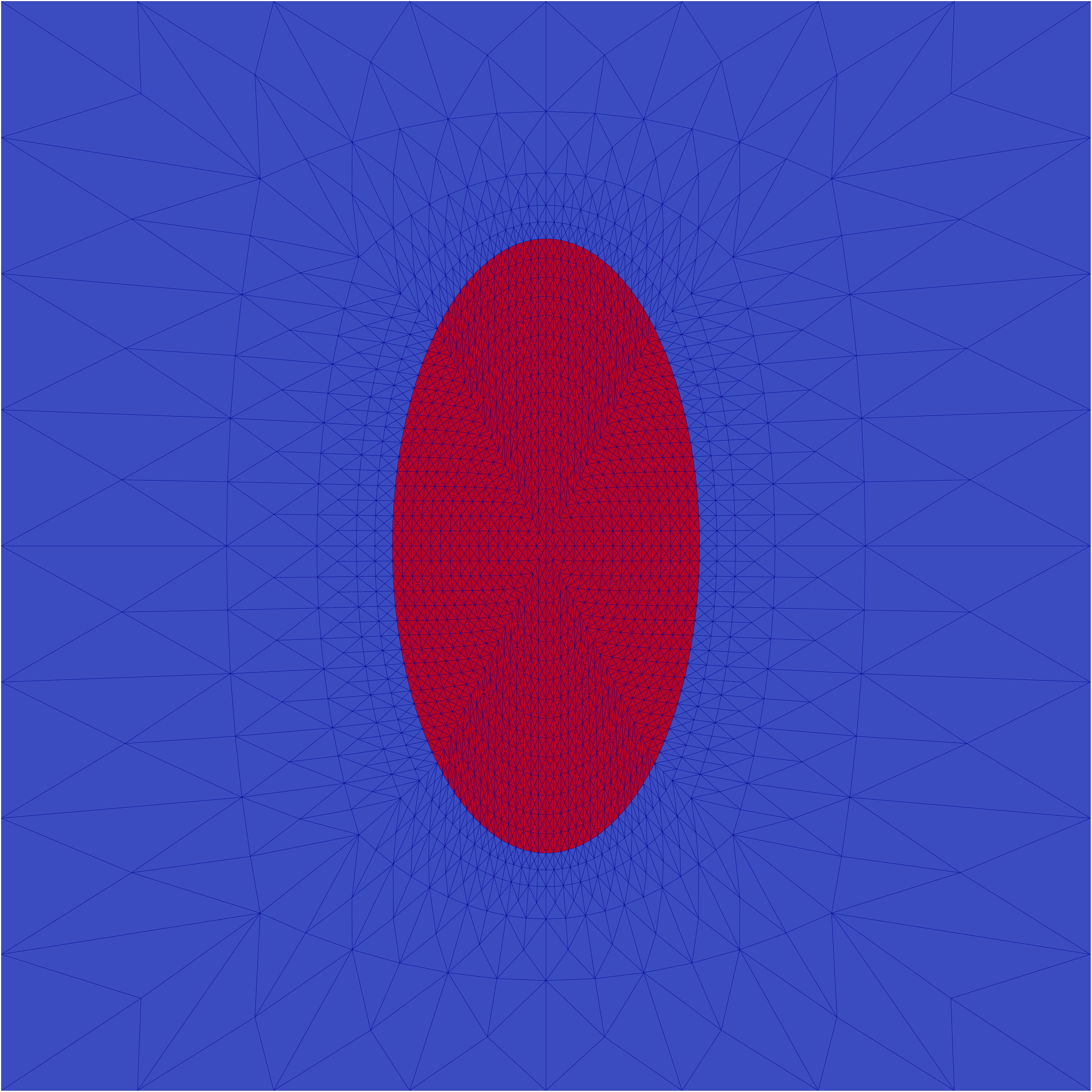}
    \caption{Initial domain for the first \emph{Poisson} experiment in Section \ref{sec:experiment:Poisson:Experiment1}, an approximation of the ellipse with semiaxes $\frac{2}{\sqrt{\pi}}$ and $\frac{1}{\sqrt{\pi}}$ at the origin, is in red and the hold-all, $(-2,2)^2$ in blue.}
    \label{fig:experiment:Poisson:Experiment1:InitialDomain}
    \end{subfigure}
    \begin{subfigure}[b]{.49\linewidth}
    \centering
    \includegraphics[width=1\linewidth]{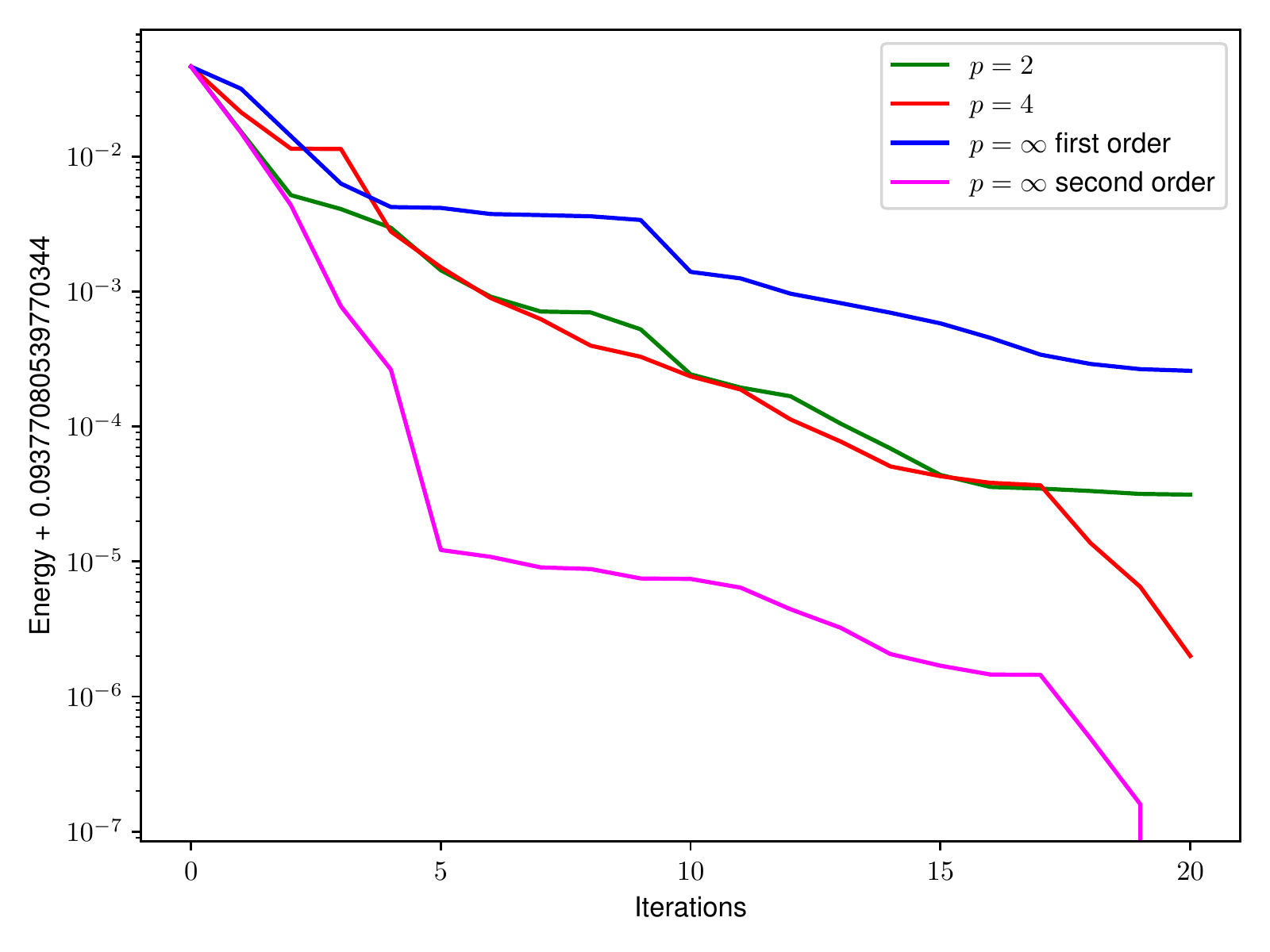}
    \caption{Graph of the energy for the iterates in the first \emph{Poisson} experiment in Section \ref{sec:experiment:Poisson:Experiment1}.
    We see that $p=\infty$ does not perform as well as the traditional $p=2$ method.
    The Newton method performs well energetically.}
    \label{fig:experiment:Poisson:Experiment1:graphs}
    \end{subfigure}
    \caption{Initial mesh and graph of the energy for the experiment in Section \ref{sec:experiment:Poisson:Experiment1}. }
\end{figure}
In Figure \ref{fig:experiment:Poisson:Experiment1:meshes}, the meshes for the final domains $\Omega$ for each of the methods are given.
\begin{figure}
    \vspace{-.45\linewidth} 
    \centering
    \adjustbox{trim = {0\width} {0.5\height} {0\width} {-1\height}, clip ,width = .35\linewidth }{\includegraphics[width=.35\linewidth]{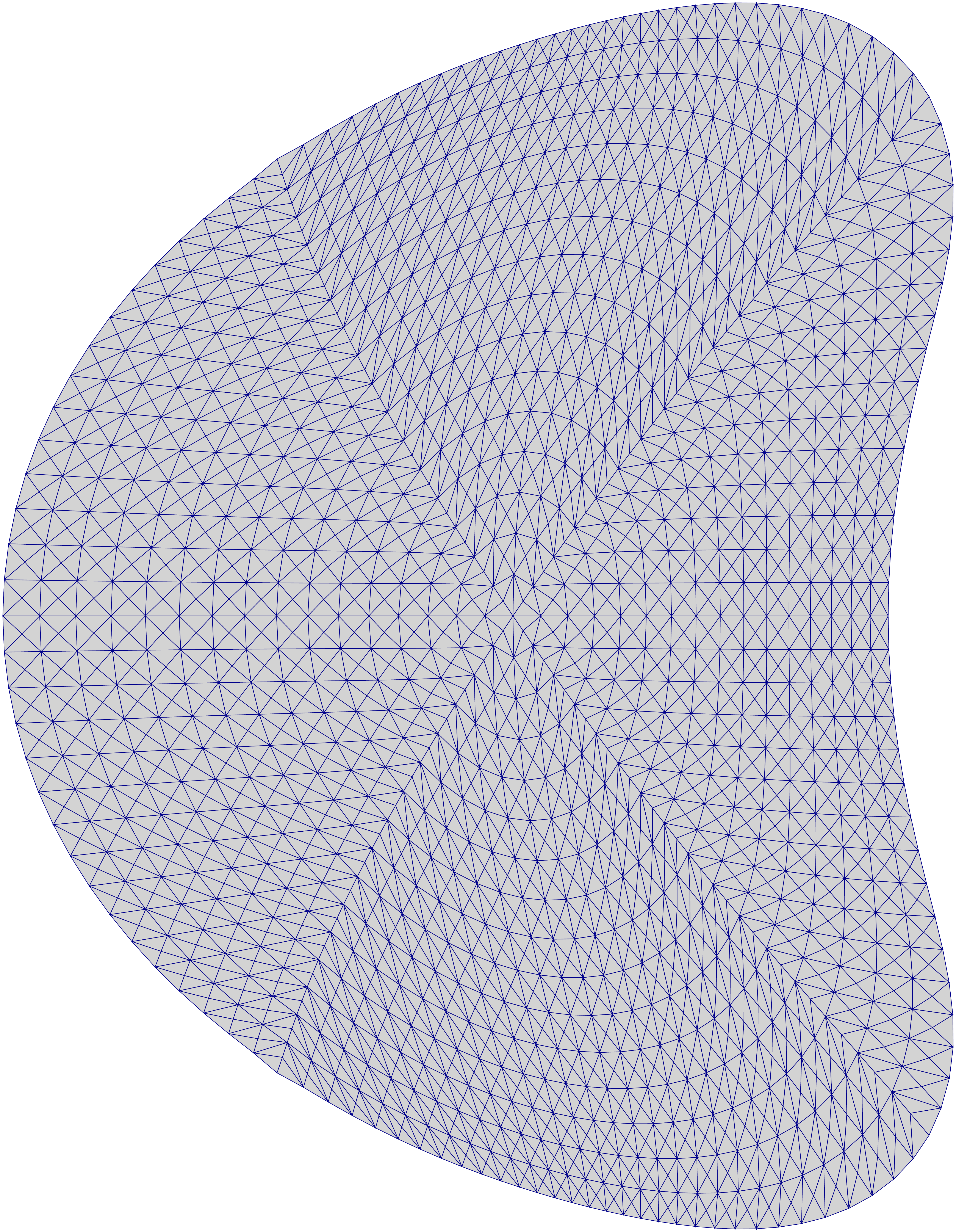}}
    \hspace{0.05\linewidth}
    \adjustbox{trim = {0\width} {0.5\height} {0\width} {-1\height}, clip ,width = .35\linewidth }{\includegraphics[width=.35\linewidth]{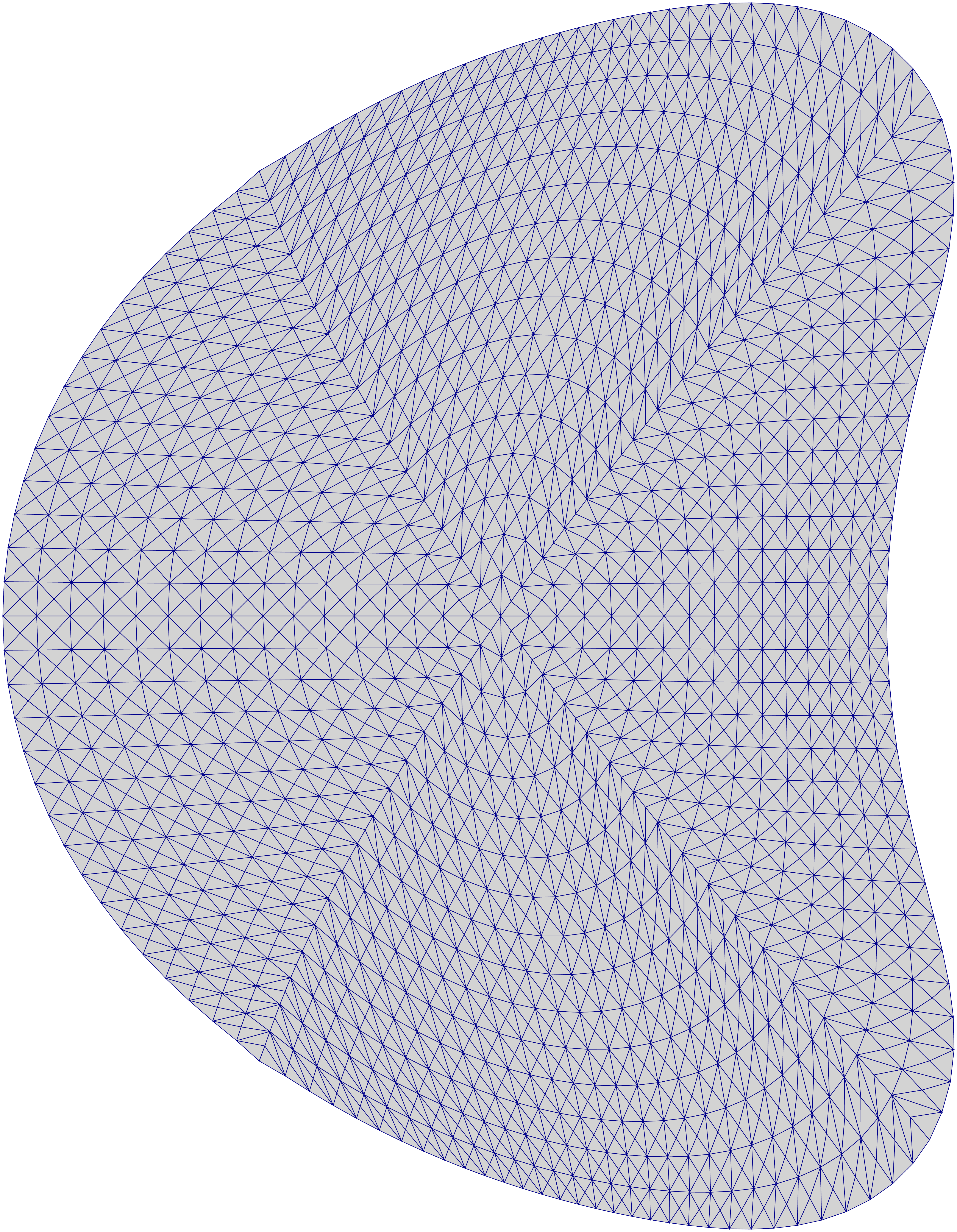}}
    \\
    \adjustbox{trim = {0\width} {0\height} {0\width} {.5\height}, clip ,width = .35\linewidth }{\includegraphics[width=.35\linewidth]{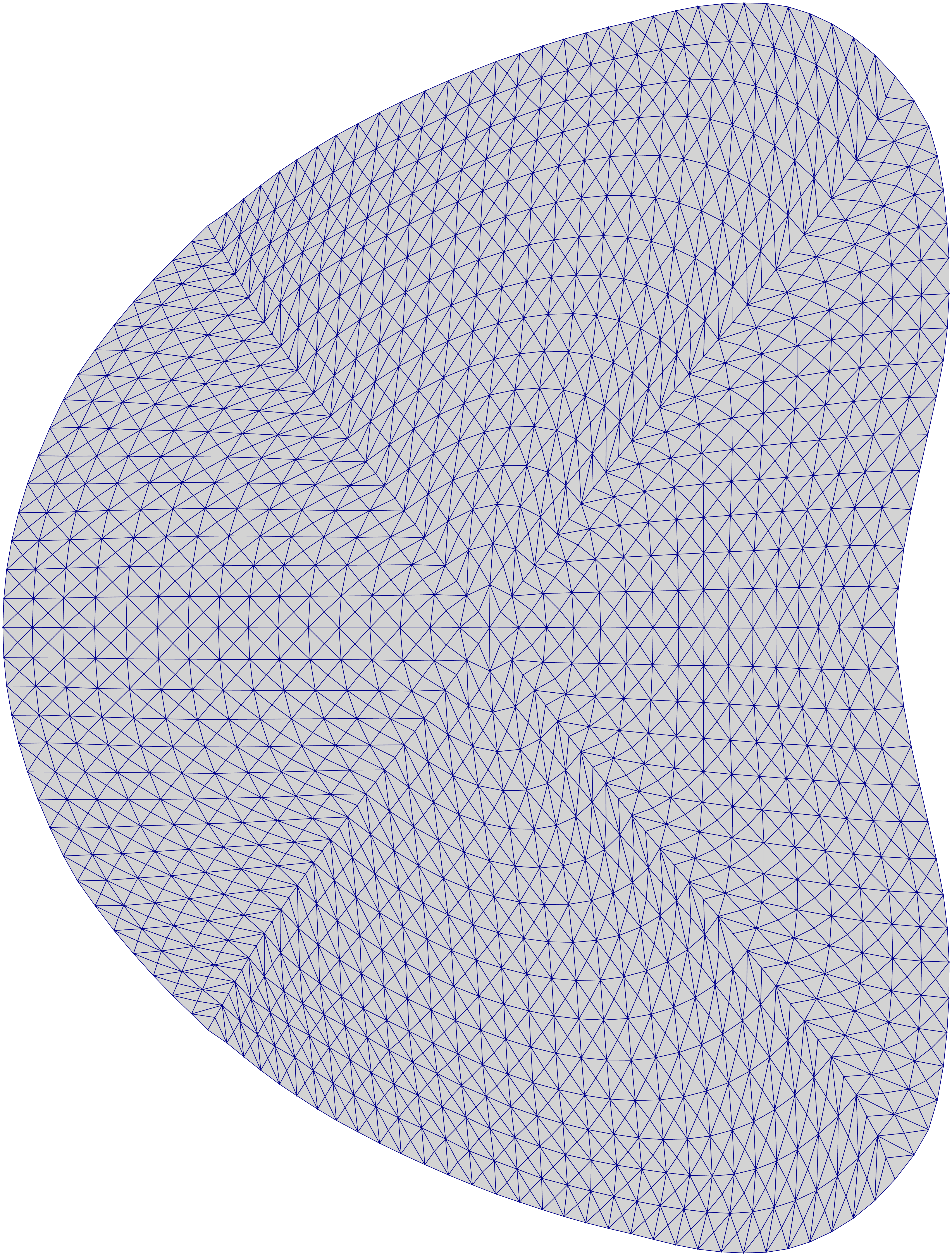}}
    \hspace{0.05\linewidth}
    \adjustbox{trim = {0\width} {0\height} {0\width} {.5\height}, clip ,width = .35\linewidth }{
    \includegraphics[width=.35\linewidth]{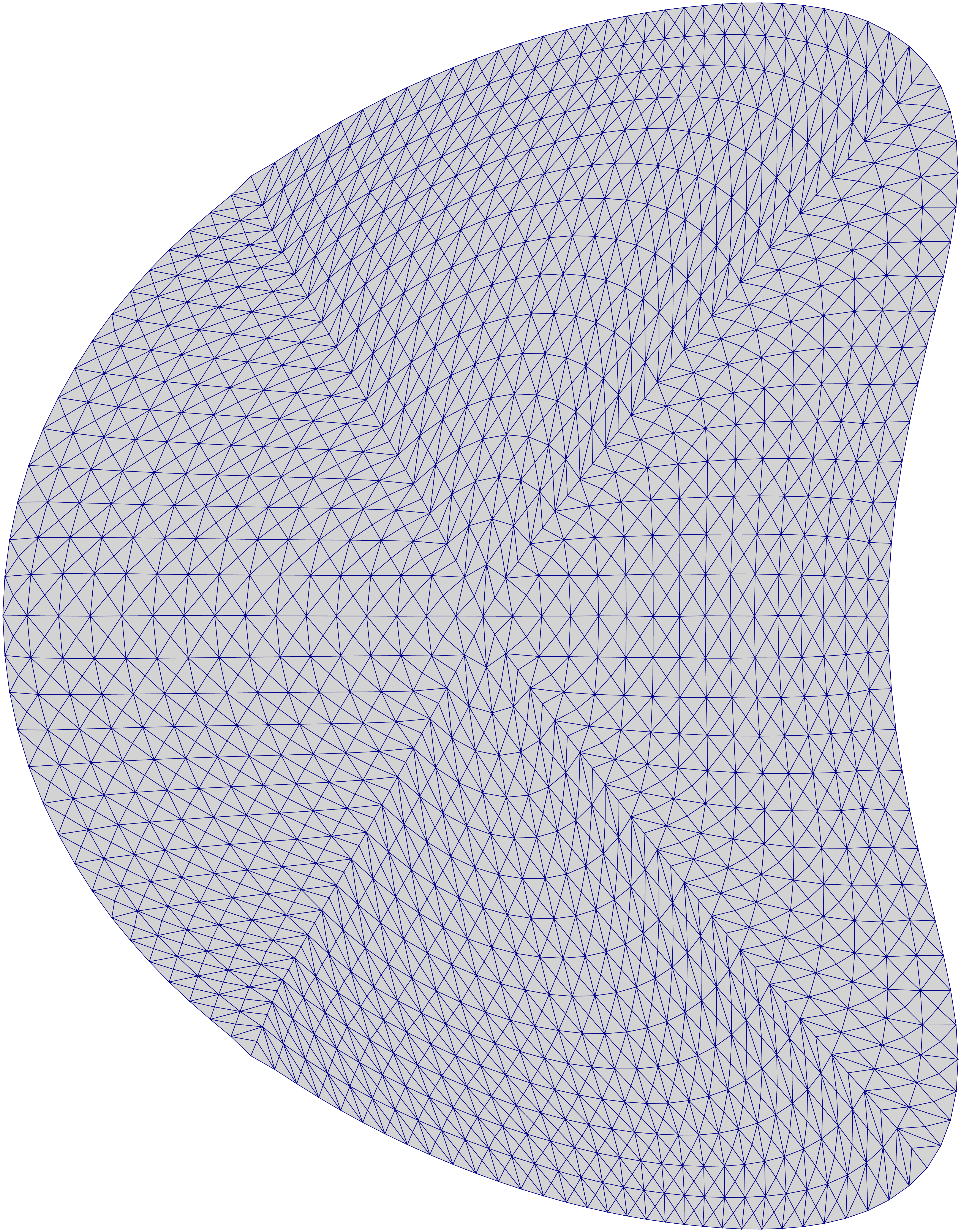}}
    \caption{The meshes of $\Omega$ for the final domains produced in the first \emph{Poisson} experiment in Section \ref{sec:experiment:Poisson:Experiment1}: top left to bottom, $p=2,4,\infty$ and second order.
    Due to symmetry of the result, we show only a half of each mesh.
    We see that all of the shapes are rather similar, with the $p=\infty$ being a slight outlier.
    The triangles for both $p=\infty$ and the Newton method appear to be more regularly spaced at the boundary.}
    \label{fig:experiment:Poisson:Experiment1:meshes}
\end{figure}

\subsubsection{Poisson experiment 2}\label{sec:experiment:Poisson:Experiment2}
For this experiment we consider $j(x,y) = \frac{1}{2}(y-y_d(x))^2$ where $y_d(x) = \frac{4}{\pi} - |x|^2$ and $F = 1$.
Let us note that $-\Delta y_d = 4 F$.
For the Newton direction we take $t= 0.125$.
This experiment will be equipped with an area constraint that the domain has fixed area $4$ - we will use the same linear constraint on the update direction and projection as in Section \ref{sec:experiment:NOPDE:Experiment2}.
In this setting, we expect the minimiser to be given by the ball of radius $\frac{2}{\sqrt{\pi}}$ at the origin which has energy $\frac{6}{\pi^2}$.

We start with the square $(-1,1)^2$.
The triangulation of the domain and hold-all is displayed in Figure \ref{fig:experiment:Poisson:Experiment2:InitialDomain}.
In Figure \ref{fig:experiment:Poisson:Experiment2:graphs}, the energy of shapes along the minimising sequences we produce are given.
\begin{figure}\centering
    \begin{subfigure}[b]{.49\linewidth}
    \centering
    \includegraphics[width=.8\linewidth]{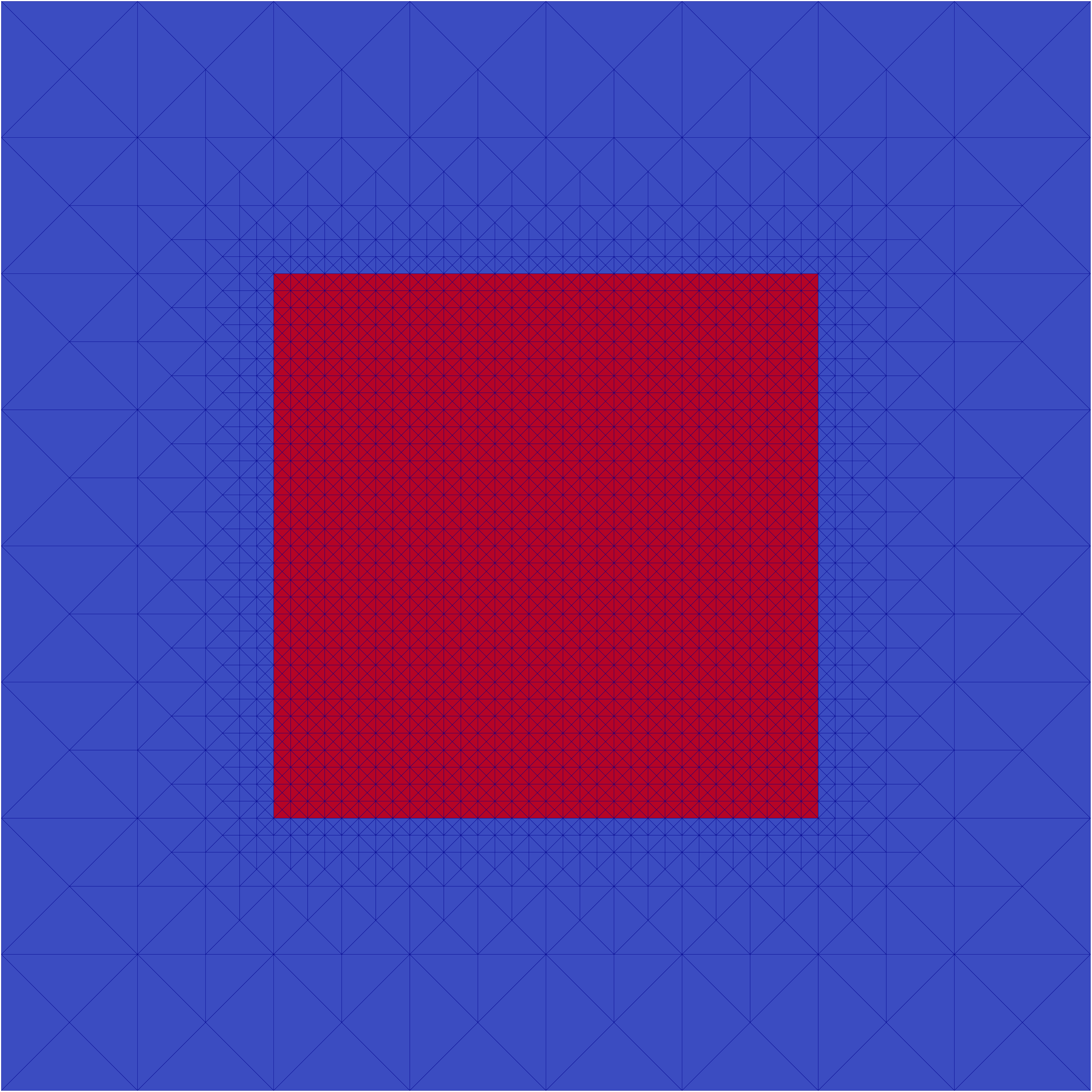}
    \caption{Initial domain for the second \emph{Poisson} experiment in Section \ref{sec:experiment:Poisson:Experiment2}, $(-1,1)^2$ is in red and the hold-all, $(-2,2)^2$ in blue.}
    \label{fig:experiment:Poisson:Experiment2:InitialDomain}
    \end{subfigure}\hfill
    \begin{subfigure}[b]{.49\linewidth}
    \centering
    \includegraphics[width=1\linewidth]{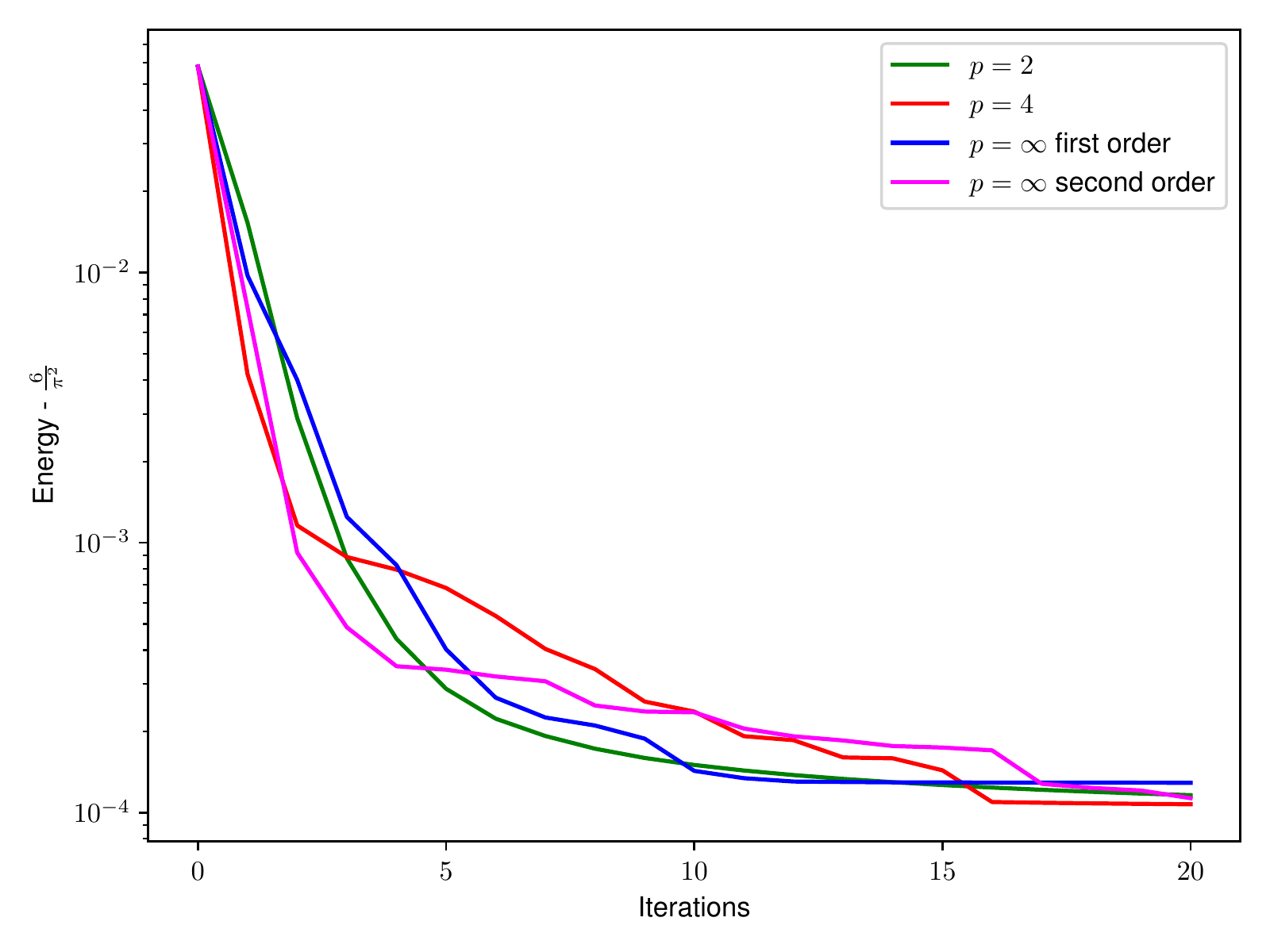}
    \caption{Graph of the energy for the iterates in the second Poisson experiment in Section \ref{sec:experiment:Poisson:Experiment2}.
    We see that all the 
    methods give roughly the same energy.}
    \label{fig:experiment:Poisson:Experiment2:graphs}
    \end{subfigure}
    \caption{Initial mesh and graph of the energy for the experiment in Secetion \ref{sec:experiment:Poisson:Experiment2}.} 
\end{figure}
In Figure \ref{fig:experiment:Poisson:Experiment2:meshes}, the meshes for the final domains $\Omega$ for each of the methods are given.
\begin{figure}
    \vspace{-.5\linewidth} 
    \centering
    \adjustbox{trim = {0\width} {0.5\height} {0.5\width} {-1\height}, clip ,width = .25\linewidth }{\includegraphics[width=.35\linewidth]{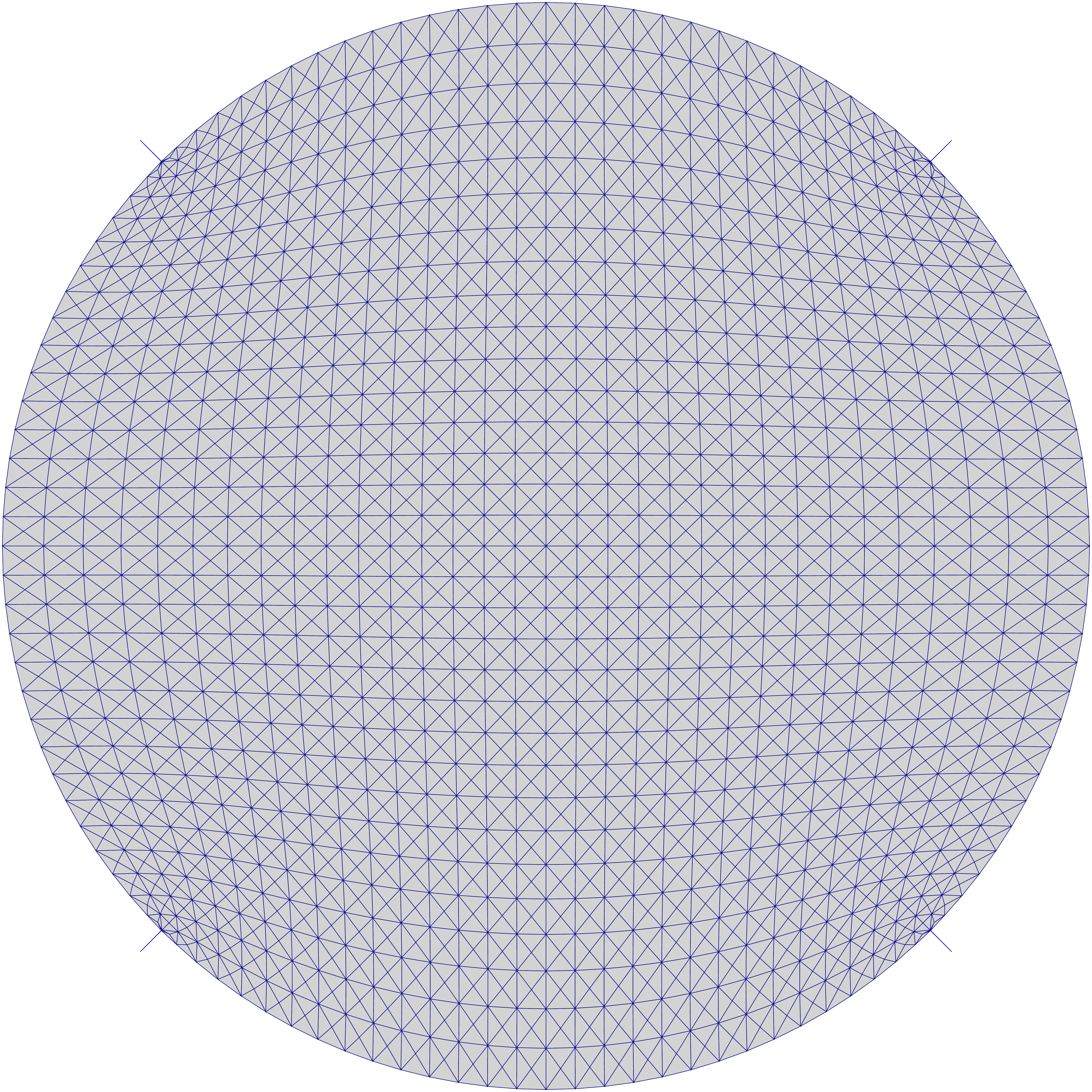}}
    \adjustbox{trim = {0.5\width} {0.5\height} {0\width} {-1\height}, clip ,width = .25\linewidth }{\includegraphics[width=.35\linewidth]{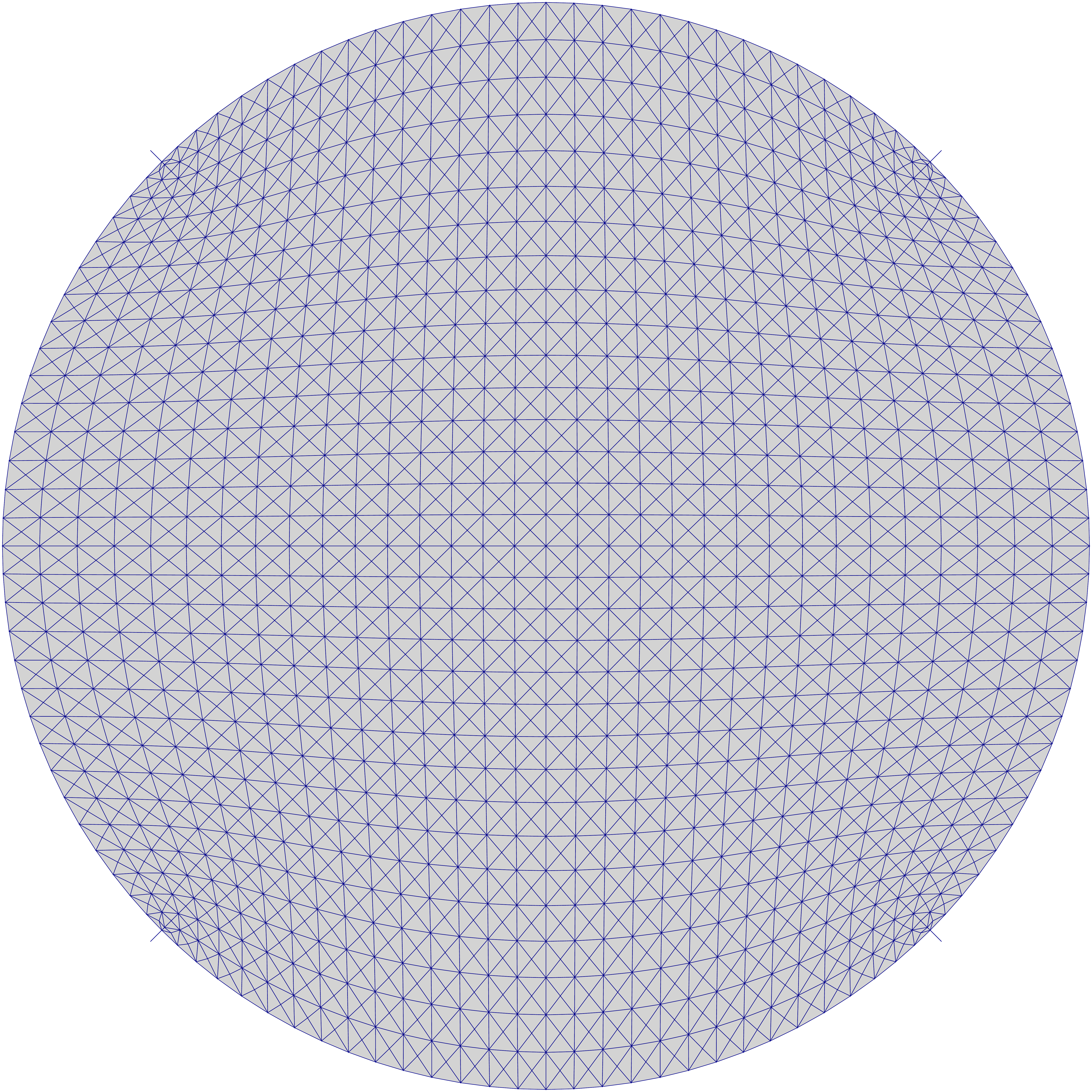}}
    
    \adjustbox{trim = {0\width} {0\height} {0.5\width} {.5\height}, clip ,width = .25\linewidth }{\includegraphics[width=.35\linewidth]{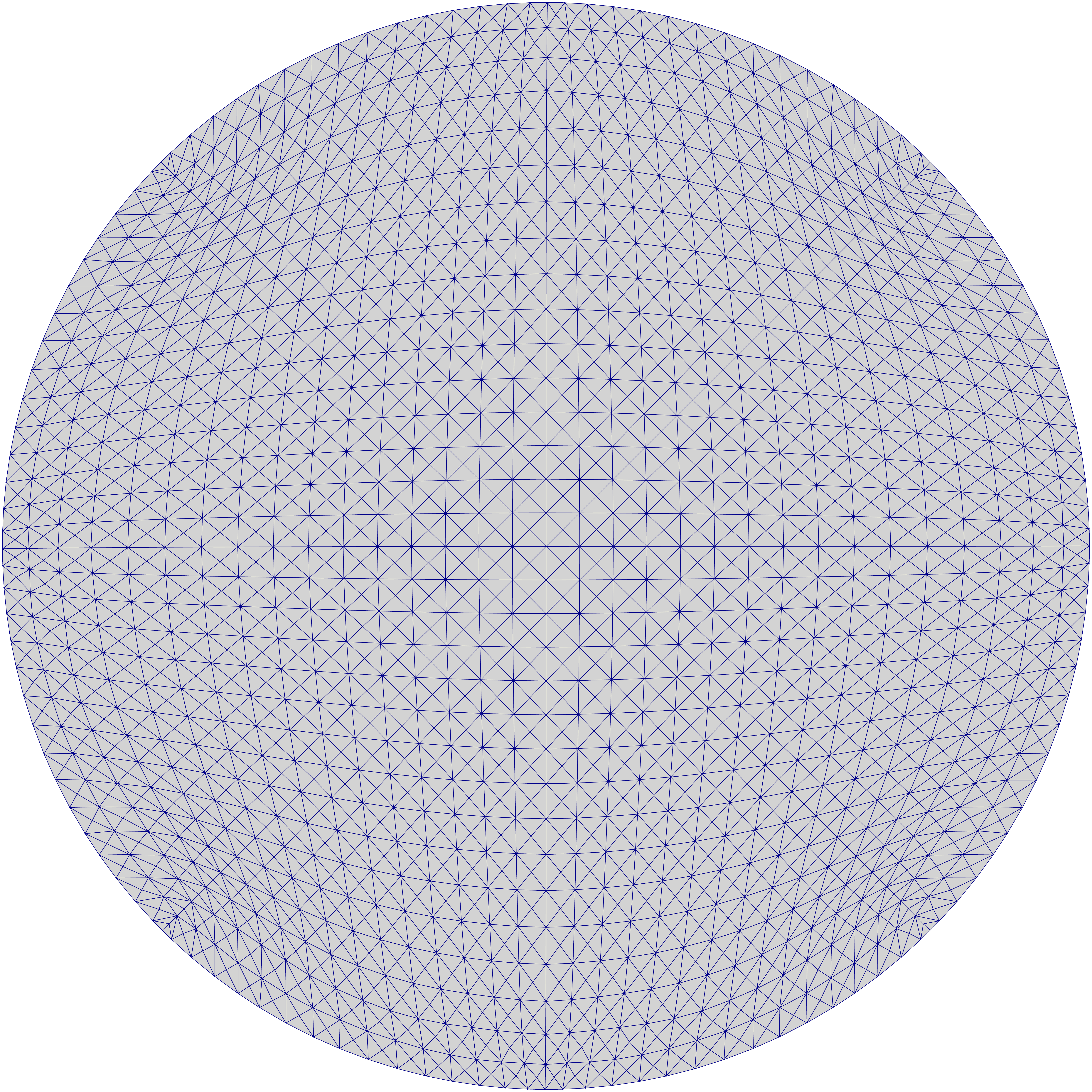}}
    \adjustbox{trim = {0.5\width} {0\height} {0\width} {.5\height}, clip ,width = .25\linewidth }{\includegraphics[width=.35\linewidth]{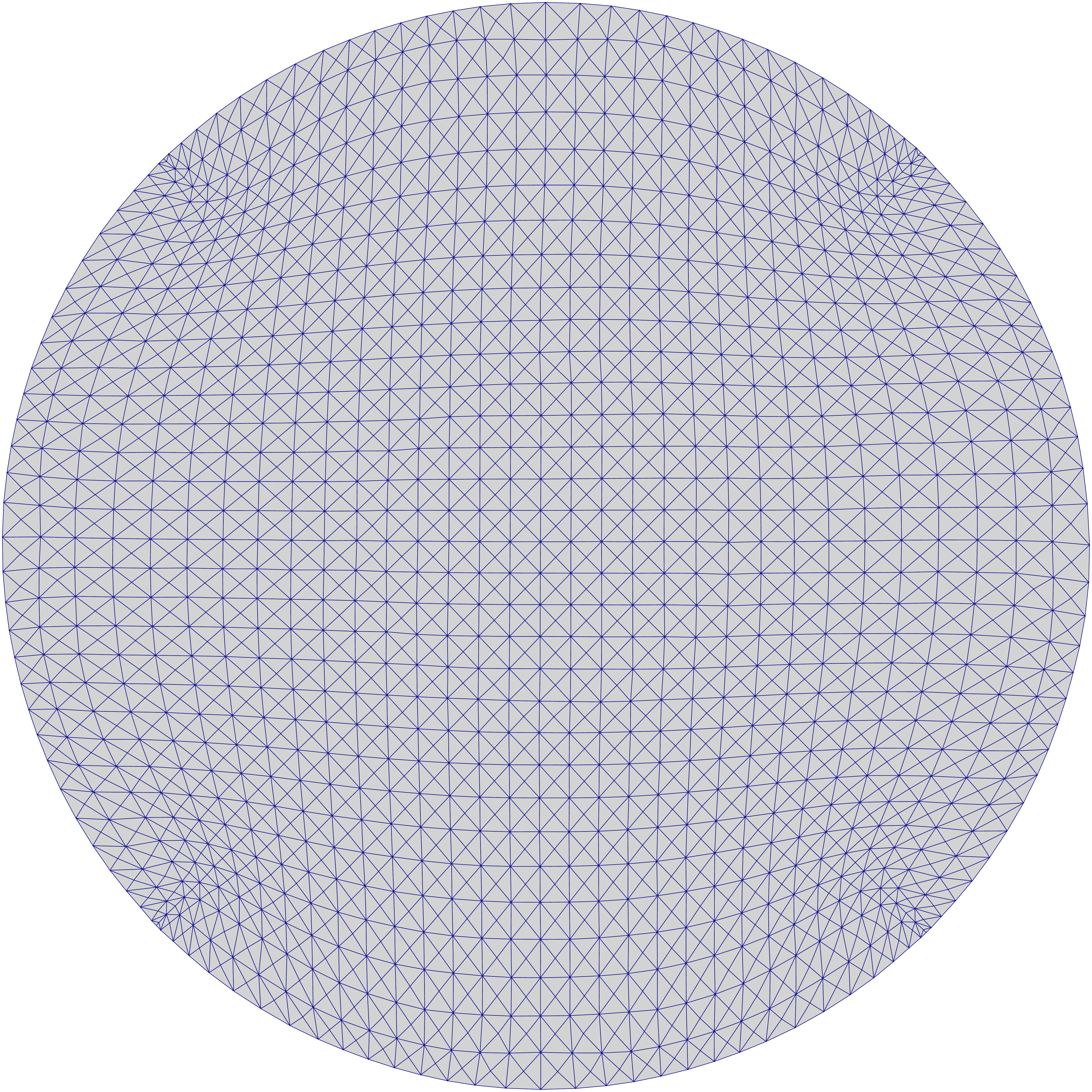}}
    \caption{The meshes of $\Omega$ for the final domains produced in the second \emph{Poisson} experiment in Section \ref{sec:experiment:Poisson:Experiment2}: top left to bottom, $p=2,4,\infty$ and second order.
    Due to symmetry of the result, we show only a quarter of each mesh.
    We see that both $p=2$ and $p=4$ have very degenerate elements where the corners of the original domain were.
    Both $p=\infty$ and Newton methods do not have these degenerate triangles.
    We notice that the triangles which previously made up the corners of the original domain are rather regular in the Newton method.
    For $p=\infty$, there are many triangles which have obtuse angles.
    }
    \label{fig:experiment:Poisson:Experiment2:meshes}
\end{figure}

It is worth noting that when larger values of $t$ were taken during testing, the Newton-type method struggled to perform well.
With the Newton method, once the shape was sufficiently close to a ball, the directions generated by ADMM would rotate the almost-ball by large angles which caused large deformations of the mesh in the hold-all.

\subsection{A coupled Poisson problem}\label{sec:application:coupledPoisson}

For this experiment we will consider $j(x,y) = \frac{1}{2} (y_1-y_d(x) )^2$ where $y_d(x) = 0.05 + (1-x_1^2)^3 (1-x_2^2)^3$ and
$F(x) = \Delta^2 \left( (1-x_1^2)^3 (1-x_2^2)^3 \right)$.
For the Newton direction we take $t= 0.0625$.
This experiment will be equipped with an area constraint that the domain has fixed area $4$ - we will use the same linear constraint on the update direction and projection as in Section \ref{sec:experiment:NOPDE:Experiment2}.
One would expect the minimiser to be relatively close to the square $(-1,1)^2$ which should have energy $0.005$.

We start with the square $(-1,1)^2$.
The triangulation of the domain and hold-all is displayed in Figure \ref{fig:experiment:coupledPoisson:Experiment1:InitialDomain}.
In Figure \ref{fig:experiment:coupledPoisson:Experiment1:graphs}, the energy of shapes along the minimising sequences we produce are given.
\begin{figure}\centering
    \begin{subfigure}[b]{.49\linewidth}
    \centering
    \includegraphics[width=.8\linewidth]{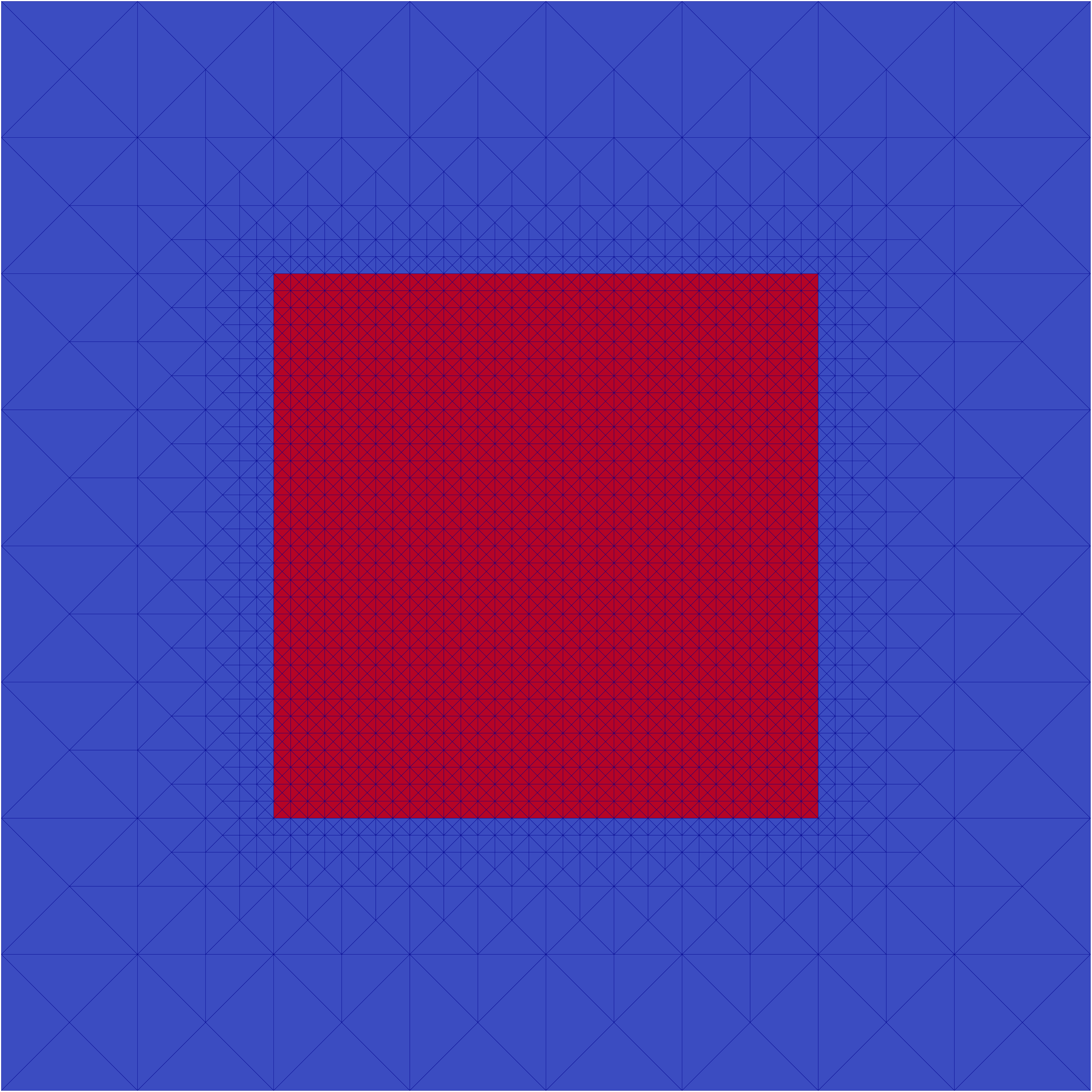}
    \caption{Initial domain for the \emph{coupled Poisson} experiment in Section \ref{sec:application:coupledPoisson}, $(-1,1)^2$ is in red and the hold-all, $(-2,2)^2$ in blue.}
    \label{fig:experiment:coupledPoisson:Experiment1:InitialDomain}
    \end{subfigure}\hfill
    \begin{subfigure}[b]{.49\linewidth}
    \centering
    \includegraphics[width=1\linewidth]{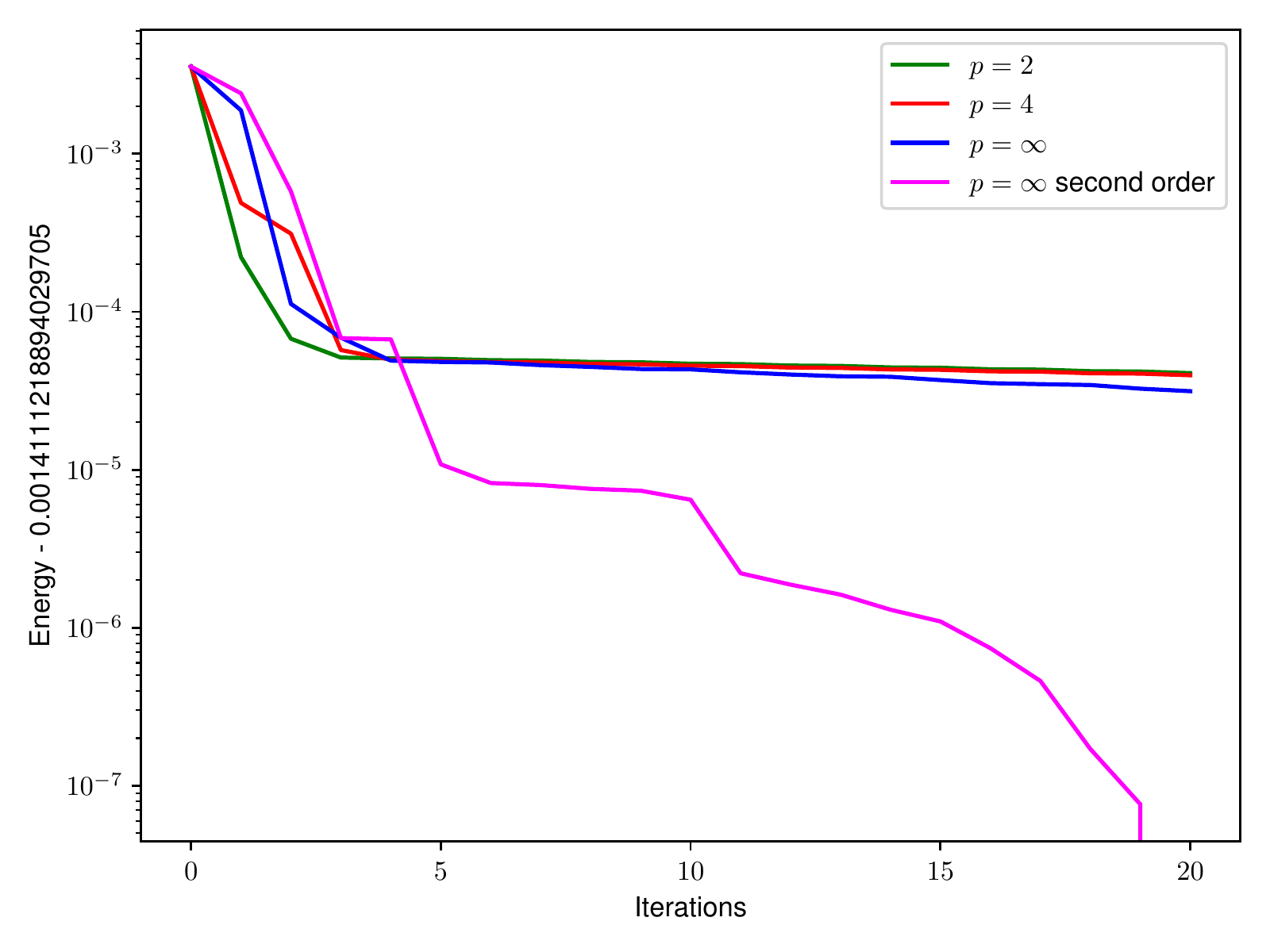}
    \caption{Graph of the energy for the iterates in the \emph{coupled Poisson} experiment in Section \ref{sec:application:coupledPoisson}.
    We see that $p=\infty$ outperforms the finite $p$ experiments, but the Newton method is energetically performing the best.
    }
    \label{fig:experiment:coupledPoisson:Experiment1:graphs}
    \end{subfigure}
    \caption{Initial mesh and graph of the energy for the experiment in Section \ref{sec:application:coupledPoisson}.}
\end{figure}
In Figure \ref{fig:experiment:coupledPoisson:Experiment1:meshes}, the meshes for the final domains $\Omega$ for each of the methods are given.
\begin{figure}
    \vspace{-.5\linewidth} 
    \centering
    \adjustbox{trim = {0\width} {0.5\height} {0.5\width} {-1\height}, clip ,width = .25\linewidth }{\includegraphics[width=.35\linewidth]{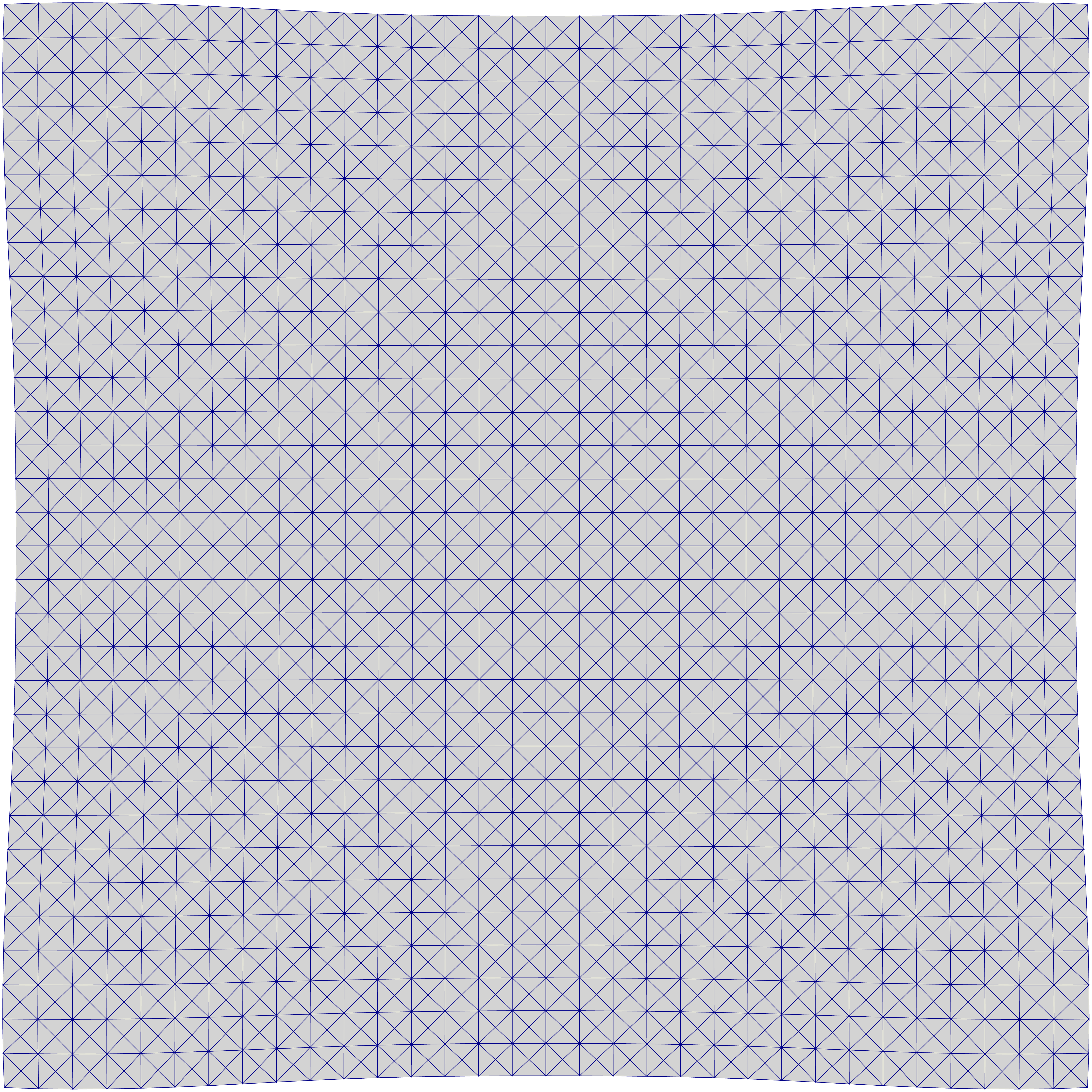}}
    \adjustbox{trim = {0.5\width} {0.5\height} {0\width} {-1\height}, clip ,width = .25\linewidth }{\includegraphics[width=.35\linewidth]{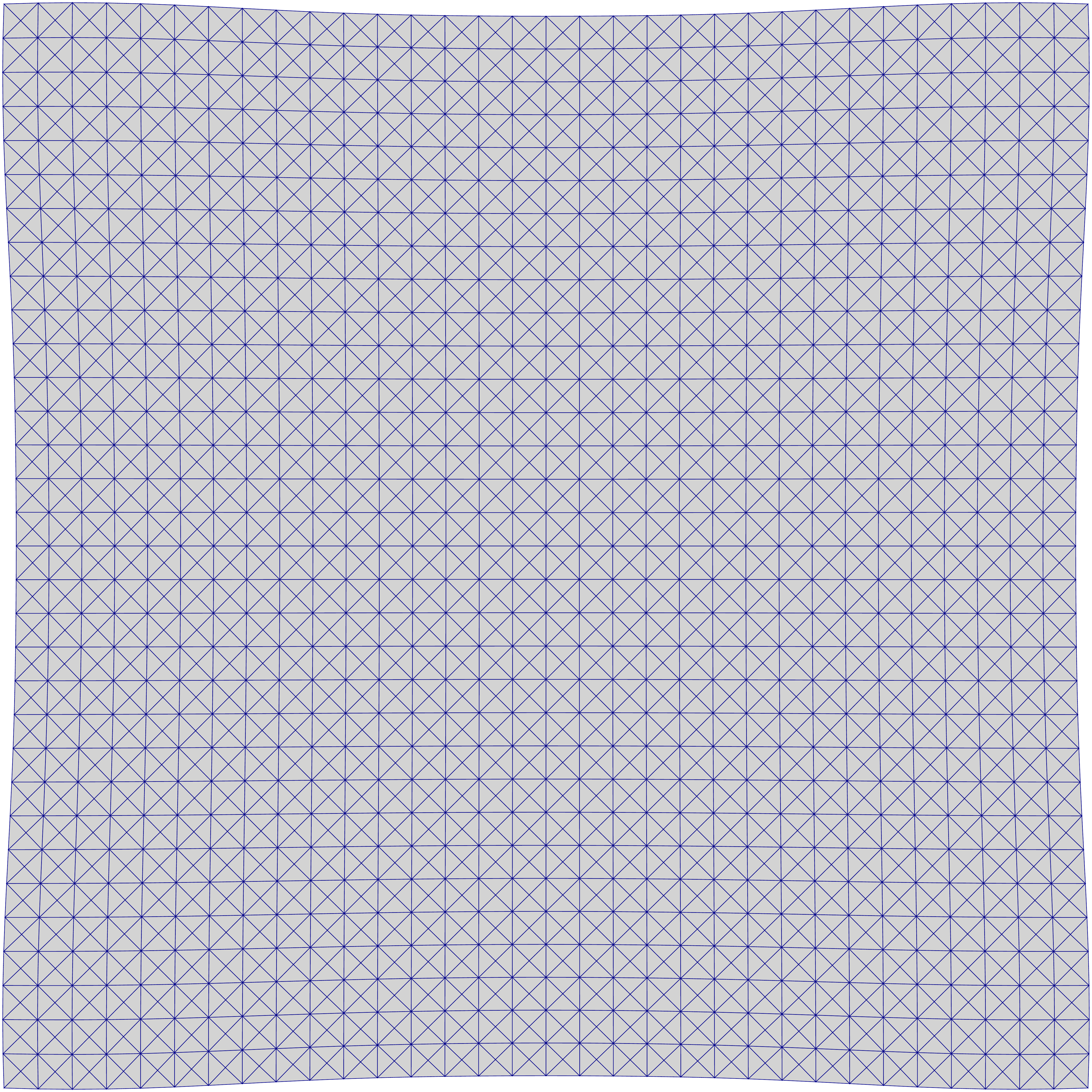}}
    \\
    \adjustbox{trim = {0\width} {0\height} {0.5\width} {.5\height}, clip ,width = .25\linewidth }{\includegraphics[width=.35\linewidth]{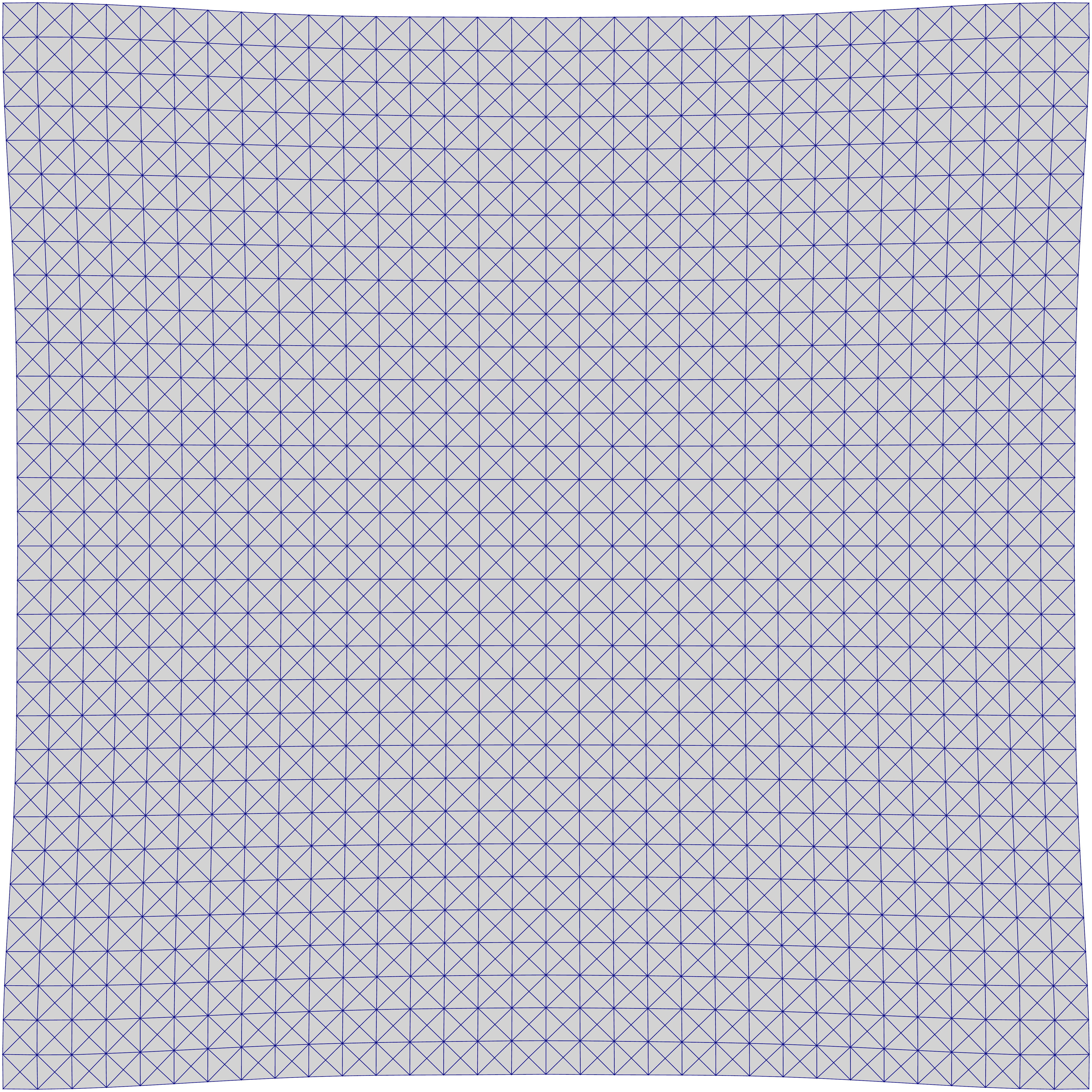}}
    \adjustbox{trim = {0.5\width} {0\height} {0\width} {.5\height}, clip ,width = .25\linewidth }{
    \includegraphics[width=.35\linewidth]{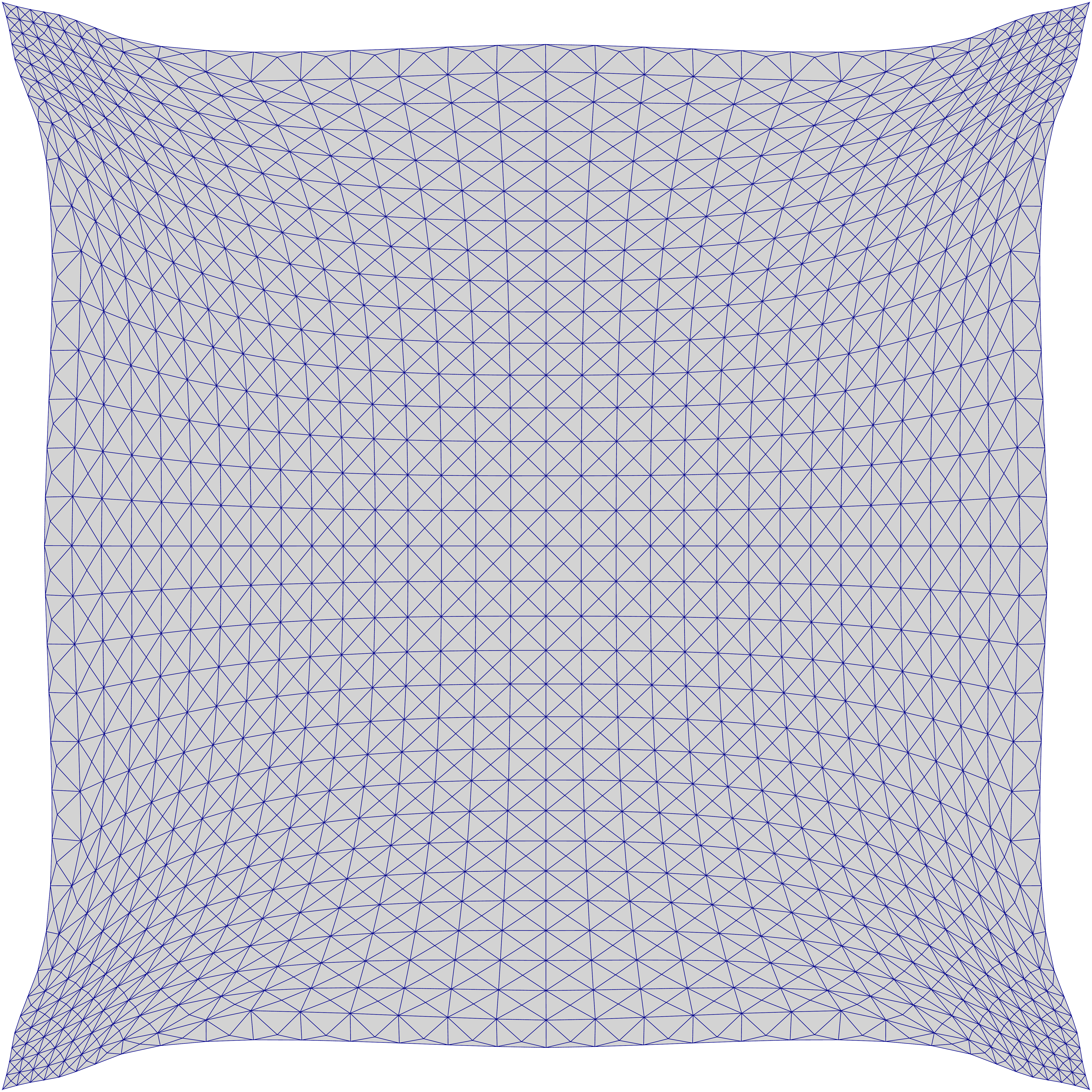}}
  \caption{The meshes of $\Omega$ for the final domains produced in the \emph{coupled Poisson} experiment in Section
    \ref{sec:application:coupledPoisson}: top left to bottom, $p=2,4,\infty$ and second order.
    Due to symmetry of the result, we show only a quarter of each mesh.
    To the eye, the first order methods are all seemingly the same.
    The Newton method appears to have found a more pronounced shape than the others.
    }
    \label{fig:experiment:coupledPoisson:Experiment1:meshes}
\end{figure}

\subsection{Optimisation of the first eigenvalue for the Laplacian}\label{sec:experiment:eval:Experiment1}

We use the function \texttt{eigs} from the module \texttt{sparse.linalg} in \texttt{scipy} \cite{scipy20} to find pairs $(v_h,\lambda_h)\in \R^{N_h}\times \R$ such that $B_h v_h = \lambda_h M_h v_h$, where $B_h\in \R^{N_h\times N_h}$ is the stiffness matrix and $M_h\in \R^{N_h\times N_h}$ is the mass matrix.
This experiment will be equipped with an area constraint that the domain has fixed area $4$ - we will use the same linear constraint on the update direction and projection as in Section \ref{sec:experiment:NOPDE:Experiment2}.
For the Newton direction we take $t= 0.125$.
In this setting, the minimiser is known to be the ball of radius $\frac{2}{\sqrt{\pi}}$ which has $\lambda_1(B_{\frac{2}{\sqrt{\pi}}}) = \beta_{0,1}^2 \frac{\pi}{4}\approx 4.54210$, where $\beta_{0,1}$ is the first zero of the $0^{th}$ Bessel function.

We start with the square $(-1,1)^2$.
The triangulation of the domain and hold-all is displayed in Figure \ref{fig:experiment:eval:Experiment1:InitialDomain}.
In Figure \ref{fig:experiment:eval:Experiment1:graphs}, the energy of shapes along the minimising sequences we produce are given.
\begin{figure}\centering
    \begin{subfigure}[b]{.49\linewidth}
    \centering
    \includegraphics[width=.8\linewidth]{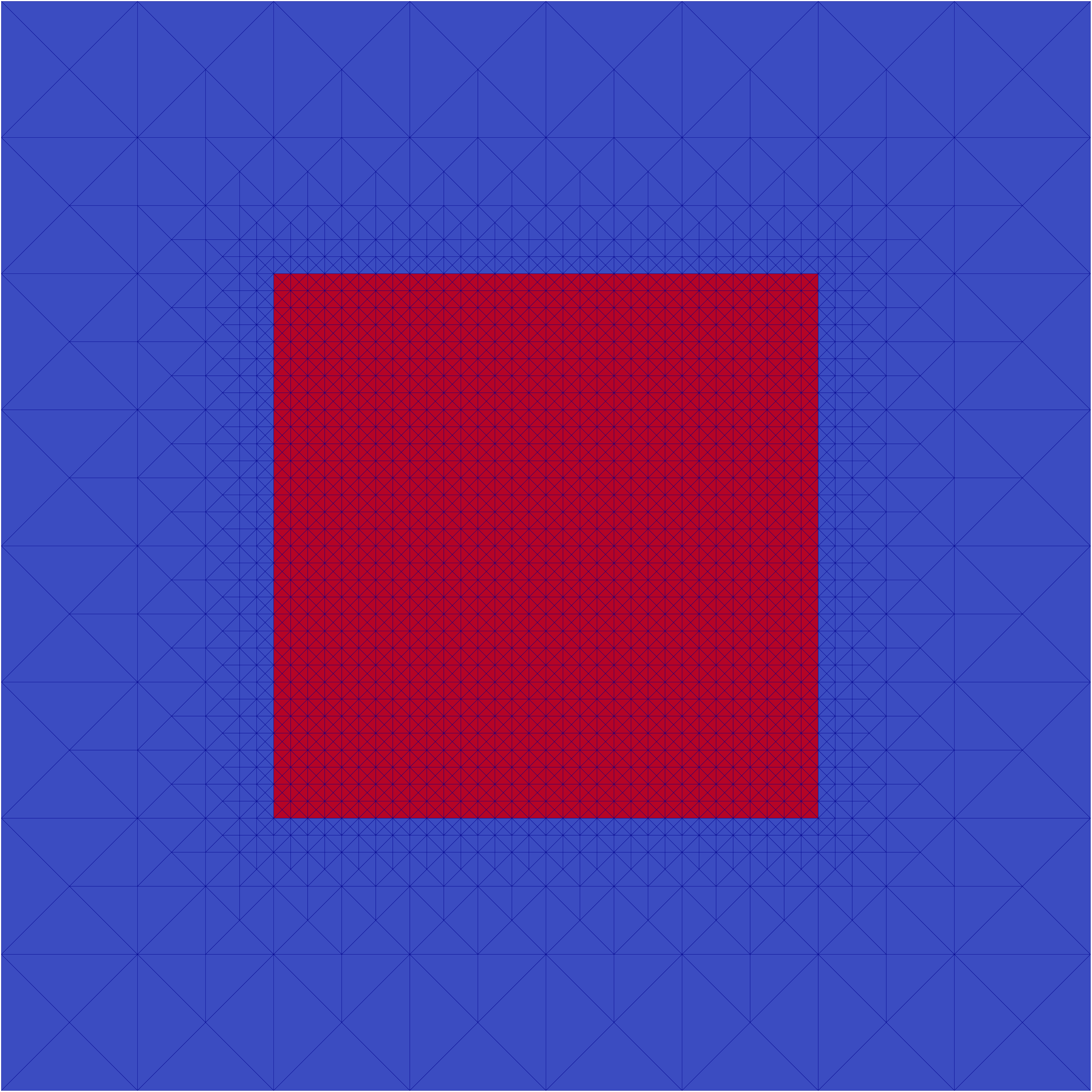}
    \caption{Initial domain for the \emph{Eigenvalue} experiment in Section \ref{sec:experiment:eval:Experiment1}, $(-1,1)^2$ is in red and the hold-all, $(-2,2)^2$ in blue.}
    \label{fig:experiment:eval:Experiment1:InitialDomain}
    \end{subfigure}\hfill
    \begin{subfigure}[b]{.49\linewidth}
    \centering
    \includegraphics[width=1\linewidth]{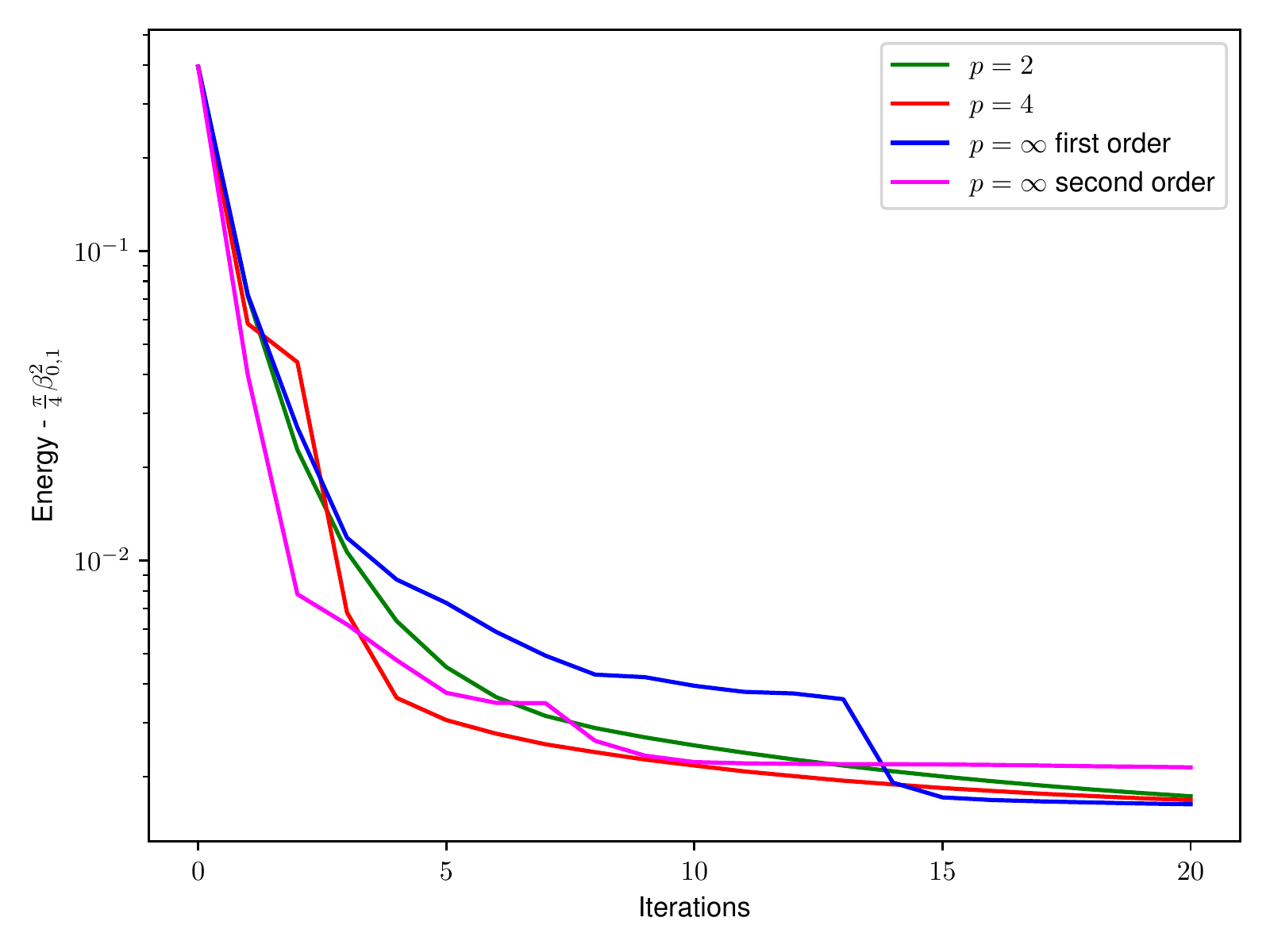}
    \caption{Graph of the energy for the iterates in the \emph{Eigenvalue} experiment in Section \ref{sec:experiment:eval:Experiment1}.
    We see that all the methods are performing roughly the same.}
    \label{fig:experiment:eval:Experiment1:graphs}
    \end{subfigure}
    \caption{Initial mesh and graph of the energy for the experiment in Section \ref{sec:experiment:eval:Experiment1}.}
\end{figure}
In Figure \ref{fig:experiment:eval:Experiment1:meshes}, the meshes for the final domains $\Omega$ for each of the methods are given.
\begin{figure}
    \vspace{-.5\linewidth} 
    \centering
    \adjustbox{trim = {0\width} {0.5\height} {0.5\width} {-1\height}, clip ,width = .25\linewidth }{
    \includegraphics[width=.35\linewidth]{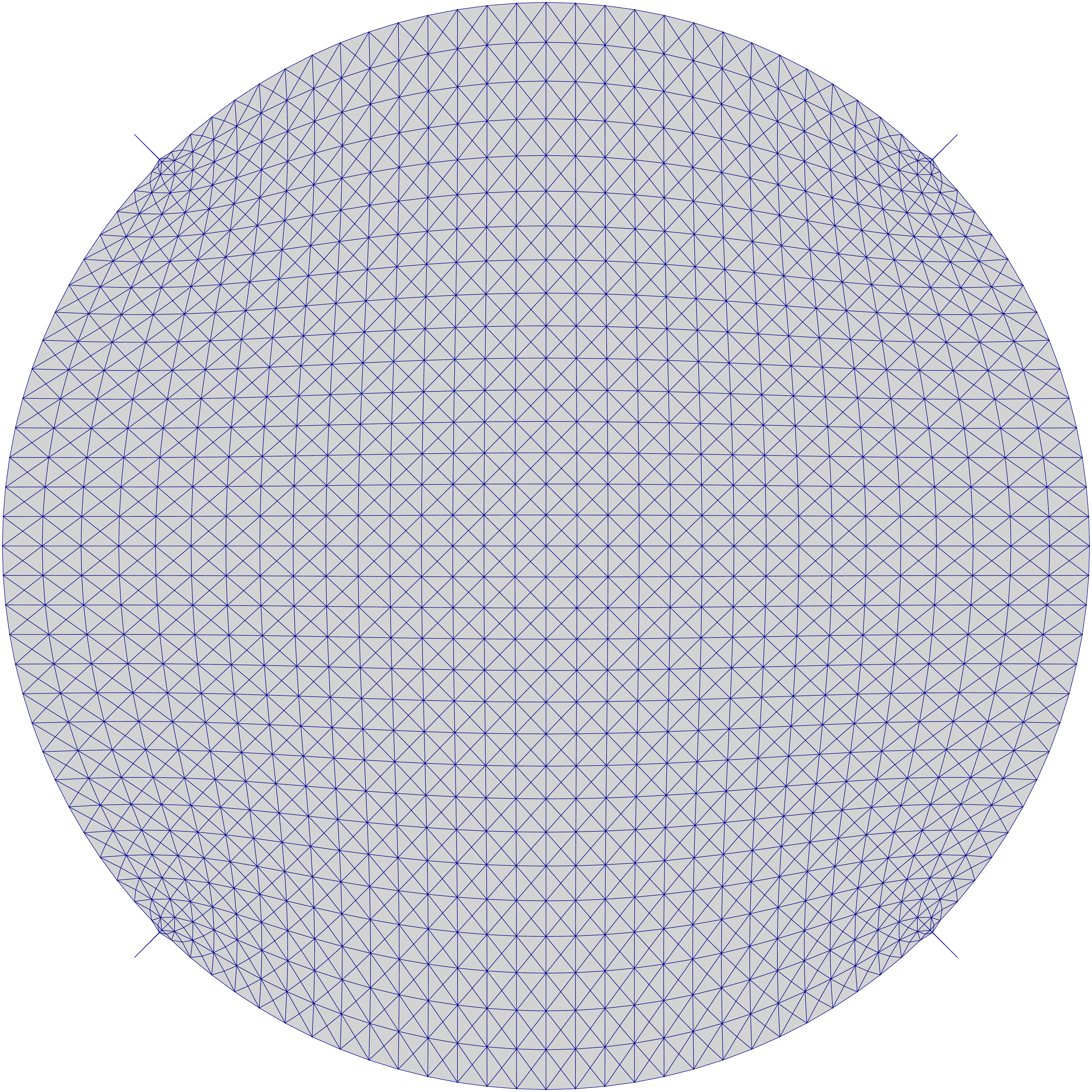}}
    \adjustbox{trim = {0.5\width} {0.5\height} {0\width} {-1\height}, clip ,width = .25\linewidth }{
    \includegraphics[width=.35\linewidth]{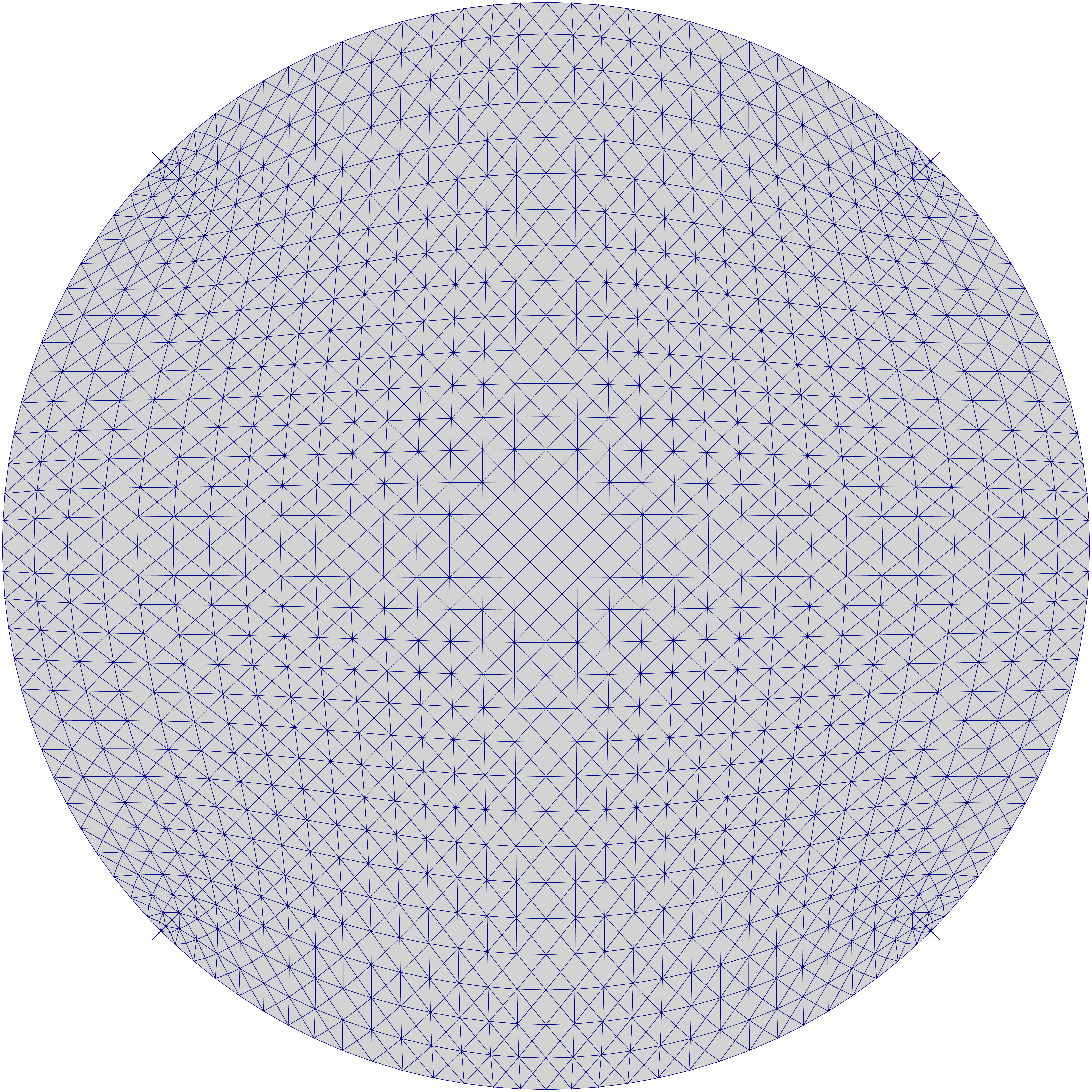}}
    \\
    \adjustbox{trim = {0\width} {0\height} {0.5\width} {.5\height}, clip ,width = .25\linewidth }{
    \includegraphics[width=.35\linewidth]{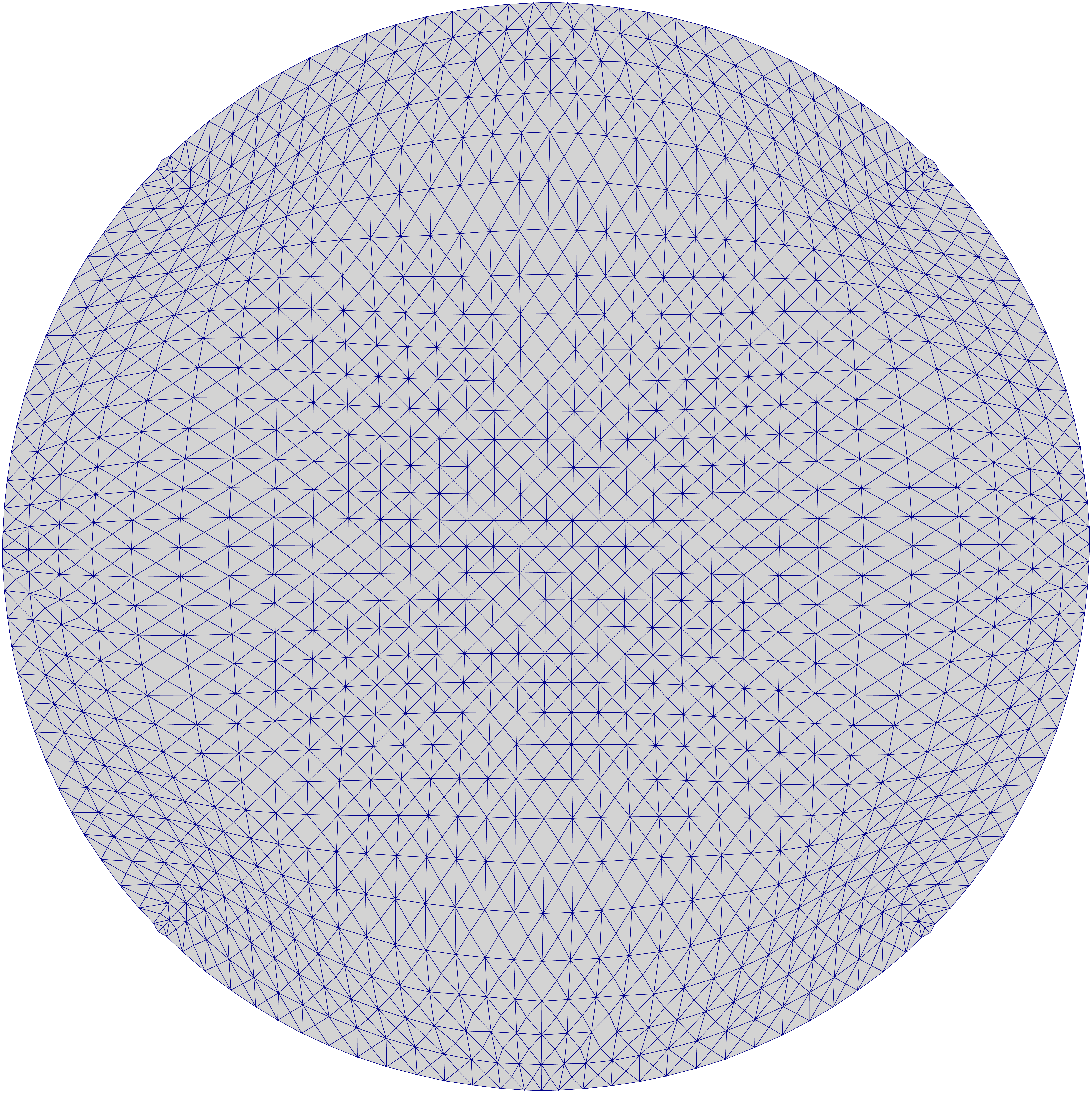}}
    \adjustbox{trim = {0.5\width} {0\height} {0\width} {.5\height}, clip ,width = .25\linewidth }{
    \includegraphics[width=.35\linewidth]{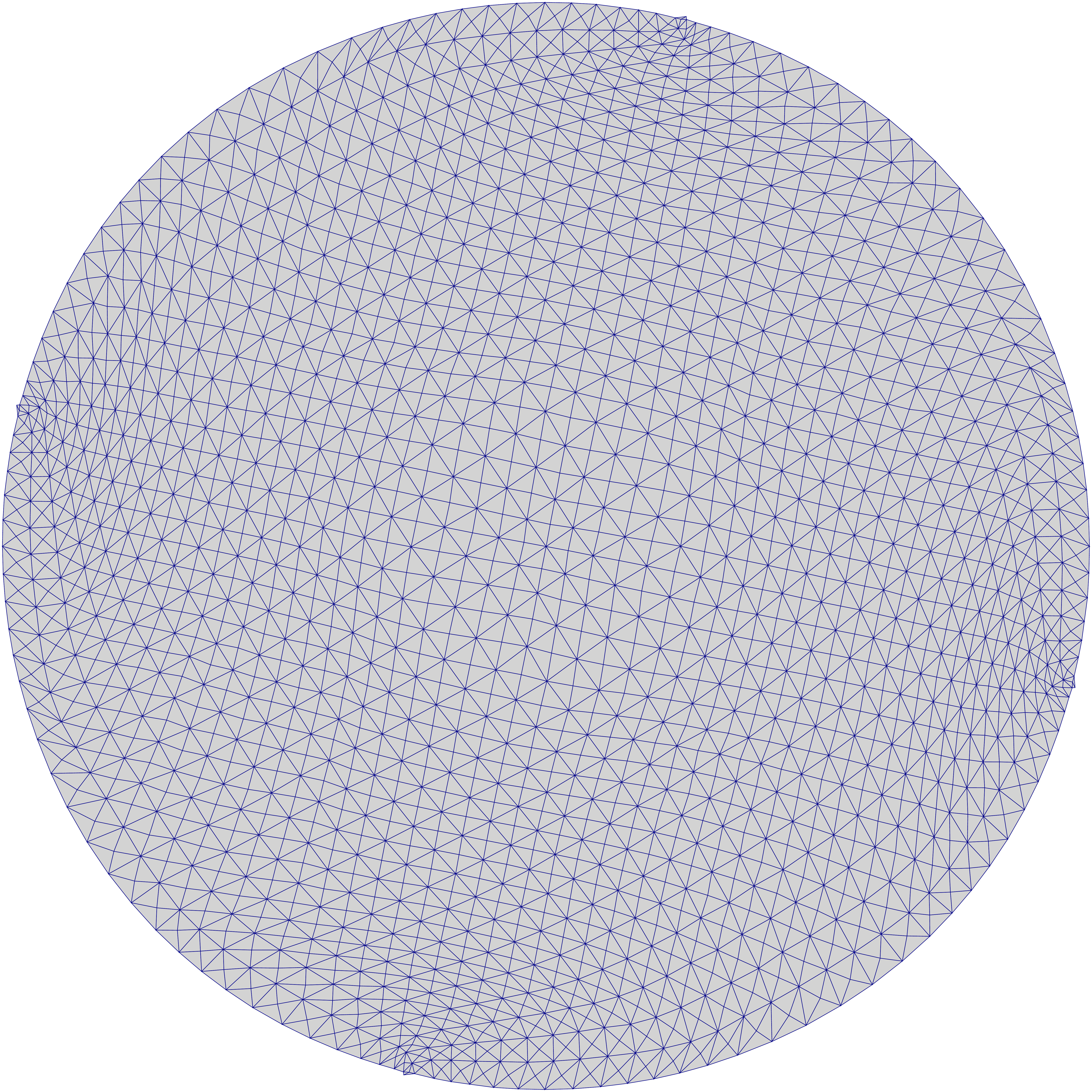}}
    \caption{The meshes of $\Omega$ for the final domains produced in the \emph{Eigenvalue} experiment in Section \ref{sec:experiment:eval:Experiment1}: top left to bottom, $p=2,4,\infty$ and second order.
    Due to symmetry of the result, we show only a quarter of each mesh.
    We see that spikes appear for $p=2$ and $p=4$ where the corners were for the original shape.
    The mesh for $p=\infty$ appears rather regular.
    We comment that the Newton method has lead to the mesh rotating when compared to its original orientation.}
    \label{fig:experiment:eval:Experiment1:meshes}
\end{figure}
Let us further elaborate on Figure \ref{fig:experiment:eval:Experiment1:meshes}, particularly the mesh produced by the infinity method.
We note that the mesh in the centre for the infinity method appears less regular than the other methods.
This appears to occur due to a somewhat undesirable approximation of the shape derivative, which should be concentrated at the boundary.
The undesirable approximation is due to the FEM approximation giving interior contributions to the shape derivatives, so-called \emph{spurious} contributions, in the limit of the mesh becoming infinitely fine this should disappear.
Investigation suggests that these are more apparent in the infinity method but disappear when using shape derivatives with only boundary contributions.
This sort of expert knowledge is often exploited in the literature and appeared in e.g. \cite{SSW16} to 'delete' contributions to the shape derivative from nodes on the interior.

\section*{Conclusion}
In this work we have introduced two new frameworks in which to perform shape optimisation.
These methods are more applicable to physical scenarios than the previous works involving $W^{1,\infty}$ as they are not restricted to star-shaped domains.
While our examples had star-shaped final domains, our framework allows for more general shapes and is more readily applicable to industrial problems.
From the experiments, it was seen that our introduced first order method did not necessarily perform well energetically, however the meshes the method produced are regular.
The second order method we introduced performs well both energetically and in terms of the regularity of the mesh, a downside however is the need to tune the damping parameter $t$.

\section*{Acknowledgements}
The authors wish to extend their gratitude to Andreas Dedner for providing useful insight into the use of the Python bindings for DUNE.
This work is part of the project P8 of the German Research Foundation Priority Programme 1962, whose support is gratefully acknowledged by the second and the third author.
P.J.H acknowledges the support of EPSRC (grant EP/W005840/1).

\appendix

\section{Calculations for the second shape derivatives}\label{Appendix}
Here we collect the derivatives of $J$ and $e$ for the examples in Section \ref{sec:Applications}.
These derivatives are particularly useful for the calculation of the second shape derivative.
All of the maps and functions we consider are sufficiently smooth that we may exchange the order of differentiation.
It will be convenient to define
\begin{equation*}
    \detDashDash[V,W]:= \left(\Div(V)\Div(W) - \Tr(DVDW)\right).
\end{equation*}
\subsection{Derivatives for the energy functionals}\label{app:JDerivatives}
Let us consider $J$ as in \eqref{eq:CostFunctionalJ}, that is $J(V,y) := \int_{\hat \Omega} j(\id + V ,y)\det(\iden + DV)$
for some fixed $\hat \Omega \Subset D$.
When $j$ is twice differentiable, it holds that
\begin{align*}
    J_V(0,y)[V]
    =&
    \int_{\hat \Omega} \Div V j(\cdot,y) + j_x(\cdot,y) \cdot V,
\quad
    J_y(0,y)[\eta]
    =
    \int_{\hat \Omega} j_y(\cdot, y)\, \eta,
\\
    J_{yV}(0, y)[\eta,V]
    =&
    \int_{\hat \Omega} \Div V j_y(\cdot, y) \, \eta + j_{yx}(\cdot, y) \cdot V \eta,
\quad
    J_{yy}(0, y)[\eta,\xi]
    =
    \int_{\hat \Omega} j_{yy}(\cdot,y)\eta \xi,
\\
    J_{VV}(0,y)[V,W]
    =&
    \int_{\hat \Omega} \detDashDash[V,W] j(\cdot,y)
    +
    \Div V j_x(\cdot,y) \cdot W
    + \Div W j_x(\cdot,y) \cdot V
    +
    j_{xx}(\cdot,y)V \cdot W.
\end{align*}



\subsection{Derivatives for PDE constraints}
Here, we collect the derivatives for the maps $e$ which appear in Sections \ref{sec:application:poisson}, \ref{sec:application:coupledPoisson}, and \ref{sec:application:evalue}.
Recall that we define $A(V):= (\iden + DV)^{-1} (\iden +DV)^{-T}$ and $\A[V]:= \iden \Div V - DV - DV^T$.
We furthermore define
\begin{align*}
    \ADashDash[V,W]:=
    &   \detDashDash[V,W]\iden -\Div(V) DW - \Div(V) DW^T
    \\
    &-  \Div(W)DV + \left( DWDV + DVDW \right) + DV DW^T
    \\
    &-  \Div(W) DV^T + DW DV^T + (DWDV+DVDW)^T.
\end{align*}

\subsubsection{Derivatives for the Poisson Problem}\label{app:Poisson:Derivatives}
In the case that
\begin{equation*}
   \langle e(V,y), p \rangle = \int_{\hat \Omega} \left( A(V) \nabla y \cdot \nabla p -  F \circ (\id +V) p \right) \,  \det(\iden+ DV)
\end{equation*}
as in Section \ref{sec:application:poisson}, then it holds that
\begin{align*}
    \langle e_V(0,y)[V],p \rangle
    =&
    \int_{\hat \Omega} \A[V]\nabla y \cdot \nabla p - \Div (V F) p,
\quad
    \langle e_y(0,y)[\eta],p \rangle
    =
    \int_{\hat \Omega} \nabla \eta\cdot \nabla p,
\\
    \langle e_{yV}(0,y)[\eta,V],p \rangle
    =&
    \int_{\hat \Omega} \A[V]\nabla \eta \cdot \nabla p,
\quad
    \langle e_{yy}(0,y)[\eta,\xi],p \rangle
    =
    0,
\\
    \langle e_{VV}(0,y)[V,W],p \rangle
    =&
    \int_{\hat \Omega} \ADashDash[V,W] \nabla y \cdot \nabla p
    - \detDashDash[V,W] F p
    - \Div(V) p W \cdot \nabla F
    \\  \nonumber
    &- \Div(W) p V \cdot \nabla F
    - W\otimes V : pD^2 F .
\end{align*}

\subsubsection{Derivatives for the coupled Poisson Problem}
\label{app:coupledPoisson:Derivatives}
In the case that
\begin{equation*}\begin{split}
    \langle e(V,y ),p \rangle
    =&
    \int_{\hat \Omega} \left(A(V) \nabla y_1 \cdot \nabla p_2 - y_2 p_2\right) \det(\iden +DV)
    \\
    &+
    \left( A(V) \nabla y_2 \cdot \nabla p_1  - p_1 F\circ (\id + V) \right) \det(\iden +DV),
\end{split}\end{equation*}
as in Section \ref{sec:application:coupledPoisson}, it holds that
\begin{align*}
    \langle e_V(0,y)[V],p \rangle 
    =&
    \int_{\hat \Omega} \A[V] \nabla y_1 \cdot \nabla p_2 - \Div V y_2 p_2
    +
    \A[V] \nabla y_2 \cdot \nabla p_1 - \Div (V F) p_1,
    \\
    \langle e_y (0,y)[\eta],p \rangle 
    =&
    \int_{\hat \Omega} \nabla \eta_1 \cdot \nabla p_2 - \eta_2 p_2
    +
    \nabla \eta_2 \cdot \nabla p_1,
    \\
    \langle e_{yV}(0,y) [\eta,V],p \rangle
    =&
    \int_{\hat \Omega} \A[V] \nabla \eta_1 \cdot \nabla p_2 - \Div V \eta_2 p_1
    +
    \A[V] \nabla \eta_2 \cdot \nabla p_2,
    \\
    \langle e_{yy}(0,y) [\eta,\xi],p \rangle
    =&
    0,
    \\
    \langle e_{VV}(0,y) [V,W],p \rangle
    =&
    \int_{\hat \Omega} \ADashDash[V,W] \nabla y_1 \cdot \nabla p_2
                -\detDashDash[V,W] y_2 p_2
                \\  \nonumber
                &+\ADashDash[V,W] \nabla y_2 \cdot \nabla p_1
                - \detDashDash[V,W] F p_1
                \\  \nonumber
                &- \Div(V) W \cdot \nabla F p_1 - \Div(W) V \cdot \nabla F p_1 - W\otimes V : D^2 F p_1.
\end{align*}

\subsubsection{Derivatives for the Eigenvalue Problem}\label{app:eValue:Derivatives}
In the case that
\begin{equation*}\begin{split}
    \langle e(V,(z,\lambda) ),(q,\mu) \rangle
    =&
    \int_{\hat \Omega} \left( A(V)\nabla z \cdot \nabla q - \lambda z q \right) \det(\iden +DV)
    \\
    &+ \mu\left(1-\int_{\hat \Omega} \det (\iden + DV)z^2 \right),
\end{split}\end{equation*}
as in section \ref{sec:application:evalue}, it holds that
\begin{align*}
    \langle e_V(0,(z,\lambda))[V],(q,\mu) \rangle 
    =&
    \int_{\hat \Omega} \A[V] \nabla z \cdot \nabla q - \lambda \Div V z q - \mu \Div V z^2,
    \\
    \langle e_y (0,(z,\lambda))[(\eta,\tilde \eta)],(q,\mu) \rangle 
    =&
    \int_{\hat \Omega} \nabla \eta \cdot \nabla q - \lambda \eta q - \tilde \eta z q
    -
    \mu z \eta,
\end{align*}
\begin{align*}
    \langle e_{Vy}(0,(z,\lambda)) [V,(\eta,\tilde{\eta})],(q,\mu) \rangle
    =&
    \int_{\hat \Omega} \A[V]\nabla q \cdot \nabla \eta
    - \Div V \lambda \eta q - \Div V \tilde{\eta} z q
    \\  \nonumber
    &- 2\mu \int_{\hat \Omega} \Div V z \eta ,
    \\
    \langle e_{yy}(0,(z,\lambda)) [(\eta,\tilde{\eta}),(\zeta,\tilde{\zeta)}],(q,\mu) \rangle
    =&
    \int_{\hat \Omega} -\tilde{\zeta} \eta q - \tilde{\eta} \zeta q - \mu \zeta \eta,
\end{align*}
\begin{align*}
    \langle e_{VV}(0,(z,\lambda)) [V,W],(q,\mu) \rangle
    =&
    \int_{\hat \Omega} \ADashDash[V,W] \nabla z \cdot \nabla q
                - \detDashDash[V,W]\left( \lambda z q - \mu z^2\right).
\end{align*}

For the energy, $J(V,(z,\lambda)) = \lambda$, it {is clear that only the derivative in the second component is non-vanishing and one has that}
\begin{align*}
    J_y(0,(z,\lambda)[(\eta,\tilde{\eta})] =& \tilde{\eta}.
\end{align*}

\printbibliography

\end{document}